\documentclass[10pt]{article}
\usepackage{amsmath,stackengine}
\usepackage{amssymb}
\usepackage{titlesec}
\usepackage[T1]{fontenc}
\usepackage{empheq}
\usepackage{graphicx}
\usepackage{titling}
\usepackage[english]{babel}
\usepackage[margin=2.5cm]{geometry}
\usepackage{csquotes}
\usepackage{lipsum}
\usepackage[english]{babel}
\usepackage{amsthm}
\usepackage{amssymb}
\usepackage{bm} 
\usepackage{hyperref}

\usepackage{mathrsfs}
\usepackage{color}
\usepackage{mathtools,amsfonts,amssymb,setspace}
\usepackage{enumitem}
\usepackage{subfigure}
\usepackage[font=small,labelfont=bf]{caption}
\newtheorem{theorem}{Theorem}[section]

\theoremstyle{definition}
\newtheorem{definition}{Definition}[section]
\newtheorem{prop}[theorem]{Proposition}
\theoremstyle{remark}
\newtheorem*{remark}{Remark}

\definecolor{green}{rgb}{0.1,0.62,0.0}
\definecolor{orange}{rgb}{0.9,0.3,0.0}
\definecolor{deepblue}{rgb}{0.0,0.0,0.7}
\definecolor{blue}{rgb}{0.0, 0.0, 1.0}
\definecolor{red}{rgb}{1.0, 0.0, 0.0}

\font\myfont=cmbx12 at 11pt
\font\myfond=cmr12 at 10pt
\title{\myfont REGULARITY AND OPTIMAL CONTROL OF NONLOCAL CAHN-HILLIARD-BRINKMAN SYSTEM WITH  SINGULAR POTENTIAL}
\author{\myfond SHEETAL DHARMATTI \thanks{sheetal@iisertvm.ac.in \; School of Mathematics, Indian Institute of Science Education and Research, Thiruvananthapuram, Maruthamala PO, Vithura, Thiruvananthapuram, Kerala, 695551, INDIA}  \;  and  \;  GREESHMA K \thanks{ greeshmak21@iisertvm.ac.in \; School of Mathematics, Indian Institute of Science Education and Research, Thiruvananthapuram, Maruthamala PO, Vithura, Thiruvananthapuram, Kerala, 695551, INDIA,\\ \textbf{Acknowledgement }: Greeshma K would like to thank the Department of Science and Technology (DST), India, for the Innovation in Science Pursuit for Inspired Research (INSPIRE) Fellowship (IF210199). } }
\date{}

\numberwithin{equation}{section}
\thanksmarkseries{arabic}

\begin{document}
\maketitle

The evolution of two incompressible, immiscible, isothermal fluids in a bounded domain and a porous media is described by the coupled Cahn-Hilliard-Brinkman (CHB) system. The CHB  system consists of the Cahn-Hilliard equation describing the dynamics of the relative concentration of fluids and the Brinkman equation for velocity. This work addresses the optimal control problem for a two-dimensional nonlocal CHB system with a singular-type potential. The existence and regularity results are obtained by approximating the singular potential by a sequence of regular potentials and introducing a sequence of mobility terms to resolve the blow-up due to the singularity of the potential. Further, we prove the existence of a strong solution under higher regularity assumptions on the initial data and the uniqueness of the solution using the weak-strong uniqueness technique. By considering the external forcing term in the velocity equation as a control, we prove the existence of an optimal control for a tracking type cost functional. The differentiability properties of the control-to-state operator are studied to establish the first-order necessary optimality conditions. Moreover, the optimal control is characterised in terms of the adjoint variable.
\setcounter{section}{1}

\section*{Introduction}
The evolution of two immiscible, incompressible fluids in a bounded domain and a porous media is modelled by the Cahn-Hilliard-Brinkman (CHB) system. The CHB system consists of two coupled equations one for the average velocity of the fluids, and the other describes the dynamics of relative concentration.
Consider $\Omega \subseteq \mathbb{R}^d, d\in \{2,3\}$, an open, bounded, connected, smooth domain with boundary $\partial\Omega$. We denote the sets $\Omega \times (0,T)$ and $\partial\Omega \times (0,T)$ by $Q$ and $\Sigma$ respectively. Let $\varphi$ be the relative concentration of two fluids ($\varphi = \varphi_1 -\varphi_2,)$ where $\varphi_1, \varphi_2$ are concentrations of individual fluids and $\textbf{u}$ be the average velocity of fluids. The Nonlocal Cahn-Hilliard-Brinkman system is given by,

\begin{align}
\varphi'+\nabla\cdot(\ \textbf{u}\varphi )\ &=\ \nabla\cdot( m(\varphi)\nabla\mu),     \; \text{in}\;\; Q,\label{eq1}
\\\mu &=\ a\varphi- J \ast\varphi + F '\left(\varphi\right),    \; \text{in}\;\; Q,\label{eq2}
\\-\nabla\cdot\left(\nu\left(\varphi\right)\nabla \textbf{u} \right) + \eta \textbf{u} + \nabla \pi &=\ \mu\nabla\varphi + \textbf{h},   \; \text{in}\;\; Q,\label{eq3}
\\ \nabla\cdot \textbf{u}  & =  0, \, \text{in}  \;\; Q.\label{eq4}
\end{align}

We endow this system with the following initial and boundary conditions.

\begin{align}
 \frac{\partial \mu}{\partial n} = 0 ,\;\; \text{on} \; \Sigma, \label{eq05}
\\\textbf{u}=0,\; \; \text{on} \;   \Sigma,\label{eq06}
\\ \varphi\left(. , 0 \right) =\varphi_{0},\; \; \text{in} \; \; \Omega.\label{eq07}
\end{align}

Let $\nu > 0$ denote the viscosity coefficient, which may depend on $\varphi$. Permeability is denoted by $\eta > 0$, and $\pi$ denotes the pressure exerted on the fluids.  The mobility $m$  depends on $\varphi$ and $J: \mathbb R^{d}\longrightarrow \mathbb R$ is a suitable interaction kernel with $a$ defined as $a(x):=\int_{\Omega}{J(x-y)dy}.$ The above system is nonlocal because of the convolution term $J*\varphi$. Local CHB system is obtained by replacing $\mu$ in equation \eqref{eq2} by $\mu = -\Delta\varphi + F'\left(\varphi\right)$. 
 Here $F$ denotes the potential, and it can be either regular or singular.  The external forcing term is denoted by $\textbf{h}$. 

An example of a  regular potential is a well known double well potential (e.g. $\frac{1}{4}(s^{2}-1)^{2}$). 
In most of the practical problems, the potential that appears is singular. The two important examples of the singular potentials that are considered in this work are:  Double Obstacle potential  defined by,
\begin{equation}\label{eq08}
F_{\mathrm{do}}(r)=\frac{1}{2} (1-r^2)+ \mathbb{I}_{[-1,1]}(r)= 
\begin{cases}
\frac{1}{2}(1-r^2) \text{ for }  r\in [-1,1],\\ 
\infty   \text{  otherwise.  }
\end{cases}
\end{equation}
and the Logarithmic potential defined by, 
\begin{equation}\label{eq09}
F_{\log}(r)=\frac{\theta}{2}((1+r) \log(1+r)+(1-r) \log(1-r))+\frac{\theta_c}{2}\left(1-r^2\right)\quad \text{ for } r \in(-1,1).
\end{equation}
where $0<\theta <\theta_c$. For logarithmic potential defined on $(-1,1)$, all its derivatives exist with a blow-up as $r$ approaches $\pm 1$. The singularity of the potential at pure phases $\varphi =\pm 1$ is resolved by approximating the potential using smooth functions. Adopting the technique used in \cite{CEG}, we introduce a compatible mobility function to resolve the blow-up of derivatives of the potential in neighbourhoods of $\pm 1$. The following choice of mobility function will work for the potential $F_{log}$,
\begin{equation}\label{eq010}
m(r)= 
\begin{cases}
(1-r^2), \text{ for } r\in [-1,1]\\
0, \text{ for } |r|>1.
\end{cases}
\end{equation}
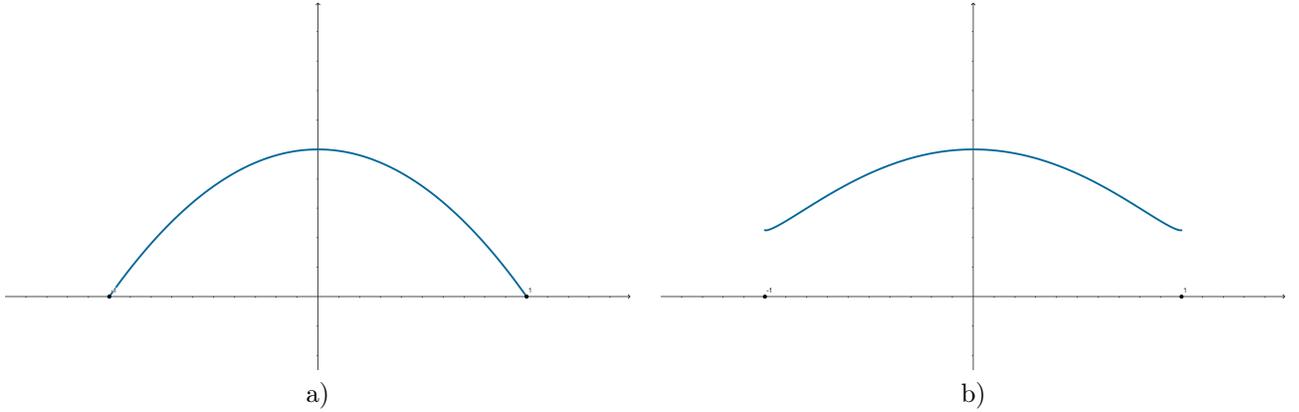
The double obstacle potential $F_{do}$ has a singularity at pure phases $\pm 1$ due to the jump discontinuity at these points. On $(-1,1)$, its derivatives of all orders exist and are bounded. Hence, we can choose a non-degenerate mobility function. In particular, we will consider a cutoff of the above mobility function. i.e, for a fixed $\epsilon>0,$
\begin{align}\label{eq1.11}
 m_{do}(s) := \begin{cases} m(s) &\text{ for } |s|\leq \epsilon, \\ m(\epsilon)  &\text{ for } |s|>\epsilon.\end{cases}
\end{align} 
In the literature, the CH equation has been studied  for well-posedness for a singular-type potential in \cite{GGG}, \cite{CEG} and \cite{DDU}. The authors have used approximations to the singular potential with potentials having polynomial  growth of suitable order. There are several works on the well-posedness of the coupled Cahn-Hilliard Navier Stokes' (CHNS) equations under varied assumptions on the potential and the mobility, e.g. \cite{CFG}, \cite{FGS}, \cite{FGR}, \cite{SEF}, \cite{FGG} and \cite{SFG}. The existence of a strong solution to the nonlocal CHNS system with singular potential and degenerate mobility was shown in \cite{SFG} using a time discretization scheme. In \cite{FGS}, the authors studied the optimal control problem of the nonlocal CHNS system with a singular potential in dimension 2. They also showed that under the zero viscosity condition, the CHB system reduces to the Cahn-Hilliard-Hele-Shaw (CHHS) system.
The existence of weak and strong solutions to the local and nonlocal CHHS system with a singular type potential is proved using a suitable approximation scheme in \cite{GAG} and \cite{DAG}, respectively.

Regarding the CHB system, the existence of a weak solution to the local CHB system with a regular double well potential, constant mobility and viscosity was analysed in \cite{BCG}. In \cite{DPG}, the above results were extended to the nonlocal CHB system with regular potential. For the nonlocal CHB system with constant mobility and regular potential, the strong solution results and the optimal control problem were studied in \cite{SDM} for dimension two. The three dimensional local-CHB system with the logarithmic potential, the well-posedness of weak solutions and the global in time existence of strong solutions were established in \cite{MCG}. 
 In \cite{CSM}, the authors approximated the nonlocal CHHS system with a singular potential and a non constant mobility by CHB system with constant viscosity, a regular potential and a non-degenerate mobility. They proved the existence of a weak solution to the CHHS system as the limit of solution of the CHB system. 
 In the same work, the existence of a strong solution is proved via fixed point arguments using Schauder's theorem. 

In this work our primary aim is to study the Optimal Control Problem (OCP) of nonlocal CHB system \eqref{eq1}-\eqref{eq07} 
with  variable mobility and singular potential (in particular either $F_{do}$ or $F_{log}$). The problem is to minimise a tracking type cost functional defined by,
\begin{align}
\mathcal{J}(\varphi,\textbf{u},\textbf{U}) \coloneqq \int\limits_{0}^{T} \lVert\varphi(t)-\varphi_{d}(t)\rVert^{2} dt + \int\limits_{0}^{T}\lVert\textbf{u}(t)-\textbf{u}_{d}(t)\rVert^{2} dt  +\int_{\Omega}\lVert\varphi(T)-\varphi_{\Omega}\rVert^{2} dx + \int\limits_{0}^{T}\lVert \textbf{U}(t)\rVert^{2} dt
\end{align}
in a bounded, closed and convex set of admissible controls defined by,
$$\mathcal{U}_{ad}=\{\textbf{U}\in \mathcal{U}: U_1(x,t)\leq \textbf{U}(x,t)\leq U_2(x,t), \text{ a.e } (x,t)\in \Omega\times (0,T) \}$$
subject to the nonlocal CHB system \eqref{eq1}-\eqref{eq07} 
Here $(\varphi,\textbf{u})$ represents a solution to the system \eqref{eq1}-\eqref{eq07} corresponding to the external force, $\textbf{U}$ and $(\varphi_d, \textbf{u}_d, \varphi_{\Omega})$ is the desirable state. The external forcing term $\textbf{U}$ acts as a control and we consider controls from the set, $\mathcal{U}=\{\textbf{U}\in L^\infty(0,T;\mathbb{G}_{div})|\textbf{U}_t\in L^2(0,T;\mathbb{V}_{div}')\}$ and $U_1, U_2 \in L^{\infty} (\Omega)$ are fixed functions.

 To study the OCP corresponding to the system \eqref{eq1}-\eqref{eq07}, we need to study the properties of the control-to-state operator. In particular we need to show that the control-to-state operator is  well-defined by  establishing the existence of a weak as well as strong solution of the associated system.  These results can be derived as in \cite{CSM}. As mentioned above, \cite{CSM}  proves strong solution results using fixed point arguments. We use a-priori estimates to prove the strong solution and also derive
  stability estimates in more regular spaces using these a-priori estimates which are needed to study the control problem. In turn we study the differentiability properties of the control-to-state operator. The main difficulty in the estimations is posed by  the nonlocal term and the singularity of the potential. To overcome these difficulties we introduce the following operators and rewrite the system using them.
 \begin{align}
 \lambda(s)&:= m(s)F''(s) \label{eq1.13a}\\
   B(s)&:=\int\limits_{0}^{s}\lambda(\bm{\tau})d\bm{\tau},\hspace{.5cm} \nabla B(\varphi)= \lambda(\varphi)\nabla\varphi \hspace{.25cm}; \hspace{1cm} b(s):=\int\limits_{0}^{s}m(\bm{\tau})d\bm{\tau},\hspace{.5cm} \nabla b(\varphi)= m(\varphi)\nabla\varphi \label{eq1.14}\\
   \widetilde{B}(s)&:= B(s)+ a(\cdot)b(s) \label{eq1.15}
\end{align}
According to the above notations \eqref{eq1} can be rewritten as,
\begin{align}
    \varphi' + \textbf{u}\nabla\varphi &= \Delta \widetilde{B}(\varphi) + \nabla\cdot\big(m(\varphi)\big((\nabla a) \varphi -\nabla J*\varphi\big)\big)-\nabla\cdot\big(b(\varphi)\nabla a\big).\label{eq1.13}
\end{align}
This representation enables us to derive the stability estimates and, further, the differentiability of the control to state operator.  

The paper is organised as follows. In the next section, we introduce the mathematical setup, the notations used, assumptions set and preliminary results regarding the well-posedness of the system. In section 3, we prove the existence of a strong solution by considering an approximate system obtained by replacing the potential with regular potential and then passing to the limit. Further the uniqueness of strong solution is proved by a weak-strong uniqueness argument and a stability estimate of higher regularity is derived. In section 4, we define the optimal control problem and prove that an optimal control exists. The differentiability of the control-to-state operator is also studied in this section. In section 5, we introduce a suitable adjoint system and derive a first order optimality condition. Lastly, we give a characterisation of the optimal control in terms of the adjoint variable.

\section{Mathematical setting and Well-Posedness of the System}
Let $\Omega \subseteq \mathbb{R}^d$, $d=2,3$, be an open bounded connected domain with a smooth boundary $\partial\Omega.$ Let H be the set of square-integrable functions on $\Omega$, $L^2(\Omega).$ V be the space $H^1(\Omega)$. Denote the norm in $H$ and $V$ by $\lVert \,\cdot\,\rVert$ and $\lVert \,\cdot\,\rVert_{V}$ respectively. The duality pairing between spaces $V$ and $V'$ is denoted by $\langle.,.\rangle$. We denote the closure of divergence-free space, $\mathcal{V } :=\{ u\in (C_0^\infty(\Omega)^d; \nabla. u = 0\}$ in $L^2(\Omega;\mathbb{R}^n)$ and $H^1(\Omega;\mathbb{R}^n)$ by $\mathbb{G}_{div}$ and $\mathbb{V}_{div}$ respectively. Then we have from \cite{TEM}, 
\begin{align*}
\mathbb{G}_{div} &= \{ \textbf{u}\in L^2(\Omega;\mathbb{R}^n): \text{div}(\textbf{u})=0, \textbf{u}.\textbf{n}|_{\partial \Omega}=0\},\\
 \mathbb{V}_{div} &= \{ \textbf{u}\in H^1_0(\Omega;\mathbb{R}^n): \text{div}(\textbf{u})=0\}.
\end{align*}
Let $\lVert \cdot \rVert$, $\lVert \cdot \rVert_{\mathbb{V}_{div}}$ be norm in $\mathbb{G}_{div}$ and $\mathbb{V}_{div}$ respectively. Let $\mathbb{V}_{div}'$ be the dual space of $\mathbb{V}_{div}$ and $\langle.,.\rangle$ denotes the duality pairing between $\mathbb{V}_{div}$ and $\mathbb{V}_{div}'.$ 
\\\\Well-posedness of the CHB system, with a regular potential has been studied in \cite{SDM} with the following set of assumptions on the potential.
\begin{enumerate}[font={\bfseries},label={A\arabic*.}]
\item[{[H1]}] $F \in C^{2,1}_{loc}(\mathbb{R})$ and there exists $c_0 > 0$ such that
    $$F''(s) + a(x) \geq c_0, \,\forall s \in \mathbb{R} \text{ and a.e } x \in \Omega.$$
    \item[{[H2]}] There exist $c_1>0, c_2>0$ and $q>0$ if $d=2,$ $q\geq \frac{1}{2}$ if $d=3$ such that 
    $$F''(s) + a(x) \geq c_1|s|^{2q}-c_2, \forall s \in \mathbb{R} \text{ and a.e } x \in \Omega.$$
    \item[{[H3]}] There exists $c_3>0$ and $p \in (1,2]$ such that 
    $$|F'(s)|^p \leq c_3(|F(s)+1|), \, \forall s\in \mathbb{R}.$$
\end{enumerate}
To study the case of singular potential   we set the following assumptions. 
\begin{enumerate}[font={\bfseries},label={A\arabic*.}]
    \item[{[N]}] The viscosity $\nu$ is Lipschitz continuous on $\mathbb{R}$ and there exist some $\nu_0, \nu_1 > 0$ such that  $\nu_0 \leq \nu(s) \leq \nu_1, \forall \ s\in \mathbb{R},$\; 
    and $\eta \in L^{\infty}(\Omega)$ is such that $\eta(x) \geq 0,$ a.e $x \in \Omega.$
    \item[{[J]}] $J \in W^{1,1}(\Omega)$ is an even function and $a(x) := \int_{\Omega} J(x-y)dy \geq 0 $ for almost all $ x\in \Omega .$
    And \\$\sup_{x\in \Omega}\int_{\Omega}|J(x-y)|dy  < \infty,$ $b:= \sup_{x\in \Omega}\int_{\Omega}|\nabla J(x-y)|dy < \infty.$
\end{enumerate}
To treat the case of singular potential we need to suitably modify the assumptions on the potential as follows:
\begin{enumerate}[font={\bfseries},label={A\arabic*.}]
    \item[{[A1]}] The mobility $m\in C([-1,1])$ such that $m(s)\geq 0, \forall s \in [-1,1]$ and $m(s) =0$ iff $s=\pm 1.$ Further, there exists  $\epsilon_0 > 0$ such that $m$ is non-increasing in $[1-\epsilon_0,1]$ and non-decreasing in $[-1,-1+\epsilon_0].$
    \item[{[A2]}] The potential, $F = F_1 + F_2$ where $F_1\in C^2(-1,1)$ and $F_2\in C^2[-1,1].$ And 
    $\lambda_1:= mF_1''\in C([-1,1])$ is such that there exists some $\alpha_0 >0$ such that $\lambda_1(s)\geq \alpha_0, \forall s\in[-1,1].$
    \item[{[A3]}] There exists  $\epsilon_1 > 0$ such that $F''$ is non-decreasing in $[1-\epsilon_1,1)$ and non-increasing in $(-1,-1+\epsilon_1].$
    \item[{[A4]}] There exists  $c_0 >0$ such that $m(s)(F''(s) + a(x)) \geq \alpha_1, \forall s\in(-1,1)$ and a.e $x \in \Omega.$
\end{enumerate}

To study the optimal control problem [OCP], we need to prove the strong solution results for the underlying system \eqref{eq1}- \eqref{eq07}. We start by defining a weak solution and available results towards it.

\begin{definition}[Weak solution]\label{def1}
Let $T\geq 0$ be given, $\varphi_0 \in H$ be such that $F(\varphi_0) \in L^1(\Omega)$ and $\textbf{h}\in L^2(0,T;\mathbb{V}_{div}')$. A pair  $(\varphi, \textbf{u})$ is a weak solution to the system \eqref{eq1}-\eqref{eq4} with initial and boundary conditions \eqref{eq05}-\eqref{eq07} on $[0,T]$ if
\begin{align*}
\varphi  &\in L^2(0,T;V) \cap C([0,T];H),\\
\varphi' &\in L^2(0,T;V'),\\
m(\varphi)\nabla\mu &= m(\varphi)(\nabla(a\varphi)- \nabla J*\varphi)+ \lambda(\varphi)\nabla\varphi\in L^2(0,T;H),\\
\textbf{u} &\in L^2(0,T;\mathbb{V}_{div}).
\end{align*}
And it satisfies the following weak formulation.
\begin{align}
<\varphi',\psi> +(\textbf{u}\cdot\nabla\varphi,\psi)\ & + (m(\varphi)\nabla \mu,\nabla\psi) = 0,     \; \forall \psi\in V, \;\; \text{a.e \; in } \left(0,T\right), \label{eq7}\\
(\nu(\varphi)\nabla \textbf{u},\nabla \textbf{v})  + (\eta \textbf{u},\textbf{v}) &= (\mu\nabla\varphi,\textbf{v}) + <\textbf{h},\textbf{v}>,   \; \forall \textbf{v}\in \mathbb{V}_{div},\;\; \text{a.e in} \left(0,T\right),\label{eq8} \\
 \varphi(x,0) &=  \varphi_0, \, \text{a.e\; in }  \Omega .
\end{align}
\end{definition}

 \begin{prop}[\textbf{Existence of a weak solution to the nonlocal CHB system with a singular potential}]
Let $\varphi_0\in H$ be such that $F(\varphi_0) \in L^1(\Omega)$, $\textbf{h}\in L^2(0,T;\mathbb{V}_{div}')$ and assumptions (\textbf{N}), (\textbf{J}), (\textbf{A1})-(\textbf{A4}) hold. Further, assume there exists a function $M \in C^2(-1,1)$ such that $m(s)M''(s) =1, \forall s\in [-1,1]$ with $ M(0)= M'(0)= 0$ and satisfies the condition $M(\varphi_0) \in L^1(\Omega)$.

Then there exists a weak solution $(\varphi, \textbf{u})$ in the sense of Definition \ref{def1} to the nonlocal CHB system \eqref{eq1}-\eqref{eq07} and with potential given by either \eqref{eq08} or \eqref{eq09}.
Moreover if we set,
\begin{equation*}\label{eqE}
   \mathcal{E}(t)= \int_{\Omega}F(\varphi)dx- \frac{1}{2}\int_{\Omega}\int_{\Omega}J(x-y)\varphi(x)\varphi(y))dxdy, 
\end{equation*}
Then the following energy estimate holds for almost all $t\in [0, T].$ 
 \begin{align}\label{eq1.4}
   \mathcal{E}(t)+ \frac{c_0}{2}\lVert\varphi\rVert_{L^\infty(0,T;H)}^2 &+\int\limits_{0}^{t}\Big( \lVert\sqrt{m(\varphi)}\nabla\mu\rVert^2 +  \lVert\sqrt{\nu(\varphi)}\nabla u\rVert^2+ \frac{\eta}{2} \lVert u\rVert^2\Big) \nonumber\\&\leq C\big(\mathcal{E}_0+\lVert\varphi_0\rVert^2+ \lVert \textbf{h}\rVert_{L^2(0,T;\mathbb{V}_{div}')}^2\big).
\end{align}

In addition, assume $\varphi_0\in L^\infty(\Omega)$ and $\textbf{h}\in L^\infty(0,T;\mathbb{V}_{div}')$, then a solution $(\varphi, \textbf{u})$  has the following regularity.
    $$\varphi \in L^\infty(\Omega\times [0,T]),\,\, |\varphi(x)|\leq 1 \text{ a.e in } \Omega \times [0,T] \text{ and } \,\,\textbf{u}\in L^{\infty}(0,T;\mathbb{V}_{div} ).$$
\end{prop}
\begin{proof} 
The existence of a weak solution can be proved by adopting the techniques used in [\cite{CSM} Theorem 3.1].  
In \cite{CSM}, for each $\epsilon > 0$,  authors consider  $(\nu_\epsilon, F_\epsilon, m_\epsilon)$, approximations to the viscosity $\nu = 0$, the singular potential $F$ and a compatible mobility function $m$ respectively in the CHHS system. Finally passing to the limit  in the solution of the approximated CHB system  corresponding to approximations $(\nu_\epsilon, F_\epsilon, m_\epsilon)$, they obtained a solution to the nonlocal CHHS system with potential $F$ and mobility $m$. Thus their interim result  would give existence of a weak solution for CHB system for fixed $\nu_\epsilon$.
 In our case we  consider a non-constant viscosity function that satisfies $\textbf{[N]}$ that is $\nu$ being bounded, one can still conclude the result for the weak solution of \eqref{eq1}-\eqref{eq07}  by similar arguments as in [\cite{CSM} Theorem 3.1.]
 
 Further, the energy estimate \eqref{eq1.4} and an additional regularity of $\varphi$ can be proved by adopting the techniques used in \cite{CSM} and \cite{SFG}.
\end{proof}

To study the optimal control problem for the system, the weak solution regularity is not sufficient. In the next section we  discuss the existence of a strong solution to the system \eqref{eq1}-\eqref{eq07} by deriving a-priori estimates. This also helps us to prove stability of the solutions in higher regular spaces  which in turn will  be useful later while studying OCP.

\section{Existence of a Strong Solution to the Nonlocal CHB system with a Singular Potential and Variable Mobility}

In this section, we prove the existence of a strong solution to the system \eqref{eq1}-\eqref{eq07} with a singular potential and non-constant mobility.  For the proof, we will use the strong solution results available in the literature for the regular potential case \cite{SDM}, for constant viscosity and for $d =2$. Hence now onwards we will restrict ourselves to a system in dimension 2 with constant viscosity. 
For the nonlocal CHB system with regular potential even the weak solution result is open in dimension 3.  In the following subsection, we introduce regular approximations to the potential $F_\delta$, such that it satisfies assumptions \textbf{[H1]} - \textbf{[H3]}. We consider the system with  approximated regular potentials and obtain the existence of a strong solution for this approximated system using results available in the literature. Finally, we prove  the solution to the system with singular potential   is the limit of solutions to the approximated system, which is the most crucial step of the proof.  

\subsection{Approximations to Double Obstacle and Logarithmic  Potential:}
To prove the existence of a weak solution, the $C^2$ approximations to the singular potential that were devised in \cite{MEL} are used. One can also use the $C^{2,1}_{loc}$ approximations from \cite{CSM}.
Since we are going to use the existence result from \cite{SDM} to prove the existence of a strong solution, $C^2$ approximations are not sufficient. Here we will introduce $C^3$ functions which approximate singular potentials $F_{do}$ and $F_{log}$. 
\\Consider the Double obstacle potential, 
\begin{align}
    F_{do}(r)&= F_1(r) +F_2(r),\\
  \text{ where, }  F_1(r) &= \begin{cases}
        \frac{r^2}{2}, \text{ for } |r|\leq 1,\\
        \infty, \text{ for } |r|> 1,\\
    \end{cases}\nonumber\\
    F_2(r) &= \frac{(1-2r^2)}{2}.\nonumber
\end{align}
We will approximate $F_1$ using $C^3$ functions, say $F_{1,\delta}, \delta>0$ so that $F_{do,\delta} = F_{1,\delta}+F_2$ satisfies assumptions [\textbf{H1}]-[\textbf{H3}]. We will approximate $F_{do}$ by the following functions,
\begin{equation}\label{eq0082}
    F_{do,\delta}(r) := \Big(\beta_{do,\delta}(r)+\frac{r^2}{2}\Big) +\frac{1}{2}(1-2r^2),
\end{equation}
Where $\beta_{do,\delta}(r)$ is defined as follows.
\begin{equation*}
\beta_{\mathrm{do}, \delta}(r)= \begin{cases}\frac{4}{\delta^2}\left(r-\left(1+\frac{\delta}{2}\right)\right)^3+\left(r-\left(1+\frac{\delta}{2}\right)\right) & \text { for } r \geq 1+\delta, \\ \frac{1}{\delta^3}(r-1)^4 & \text { for } r \in(1,1+\delta), \\ 0 & \text { for }|r| \leq 1, \\ \frac{1}{\delta^3}(r+1)^4 & \text { for } r \in(-1-\delta,-1), \\ -\frac{4}{\delta^2}\left(r+\left(1+\frac{\delta}{2}\right)\right)^3-\left(r+\left(1+\frac{\delta}{2}\right)\right) & \text { for } r \leq-1-\delta.\end{cases}
\end{equation*}

$$\beta_{\mathrm{do}, \delta}(r) \longrightarrow \mathbb{I}_{[-1,1]}(r):=\begin{cases} 0 & \text { for } r\in [-1,1] \\\infty  & \text { otherwise. } 
\end{cases} $$ 
pointwise as $\delta \rightarrow 0$ on $\mathbb{R}$ and uniformly on $[-1,1]$. Therefore,
$$ F_{do,\delta} \rightarrow F_{do} \text{ uniformly in } [-1,1] \text{ as } \delta\rightarrow 0.$$

Similarly, for the logarithmic potential,
\begin{align}
    F_{log}(r)&:= F_1(r) +F_2(r),\\
   F_1(r) &= \frac{\theta}{2}\big((1+r)log(1+r)+(1-r)log(1-r)\big),\nonumber\\
    F_2(r) &= \theta_c\frac{(1-r^2)}{2}.\nonumber
\end{align}

As in the case of $F_{do}$, we will approximate the singular part, $F_1$ by $F_{1,\delta} \in C^3.$
For each $\delta>0, F_\delta= F_{1,\delta}+ F_2$ satisfies assumptions [\textbf{H1}]-[\textbf{H3}]. A particular choice of regular approximations to the logarithmic potential in $C^3$ is given by,
for $\delta >0 $,
\begin{equation}\label{eq0084}
    F_{\log , \delta}(r) = F_{1 , \delta}(r) + F_2(r),
\end{equation}
\begin{equation*}
F_{1 , \delta}(r)= \begin{cases}F_{1 }(1-\delta)+F_{1 }^{\prime}(1-\delta)(r-(1-\delta))  +\frac{1}{2} F_{1 }^{\prime \prime}(1-\delta)(r-(1-\delta))^2 &\\ + \frac{1}{6} F_{1 }^{\prime \prime\prime}(1-\delta)(r-(1-\delta))^3 & \text { for } r \geq 1-\delta, \\ F_{1 }(r) & \text { for }|r| \leq 1-\delta, \\ F_{1 }(\delta-1)+ F_{1}^{\prime}(\delta-1)(r-(\delta-1))  +\frac{1}{2} F_{1 }^{\prime \prime}(\delta-1)(r-(\delta-1))^2 & \\ +\frac{1}{6} F_{1 }^{\prime \prime \prime}(\delta-1)(r-(\delta-1))^3 & \text { for } r \leq-1+\delta. \end{cases}
\end{equation*}
\\
Since the interval $(-1+\delta, 1-\delta)$ increases to (-1,1) as $\delta \rightarrow 0$, $F_{\log , \delta}(r) \rightarrow F_{log}(r)$ uniformly on (-1,1). 

Further, we will choose $m\in C^1[-1,1]$ so that \textbf{[A1]},  \textbf{[A2]} hold true. Since $F_{log}''$ blows up near points 1 and -1, we will choose a degenerate mobility that degenerates precisely at these points. We can find a sequence of non-degenerate mobility functions that approximate $m$. For the potential $F_{do}$, since all its derivatives exist and are bounded on $(-1,1)$, it is sufficient to  choose a mobility function that is bounded below by a positive constant. 

Using these approximations of potential function and mobility we will  find approximate solutions of the system \eqref{eq1}-\eqref{eq07}. The existence of approximate solutions will be followed by existence of strong solution results for the nonlocal CHB system with a regular potential and a constant mobility proved in \cite[Theorem 2.9]{SDM}.

Our main aim of the paper is to study the distributed optimal control problem. Hence 
We choose the source term in an appropriate space which would be useful to later on study the strong solution and control problem.  \\

Define $\mathcal{U}:=\{\textbf{h}\in L^\infty(0,T;\mathbb{G}_{div})|\textbf{h}_t\in L^2(0,T;\mathbb{V}_{div}')\}$ with norm,
$$\lVert \textbf{h}\rVert_\mathcal{U}:= \lVert \textbf{h}\rVert_{L^\infty(0,T;\mathbb{G}_{div})}+\lVert \textbf{h}_t\rVert_{L^2(0,T;\mathbb{V}_{div}')}.$$

\begin{theorem}[\textbf{Existence of a strong solution to the nonlocal CHB system with a singular potential and non-constant mobility}] 
 Let $\varphi_0 \in H^2(\Omega) \cap L^\infty(\Omega),$ $\textbf{h}\in \mathcal{U}$ and $J \in W^{2,1}(\Omega).$ Assume $[\textbf{J}], [\textbf{A1}] -[\textbf{A4}]$ are satisfied. And in addition assume that $m \in C^1[-1,1]$ and $\lambda \in C^1[-1,1]$. Then there exists a strong solution $(\varphi, \textbf{u})$ for the system \eqref{eq1}-\eqref{eq07} with potential either $F_{do}$ or $F_{log}$ on $[0,T]$, that has the following regularity.
\begin{align}
\varphi &\in L^\infty(0,T;H^2(\Omega)) ,\label{eq2.08}\\
\varphi' &\in L^2(0,T,V)\cap L^\infty(0,T;H),\label{eq2.09}\\
\textbf{u} &\in L^2(0,T;\textbf{H}^2(\Omega)).\label{eq2.010}
\end{align}
\end{theorem}

\begin{proof}We are going to prove the theorem in two steps. We first use the approximations detailed above for the potential and mobility and write an approximate problem satisfied by the corresponding $(\varphi_\delta, \textbf{u}_\delta)$. We use existing results from \cite{SDM} to show the existence of a strong solution to the approximated problem. The most important contribution from our side is the crucial second step where we have proved convergence of $(\varphi_\delta, \textbf{u}_\delta)$ to the strong solution of the nonlocal CHB system with a singular potential. 
\\\\Consider the following approximated problem, for $\delta>0$,
\begin{align}
\varphi'_{\delta}+\nabla\cdot(\textbf{u}_{\delta}\varphi_{\delta} )\ &=\ \nabla\cdot( m_\delta(\varphi_\delta)\nabla\mu_{\delta}),     \; \text{in}\;\; Q,\label{eq23a}
\\\mu_\delta &:=\ a\varphi_\delta- J \ast\varphi_\delta + F_\delta '\left(\varphi_\delta\right),    \; \text{in}\;\; Q,\label{eq24a}
\\-\nabla\cdot\left(\nu\left(\varphi_\delta\right)\nabla \textbf{u}_\delta \right) + \eta \textbf{u}_\delta + \nabla \pi_\delta &=\ \mu_\delta\nabla\varphi_\delta + \textbf{h},   \; \text{in}\;\; Q,\label{eq25a}
\\ \nabla\cdot \textbf{u}_\delta  & =  0, \, \text{in}  \;\; Q.\label{eq26a}
\\ (m_\delta(\varphi_\delta)\nabla\mu_\delta)\cdot \textbf{n} &= 0 ,\;\; \text{on} \; \Sigma, \label{eq27a}
\\\textbf{u}_\delta &= 0,\; \; \text{on} \;   \Sigma,\label{eq28a}
\\ \varphi_\delta\left(. , 0 \right) &=\varphi_{0},\; \; \text{in} \; \; \Omega.\label{eq29a}
\end{align}

Here $F_\delta,\hspace{.25cm}(\delta >0)$ represents regular approximations detailed in section 3.1 which satisfy assumptions \textbf{[H1]}-\textbf{[H3]} and $m_\delta,$ approximations to the  mobility function. 
The existence of a strong solution to the system \eqref{eq23a}-\eqref{eq29a} follows by similar arguments as in the proof of \cite[Theorem 2.9]{SDM}. 
\\In \cite[Theorem 2.9]{SDM} authors assume a constant mobility case. For our problem, we consider a non-constant mobility function which is bounded below by a positive constant.
That is, $\exists \ m_0 > 0$ such that $m_\delta(s)\geq m_0, \forall s \in [-1,1]$. This is sufficient to use  the arguments of \cite[Theorem 2.9]{SDM} to get a strong solution of the system \eqref{eq23a}-\eqref{eq29a}  with the regularity:
\begin{align}
\varphi_\delta &\in L^\infty(0,T;H^2(\Omega)),\label{eq3.15a}\\
\varphi_\delta' &\in L^\infty(0,T;H) \cap L^2(0,T,V),\label{eq3.16a}\\
\textbf{u}_\delta &\in L^2(0,T;\textbf{H}^2(\Omega)),\label{eq3.17a}
\end{align}
for each $\delta >0$.
Now we will derive uniform bounds for $(\varphi_\delta, \textbf{u}_\delta)$ in appropriate spaces so that we can extract convergent subsequences and further show that the  limit of these subsequences is a solution to the \eqref{eq1}-\eqref{eq07}.
We will introduce the following operators as in \eqref{eq1.13a}-\eqref{eq1.15}.
\begin{align*}
   \lambda_\delta &:= m_\delta(s)F_{1,\delta}''(s), \\
   B(s)&:=\int\limits_{0}^{s}\lambda_{\delta}(\tau)d\tau,\hspace{.5cm} \nabla B(\varphi_{\delta})= \lambda_{\delta}(\varphi_{\delta}) \nabla\varphi_{\delta} ,\\
   b(s)&:=\int\limits_{0}^{s}m_{\delta}(\tau)d\tau,\hspace{.5cm} \nabla b(\varphi_{\delta})= m_{\delta}(\varphi_{\delta}) \nabla\varphi_{\delta},\\
  \widetilde{B}(s)&:= B(s)+ a(\cdot)b(s).
\end{align*}
Then the weak formulation of the approximated problem can be rewritten as follows: $\forall \psi\in V$ and $\textbf{v}\in \mathbb{V}_{div}$,
\begin{align}
    \langle \varphi_\delta',\psi\rangle + (\textbf{u}_\delta\nabla\varphi_\delta,\psi) &= (\Delta \widetilde{B}(\varphi_\delta),\psi) - \Big(m(\varphi_\delta)\big((\nabla a) \varphi_\delta -\nabla J*\varphi_\delta\big),\nabla\psi\Big) +\big(b(\varphi_\delta)\nabla a, \nabla\psi\big),\label{eq018}\\ 
-(\nu\Delta\textbf{u}_{\delta},\textbf{v})+ (\eta \textbf{u}_{\delta},\textbf{v}) &= \Big(\frac{-(\nabla a) \varphi_\delta^2}{2} + (J*\varphi_\delta)\nabla\varphi_\delta , \textbf{v}\Big)+ \langle\textbf{h},\textbf{v}\rangle.\label{eq019}
\end{align}

Now we derive the uniform bounds. In the following calculations $c$ and $C$ denote generic constants which are independent of $\delta$.
\\\\(\textbf{Estimate for $\textbf{u}_\delta$})
\\We use  $\textbf{v} = -\Delta \textbf{u}_{\delta}$ in \eqref{eq019}. By using H\"olders, Young's and Gagliardo-Nirenberg inequalities, we estimate each of its terms  and then by absorbing the term with $\lVert\Delta\textbf{u}_\delta\rVert$ to the L.H.S, we obtain, 
\begin{align*}
    \nu\lVert\Delta \textbf{u}_{\delta}\rVert^2+ \eta\lVert\nabla \textbf{u}_{\delta}\rVert^2 &\leq \Big(\Big|\frac{\nabla a}{2}\Big|_{L^\infty}+ \lVert\nabla J \rVert_{L^1}\Big)\lVert\varphi_\delta\rVert_{L^4}^2\lVert\Delta \textbf{u}_{\delta}\rVert + \lVert \textbf{h}\rVert\lVert \Delta \textbf{u}_{\delta}\rVert,\\
    \frac{\nu}{2}\lVert\Delta \textbf{u}_{\delta}\rVert^2 &\leq c\Big(\Big|\frac{\nabla a}{2}\Big|_{L^\infty}+ \lVert\nabla J \rVert_{L^1}\Big)^2\lVert\varphi_\delta\rVert_{L^4}^4 + c\lVert \textbf{h}\rVert^2,
\end{align*}\vspace{-.5cm}
\begin{align}
    \frac{\nu}{2}\lVert\Delta \textbf{u}_{\delta}\rVert^2 &\leq C(\lVert\varphi_\delta\rVert^2\lVert\nabla\varphi_\delta\rVert^2+ \lVert \textbf{h}\rVert^2).
\end{align}
Combined with the energy estimate \eqref{eq1.4} satisfied by the weak solution, we have,
$$\lVert\Delta \textbf{u}_{\delta}\rVert^2 \leq C.$$
where generic constant $C$ is independent of $\delta.$ 
 Using  \cite[Lemma 3.7]{TEM}, we have the equivalence of norms $\lVert\Delta u\rVert$ and the norm induced by $\textbf{H}^2(\Omega)$ on $\mathbb{V}_{div}\cap \textbf{H}^2(\Omega)$ and hence we get $\textbf{u}_{\delta} \in \mathbb{V}_{div}\cap \textbf{H}^2(\Omega)$. 
Moreover, using Agmon's inequality  we get,
\begin{equation}\label{eq093}
    \lVert\textbf{u}_\delta\rVert_{L^\infty} \leq c\lVert\textbf{u}_\delta\rVert_{\textbf{V}_{div}}^\frac{1}{2}\lVert\textbf{u}_\delta\rVert_{\textbf{H}^2}^\frac{1}{2} \leq C.
\end{equation}

Also, $ \| \textbf{u}_{\delta} (t) \|_{\textbf{H}^2} \leq C, \ \forall t \in \ [0, T] $ implies that $ {u}_{\delta} \in L^p(0,T;\textbf{H}^2(\Omega)) \ \text{ for } 1\leq p < \infty. $  In particular we have,
\begin{equation}\label{eq93}
\lVert \textbf{u}_{\delta}\rVert_{L^2(0,T;\textbf{H}^2(\Omega))} \leq C.
\end{equation}

(\textbf{Estimate for $\varphi_\delta$})
\\Using the test function $\psi = \widetilde{B}(\varphi_\delta)' = \big(m(\varphi_\delta)a+\lambda(\varphi_\delta)\big)\varphi_\delta'$ in \eqref{eq018} we get,
\begin{align}\label{eq025}
\frac{1}{2}\frac{d}{dt}\big[\lVert\nabla\widetilde{B}(\varphi_\delta)\rVert^2 &+ 2\big(m(\varphi_\delta)\big((\nabla a)\varphi_\delta-\nabla J*\varphi_\delta\big), \nabla\widetilde{B}(\varphi_\delta)\big) -2\big((\nabla a)b(\varphi_\delta),\nabla\widetilde{B}(\varphi_\delta)\big)\big] + (\varphi_\delta',(m(\varphi_\delta)a+\lambda(\varphi_\delta))\varphi_\delta') \nonumber\\  & = -(\textbf{u}_\delta\varphi_\delta', \nabla\widetilde{B}(\varphi_\delta)) + \big(m'(\varphi_\delta)\varphi_\delta'\big((\nabla a)\varphi_\delta-\nabla J*\varphi_\delta\big),\nabla\widetilde{B}(\varphi_\delta)\big) - \big( (\nabla a)m(\varphi_\delta)\varphi_\delta', \nabla\widetilde{B}(\varphi_\delta)\big)\nonumber\\ &\hspace{.5cm} + \big(m(\varphi_\delta)\big((\nabla a)\varphi_\delta'-\nabla J*\varphi_\delta'\big),\nabla\widetilde{B}(\varphi_\delta)\big).
\end{align}
Define  $\Phi = \lVert\nabla\widetilde{B}(\varphi_\delta)\rVert^2 + 2\big(m(\varphi_\delta)((\nabla a)\varphi_\delta-\nabla J*\varphi_\delta), \nabla\widetilde{B}(\varphi_\delta)\big) -2\big((\nabla a)b(\varphi_\delta),\nabla\widetilde{B}(\varphi_\delta)\big).$
By a direct calculation using H\"olders and Young's inequalities we get,
\begin{equation}\label{eq0026}
    \frac{1}{2}\lVert\nabla\widetilde{B}(\varphi_\delta)\rVert^2-c\lVert\varphi_\delta\rVert^2 \leq \Phi(t) \leq C(\lVert\nabla\widetilde{B}(\varphi_\delta)\rVert^2 + \lVert\varphi_\delta\rVert^2 ).
\end{equation}
In \eqref{eq025}, we estimate each of its terms on R.H.S by applying H\"older's, Young's, Ladyzhenskaya and Gagliardo-Nirenberg inequalities.
\begin{align*}
    |\big(\textbf{u}_\delta\varphi_\delta', \nabla\widetilde{B}(\varphi_\delta)\big)| &\leq \lVert\textbf{u}_\delta\rVert_{L^4}\lVert\varphi_\delta'\rVert \lVert\nabla\widetilde{B}(\varphi_\delta)\rVert_{L^4} \nonumber\\
    &\leq \frac{\alpha_1}{8} \lVert\varphi_\delta'\rVert^2 + c\lVert\nabla\textbf{u}_\delta\rVert^2 \lVert\nabla\widetilde{B}(\varphi_\delta)\rVert\lVert\widetilde{B}(\varphi_\delta)\rVert_{H^2}\nonumber\\
    &\leq \frac{\alpha_1}{8} \lVert\varphi_\delta'\rVert^2 + \epsilon\lVert\widetilde{B}(\varphi_\delta)\rVert_{H^2}^2 + c\lVert\nabla\textbf{u}_\delta\rVert^4 \lVert\nabla\widetilde{B}(\varphi_\delta)\rVert^2,\\
    |\big(m'(\varphi_\delta)\varphi_\delta'((\nabla a)\varphi_\delta-\nabla J*\varphi_\delta),\nabla\widetilde{B}(\varphi_\delta)\big)| &\leq |m'|_{L^\infty}(\lVert\nabla a\rVert_{L^\infty}+\lVert J\rVert_{L^1})\lVert\varphi_\delta\rVert_{L^\infty}\lVert\varphi_\delta'\rVert\lVert\nabla\widetilde{B}(\varphi_\delta)\rVert \nonumber\\
    &\leq \frac{\alpha_1}{8}\lVert\varphi_\delta'\rVert^2 + c\lVert\nabla\widetilde{B}(\varphi_\delta)\rVert^2, \nonumber\\
    |\big(m(\varphi_\delta)((\nabla a)\varphi_\delta'-\nabla J*\varphi_\delta'),\nabla\widetilde{B}(\varphi_\delta)\big)| &\leq |m|_{L^\infty}(\lVert\nabla a\rVert_{L^\infty}+\lVert J\rVert_{L^1})\lVert\varphi_\delta'\rVert\lVert\nabla\widetilde{B}(\varphi_\delta)\rVert \nonumber\\
    &\leq \frac{\alpha_1}{8}\lVert\varphi_\delta'\rVert^2 + c\lVert\nabla\widetilde{B}(\varphi_\delta)\rVert^2, \nonumber\\
    |\big( (\nabla a)m(\varphi)\varphi', \nabla\widetilde{B}(\varphi_\delta)\big)| &\leq |m|_{L^\infty}\lVert\nabla a\rVert_{L^\infty}\lVert\varphi_\delta'\rVert\lVert\nabla\widetilde{B}(\varphi_\delta)\rVert \nonumber\\
    &\leq \frac{\alpha_1}{8}\lVert\varphi_\delta'\rVert^2 + c\lVert\nabla\widetilde{B}(\varphi_\delta)\rVert^2.
\end{align*}
Substituting all the above estimates in \eqref{eq025}, we get,
\begin{align}\label{eq026}
    \frac{1}{2}\frac{d}{dt}\Phi + \frac{\alpha_1}{2}\lVert\varphi_\delta'\rVert^2  \leq \epsilon\lVert\widetilde{B}(\varphi_\delta)\rVert_{H^2}^2 + C(1+\lVert\nabla\textbf{u}_\delta\rVert^4) \lVert\nabla\widetilde{B}(\varphi_\delta)\rVert^2. 
\end{align}
In order to handle the term $\lVert\nabla\widetilde{B}(\varphi_\delta)\rVert$ in \eqref{eq026}, we estimate this term in terms of $\lVert\varphi_\delta\rVert_{V}$.
Consider the expression, $ \nabla\widetilde{B}(\varphi_\delta) = (m(\varphi_\delta)a+ \lambda(\varphi_\delta))\nabla\varphi_\delta + (\nabla a)b(\varphi_\delta)$. Using the lower bound for $(ma+\lambda)$ from the assumption \textbf{[A4]} and applying H\"olders and Young's inequalities we obtain,
\begin{align*}
    \lVert\nabla\widetilde{B}(\varphi_\delta)\rVert^2 &\geq \alpha_1^2\lVert\nabla\varphi_\delta\rVert^2 - 2(|m|_{L^\infty}\lVert a\rVert_{L^\infty} + |\lambda|_{L^\infty})\lVert\nabla a\rVert_{L^\infty}\lVert b(\varphi_\delta)\rVert\lVert\nabla\varphi_\delta\rVert \nonumber\\
    &\geq \frac{\alpha_1^2}{2}\lVert\nabla\varphi_\delta\rVert^2 - c\lVert b(\varphi_\delta)\rVert^2 \geq \frac{\alpha_1^2}{2}\lVert\nabla\varphi_\delta\rVert^2 - c\lVert \varphi_\delta\rVert^2. 
\end{align*}
Moreover, using the upper bound for $(ma+\lambda)$  that follows from the assumption $m\in C^1[-1,1]$ , $\lambda \in C^1[-1,1],$  for any $\psi \in H$ we will get,
\begin{align*}
 |( \nabla\widetilde{B}(\varphi_\delta), \psi)| &= |\big( (m(\varphi_\delta)a+ \lambda(\varphi_\delta))\nabla\varphi_\delta + (\nabla a)b(\varphi_\delta) ,\psi\big) |   \; \leq C \| \varphi_\delta\|_{V} \| \psi \|.
\end{align*}
gives us an upper bound for $\lVert\nabla\widetilde{B}(\varphi_\delta)\rVert$.
Combining these two estimates, we get
\begin{align}\label{eq027}
    \frac{\alpha_1^2}{2}\lVert\nabla\varphi_\delta\rVert^2 - c\lVert \varphi_\delta\rVert^2 \leq \lVert\nabla\widetilde{B}(\varphi_\delta)\rVert^2 \leq C\lVert\varphi_\delta\rVert_{V}^2.
\end{align}
\\
\textbf{(Estimate for $\lVert\widetilde{B}(\varphi_\delta)\rVert_{H^2}$)}
 We have the following elliptic estimate \cite{FGS},
\begin{equation}\label{eq028}
    \lVert\widetilde{B}(\varphi_\delta)\rVert_{H^2} \leq C(\lVert\Delta\widetilde{B}(\varphi_\delta)\rVert + \lVert\widetilde{B}(\varphi_\delta)\rVert_{V} + \lVert\nabla\widetilde{B}(\varphi_\delta).\textbf{n}\rVert_{H^\frac{1}{2}(\partial\Omega)} ).
\end{equation}

by applying H\"olders and Young's inequalities in \eqref{eq018} and using \eqref{eq093}, we get,
\begin{align*}
    \lVert\Delta\widetilde{B}(\varphi_\delta)\rVert &\leq \lVert\varphi_\delta'\rVert + \lVert\textbf{u}_\delta\rVert_{L^\infty}\lVert\nabla\varphi_\delta\rVert + |m'|_{L^\infty}(\lVert \nabla a\rVert_{L^\infty}+\lVert \nabla J\rVert_{L^1})\lVert\varphi_\delta\rVert_{L^\infty}\lVert\nabla\varphi_\delta\rVert \nonumber \\
    &\hspace{.5cm} + |m|_{L^\infty}(\lVert \Delta a\rVert_{L^\infty}+\lVert \Delta J\rVert_{L^1})\lVert\varphi_\delta\rVert + 2|m|_{L^\infty}\lVert \nabla a\rVert_{L^\infty}\lVert\nabla\varphi_\delta\rVert  +  \lVert \nabla a\rVert_{L^\infty}\lVert b(\varphi_\delta)\rVert \nonumber \\ &\leq \lVert\varphi_\delta'\rVert + c\lVert\varphi_\delta\rVert_{V}.
\end{align*}
We have on the boundary $\partial\Omega$, $ (m\nabla\mu)\cdot\textbf{n} = \big[\nabla\widetilde{B}(\varphi_\delta)- (\nabla a)b(\varphi_\delta) + m(\varphi_\delta)\big((\nabla a)\varphi_\delta-\nabla J*\varphi_\delta\big)\big]\cdot\textbf{n} = 0 $. Using the continuous embedding $V(\Omega)\subset H^\frac{1}{2}(\partial\Omega)$, we get,
\begin{align*}
    \lVert\nabla\widetilde{B}(\varphi_\delta).\textbf{n}\rVert_{H^\frac{1}{2}(\partial\Omega)} &\leq \lVert(\nabla a)b(\varphi_\delta)\cdot\textbf{n}\rVert_{H^\frac{1}{2}(\partial\Omega)} + \lVert m(\varphi_\delta)\big((\nabla a)\varphi_\delta-\nabla J*\varphi_\delta\big)\cdot\textbf{n}\rVert_{H^\frac{1}{2}(\partial\Omega)} \\
    &\leq \lVert(\nabla a)b(\varphi_\delta)\rVert_{V} + \lVert m(\varphi_\delta)\big((\nabla a)\varphi_\delta-\nabla J*\varphi_\delta\big)\rVert_{V} \\
    &\leq c\lVert\varphi_\delta\rVert_{V}.
\end{align*}
Substituting back in \eqref{eq028}, we obtain,
\begin{equation}\label{eq029}
    \lVert\widetilde{B}(\varphi_\delta)\rVert_{H^2} \leq C\big( \lVert\varphi_\delta'\rVert + \lVert\varphi_\delta\rVert_{V} \big).
\end{equation}
Substitute \eqref{eq029} in \eqref{eq026} and use \eqref{eq0026} with an choice of $\epsilon = \frac{\alpha_1}{4C}$ yields,
\begin{align}\label{eq03.28}
     \frac{1}{2}\frac{d}{dt}\Phi + \frac{\alpha_1}{4}\lVert\varphi_\delta'\rVert^2 &\leq  \frac{\alpha_1}{4}\lVert\varphi_\delta\rVert^2  + C(1+\lVert\nabla\textbf{u}_\delta\rVert^4) \lVert\nabla\widetilde{B}(\varphi_\delta)\rVert^2 \nonumber \\
    &\leq  C(1+\lVert\nabla\textbf{u}_\delta\rVert^4) (\Phi + \lVert\varphi_\delta\rVert^2).
\end{align}
 Using test function $\psi = \varphi_\delta $ in \eqref{eq018} and  applying H\"olders and Young's inequalities we get,
\begin{equation*}
    \frac{1}{2}\frac{d}{dt}\lVert\varphi_\delta\rVert^2 \leq c(\lVert\nabla\widetilde{B}(\varphi_\delta)\rVert^2 + \lVert\varphi\rVert_{V}^2).
\end{equation*}
Comparing with \eqref{eq0026}, \eqref{eq027} and adding to \eqref{eq03.28},   
\begin{align}
     \frac{1}{2}\frac{d}{dt}(\Phi + \lVert\varphi_\delta\rVert^2) + \frac{\alpha_1}{4}\lVert\varphi_\delta'\rVert^2 \leq  C(1+\lVert\nabla\textbf{u}_\delta\rVert^4) (\Phi + \lVert\varphi_\delta\rVert^2).
\end{align}
Applying Gronwall's inequality, we obtain, 
\begin{align}\label{eq3.030}
    \Phi(t) + \lVert\varphi_\delta(t)\rVert^2 + \frac{\alpha_1}{4}\lVert\varphi_\delta'\rVert_{L^2(0,T;H)}^2 &\leq C(1+ \gamma(T)) \leq C, 
\end{align}
holds for a.e $t\in [0,T]$, for some $\gamma$, function of $t$.
Now by comparing \eqref{eq027} and \eqref{eq0026} and applying the uniform estimate \eqref{eq3.030}, we obtain,
\begin{align}
\lVert \varphi_\delta\rVert_{L^\infty(0,T;V)} &\leq C.\label{eq2.26}
\end{align}
Further, substituting uniform estimates \eqref{eq3.030} and \eqref{eq2.26} in \eqref{eq029}, we get,
\begin{align}
\lVert\widetilde{B}(\varphi_\delta)\rVert_{L^2(0,T;H^2)} &\leq C.\label{eq2.27}
\end{align}
\\\textbf{(Estimate for $\lVert\varphi_\delta\rVert_{H^2}$)}
\\\\We have in a weak sense, $\partial_i\widetilde{B}(\varphi_\delta) = \big(m(\varphi_\delta)a + \lambda(\varphi_\delta)\big)\partial_i\varphi_\delta + (\partial_ia)b(\varphi_\delta)$, $i=1,2.$ 
 And the second order partial derivative, for $i, j \in \{1,2\}$, 
$$\partial_{ij}\widetilde{B}(\varphi_\delta) = (m(\varphi_\delta)a + \lambda(\varphi_\delta))\partial_{ij}\varphi_\delta + \big((m'(\varphi_\delta)a + \lambda'(\varphi_\delta))\partial_j\varphi_\delta + m(\varphi_\delta)\partial_ja\big)\partial_i\varphi_\delta + \partial_j((\partial_ia)b(\varphi_\delta)).$$  
Since by [\textbf{A4}], $(m(\varphi_\delta)a+\lambda(\varphi))$ is bounded away from zero, we write,
\begin{equation*}
    \partial_{ij}\varphi_\delta =  \frac{\Big(\partial_{ij}\widetilde{B}(\varphi_\delta)- \partial_j\big(b(\varphi_\delta)\partial_ia\big)\Big)}{\big(m(\varphi_\delta)a + \lambda(\varphi_\delta)\big)} - \frac{\big( \partial_i\widetilde{B}(\varphi_\delta)- (\partial_ia)b(\varphi_\delta)\big)\Big((m'(\varphi_\delta)a+\lambda'(\varphi_\delta))\partial_j\varphi_\delta + m(\varphi_\delta)\partial_ja \Big)}{\big(m(\varphi_\delta)a + \lambda(\varphi_\delta)\big)^2}.
\end{equation*}
Applying H\"olders, Gagliardo-Nirenberg, Young's inequalities and the assumption $\textbf{[A4]}$ in the above expression, we get the following estimate for $\lVert\partial_{ij}\varphi_\delta\rVert$.
\begin{align*}
    \lVert\partial_{ij}\varphi_\delta\rVert &\leq \frac{\lVert \widetilde{B}(\varphi_\delta)\rVert_{H^2}}{\alpha_1} + c\lVert\varphi_\delta\rVert_{V} + c\lVert\nabla\widetilde{B}(\varphi_\delta)\rVert_{L^4}\lVert\nabla\varphi_\delta\rVert_{L^4} + c\lVert b(\varphi_\delta)\rVert_{L^4}\lVert\nabla\varphi_\delta\rVert_{L^4} \\
    &\leq \frac{\lVert \widetilde{B}(\varphi_\delta)\rVert_{H^2}}{\alpha_1} + c\lVert\varphi_\delta\rVert_{V} + 2\epsilon\lVert\varphi_\delta\rVert_{H^2} + c\lVert\nabla\varphi_\delta\rVert\lVert\nabla\widetilde{B}(\varphi_\delta)\rVert\lVert\widetilde{B}(\varphi_\delta)\rVert_{H^2}.
\end{align*}
This gives for a choice of $\epsilon <\frac{1}{4},$
\begin{align}\label{eq033}
    \lVert\varphi_\delta\rVert_{H^2} &\leq \Big(\frac{1}{\alpha_1}+ c\lVert\varphi_\delta\rVert_{V}^2\Big)\lVert \widetilde{B}(\varphi_\delta)\rVert_{H^2} + c\lVert\varphi_\delta\rVert_{V}.
\end{align}
Moreover, by applying the uniform estimates \eqref{eq2.26} and \eqref{eq2.27} in \eqref{eq033}, we obtain,
\begin{align}
\lVert\varphi_\delta\rVert_{L^2(0,T;H^2)} \leq C.\label{eq2.29}
\end{align}
\\\textbf{(Estimate for $\lVert\textbf{u}_\delta'\rVert$)}
\\\\Define difference quotients, $D_k\textbf{u}_\delta(t) = \frac{\textbf{u}_\delta(t+k)-\textbf{u}_\delta(t)}{k}$, $D_k\varphi_\delta(t) = \frac{\varphi_\delta(t+k)-\varphi_\delta(t)}{k}$ and $D_k\textbf{h}_\delta(t) = \frac{\textbf{h}_\delta(t+k)-\textbf{h}_\delta(t)}{k}$ for $ k>0 $  . We will follow similar arguments as in \cite[Theorem 4.1]{DAG} used to obtain a higher regularity of $\textbf{u}'$.  We fix $k > 0$ such that $ t + k < T$.  
Consider the following equation obtained by taking the difference of equations satisfied by $u_{\delta}$ \eqref{eq019} at two different times $t+k$ and $t$.
\begin{align*}
    (-\nu\Delta D_k\textbf{u}_\delta, \textbf{v}) + (\eta D_k\textbf{u}_\delta, \textbf{v}) &= -\big(\frac{\nabla a}{2}(\varphi_\delta(t+k)+\varphi_\delta(t))D_k\varphi_\delta, \textbf{v} \big) + \big((\nabla J*\varphi_\delta(t+k))D_k\varphi_\delta, \textbf{v}\big) \\ &\hspace{.5cm}+ \big((\nabla J *D_k\varphi_\delta)\varphi_\delta(t), \textbf{v} \big) + \langle D_k\textbf{h}, \textbf{v} \rangle,
\end{align*}
For all $\textbf{v}\in \mathbb{G}_{div}$, and a.e $ t\in [0,T]$. For $\textbf{v} = D_k\textbf{u}_\delta$, we get,
\begin{align*}
    \nu\lVert\nabla D_k\textbf{u}_\delta\rVert^2 + \eta\lVert D_k\textbf{u}_\delta\rVert^2 &= -\Big(\frac{\nabla a}{2}(\varphi_\delta(t+k)+\varphi_\delta(t))D_k\varphi_\delta , D_k\textbf{u}_\delta\Big) + \Big( (\nabla J*\varphi_\delta(t+k))D_k\varphi_\delta, D_k\textbf{u}_\delta\Big) \\ &\hspace{.5cm} + \Big( (\nabla J *D_k\varphi_\delta)\varphi_\delta(t), D_k\textbf{u}_\delta\Big) + \langle D_k\textbf{h}, D_k\textbf{u}_\delta \rangle.
\end{align*}
Further, using H\"olders, Young's and Poincar\'e inequalities, we have,
\begin{align*}
    \nu\lVert\nabla D_k\textbf{u}_\delta\rVert^2 + \eta\lVert D_k\textbf{u}_\delta\rVert^2 &\leq \lVert\nabla a\rVert_{L^\infty}\lVert\varphi_\delta\rVert_{L^\infty}\lVert D_k\varphi_\delta\rVert \lVert D_k\textbf{u}_\delta\rVert + \lVert\nabla J\rVert_{L^1}\lVert\varphi_\delta\rVert_{L^\infty}\lVert D_k\varphi_\delta\rVert \lVert D_k\textbf{u}_\delta\rVert \\ &\hspace{.5cm} + \lVert\nabla J \rVert_{L^1}\lVert D_k\varphi_\delta\rVert\lVert\varphi_\delta\rVert_{L^\infty}\lVert D_k\textbf{u}_\delta\rVert + \lVert D_k\textbf{h}\rVert_{\mathbb{V}_{div}'} \lVert D_k\textbf{u}_\delta\rVert_{\mathbb{V}_{div}}\\
    &\leq \frac{\eta}{2}\lVert D_k\textbf{u}_\delta\rVert^2 + \frac{\nu}{2}\lVert\nabla D_k\textbf{u}_\delta\rVert^2 + C(\lVert D_k\varphi_\delta\rVert^2 + \lVert D_k\textbf{h}\rVert_{\mathbb{V}_{div}'}^2).
\end{align*}
Now by applying Korn's inequality \cite{SKN} to bound the L.H.S from below, we get,
\begin{align}\label{eq2.30}
    \lVert D_k\textbf{u}_\delta\rVert_{\mathbb{V}_{div}}^2 \leq C( \lVert D_k\varphi_\delta\rVert^2 + \lVert D_k\textbf{h}\rVert_{\mathbb{V}_{div}'}^2).
\end{align}
Differentiate both sides of \eqref{eq018} w.r.t '$t$' we get, $\forall \psi \in H$,
\begin{align*}
   \langle\varphi_\delta'', \psi\rangle + (\textbf{u}_\delta'\nabla\varphi_\delta, \psi) +  (\textbf{u}_\delta\nabla\varphi_\delta', \psi) &= (\Delta \widetilde{B}(\varphi_\delta)', \psi) -\big(m'(\varphi_\delta)\varphi_\delta'\big((\nabla a) \varphi_\delta -\nabla J*\varphi_\delta\big), \nabla \psi \big) \nonumber \\ &\hspace{.5cm} + \big(m(\varphi_\delta)\varphi_\delta'\nabla a, \nabla \psi\big)  - \big(m(\varphi_\delta)\big((\nabla a) \varphi_\delta' -\nabla J*\varphi_\delta'\big), \nabla\psi\big).
\end{align*}
Substitute $\psi = \widetilde{B}(\varphi_\delta)' = \big(m(\varphi_\delta)a + \lambda(\varphi_\delta)\big)\varphi_\delta'$, we get,
\begin{align}\label{eq035}
    (\varphi_\delta'',\widetilde{B}(\varphi_\delta)') &+ (\textbf{u}_\delta'\nabla\varphi_\delta,\widetilde{B}(\varphi_\delta)') +  (\textbf{u}_\delta\nabla\varphi_\delta', \widetilde{B}(\varphi_\delta)') + \lVert \nabla\widetilde{B}(\varphi_\delta)'\rVert^2 = \big(m(\varphi_\delta)\varphi_\delta'\nabla a, \nabla\widetilde{B}(\varphi_\delta)'\big) \nonumber\\ & -\big(m'(\varphi_\delta)\varphi_\delta'\big((\nabla a) \varphi_\delta -\nabla J*\varphi_\delta\big), \nabla\widetilde{B}(\varphi_\delta)'\big)  -\big(m(\varphi_\delta)\big((\nabla a) \varphi_\delta' -\nabla J*\varphi_\delta'\big), \nabla\widetilde{B}(\varphi_\delta)'\big). 
\end{align}
Estimating each of its terms, we get,
\begin{align*}
    (\varphi_\delta'',\widetilde{B}(\varphi_\delta)') &=(\varphi_\delta'',(ma +\lambda)\varphi_\delta') \geq  \frac{\alpha_1}{2}\frac{d}{dt}\lVert\varphi_\delta'\rVert^2\\
    |(\textbf{u}_\delta'\nabla\varphi_\delta,\widetilde{B}(\varphi_\delta)')| &=|(\textbf{u}_\delta'\varphi_\delta,\nabla\widetilde{B}(\varphi_\delta)')| \leq \lVert\textbf{u}_\delta'\rVert\lVert\varphi_\delta\rVert_{L^\infty}\lVert\nabla\widetilde{B}(\varphi_\delta)'\rVert\\
    &\leq \frac{\alpha_1}{10} \lVert\nabla\widetilde{B}(\varphi_\delta)'\rVert^2 + c\lVert\textbf{u}_\delta'\rVert^2,\\
    |(\textbf{u}_\delta\nabla\varphi_\delta', \widetilde{B}(\varphi_\delta)')| &= |(\textbf{u}_\delta\varphi_\delta',\nabla\widetilde{B}(\varphi_\delta)')| \leq \lVert\textbf{u}_\delta\rVert_{L^\infty}\lVert\varphi_\delta'\rVert\lVert\nabla\widetilde{B}(\varphi_\delta)'\rVert\\
    &\leq \frac{\alpha_1}{10} \lVert\nabla\widetilde{B}(\varphi_\delta)'\rVert^2 + c\lVert\varphi_\delta'\rVert^2,
\end{align*}
\begin{align*}
    |\big(m(\varphi_\delta)\varphi_\delta'\nabla a, \nabla\widetilde{B}(\varphi_\delta)'\big)| &\leq c\lVert\varphi_\delta'\rVert\lVert\nabla\widetilde{B}(\varphi_\delta)'\rVert
    \leq \frac{\alpha_1}{10} \lVert\nabla\widetilde{B}(\varphi_\delta)'\rVert^2 + c\lVert\varphi_\delta'\rVert^2,\\
    |\big(m'(\varphi_\delta)\varphi_\delta'\big((\nabla a) \varphi_\delta -\nabla J*\varphi_\delta\big), \nabla\widetilde{B}(\varphi_\delta)'\big)| &\leq c\lVert\varphi_\delta'\rVert\lVert\varphi_\delta\rVert_{L^\infty}\lVert\nabla\widetilde{B}(\varphi_\delta)'\rVert \leq \frac{\alpha_1}{10} \lVert\nabla\widetilde{B}(\varphi_\delta)'\rVert^2 + c\lVert\varphi_\delta'\rVert^2,\\
    |\big(m(\varphi_\delta)\big((\nabla a) \varphi_\delta' -\nabla J*\varphi_\delta'\big), \nabla\widetilde{B}(\varphi_\delta)'\big)| &\leq c\lVert\varphi_\delta'\rVert\lVert\nabla\widetilde{B}(\varphi_\delta)'\rVert \leq \frac{\alpha_1}{10} \lVert\nabla\widetilde{B}(\varphi_\delta)'\rVert^2 + c\lVert\varphi_\delta'\rVert^2.
\end{align*}
Therefore substituting back in \eqref{eq035}, we get,\vspace{-.15cm}
\begin{align}
    \frac{\alpha_1}{2}\frac{d}{dt}\lVert\varphi_\delta'\rVert^2 + \lVert \nabla\widetilde{B}(\varphi_\delta)'\rVert^2 \leq  C(\lVert\varphi_\delta'\rVert^2 + \lVert\textbf{u}_\delta'\rVert^2).
\end{align}
Comparing with the estimate \eqref{eq2.30} and then applying the Gronwall's inequality we obtain,
\begin{align}
\lVert\varphi_\delta'\rVert_{L^\infty(0,T;H)} &\leq C, \label{eq2.34}\\
\lVert\nabla\widetilde{B}(\varphi_\delta)'\rVert_{L^2(0,T;H)} &\leq C.\label{eq2.35}
\end{align}
Since $\nabla\widetilde{B}(\varphi_\delta)'= \big(m(\varphi_\delta)a + \lambda(\varphi_\delta)\big)\nabla\varphi_\delta' + \big(m'(\varphi_\delta)a + \lambda'(\varphi_\delta)\big) \nabla\varphi_\delta\varphi_\delta' + m(\varphi_\delta)\nabla a\varphi_\delta'$,
Using H\"olders, Gagliardo-Nirenberg and Young's inequalities we get,
\begin{align*}
    \lVert\nabla\widetilde{B}(\varphi_\delta)'\rVert^2 &\geq \alpha_1^2\lVert\nabla\varphi_\delta'\rVert^2 -c\big[\lVert\nabla\varphi_\delta'\rVert\lVert\varphi_\delta'\rVert +\lVert\nabla\varphi_\delta'\rVert\lVert\nabla\varphi_\delta\rVert_{L^4}\lVert\varphi_\delta'\rVert_{L^4} +\lVert\nabla\varphi_\delta\rVert_{L^4}^2\lVert\varphi_\delta'\rVert_{L^4}^2 \\&\hspace{.5cm}+\lVert\nabla\varphi_\delta\rVert\lVert\varphi_\delta'\rVert_{L^4}^2 +\lVert\varphi'\rVert^2 \big]\\
    &\geq \alpha_1^2\lVert\nabla\varphi_\delta'\rVert^2 -c\big[\lVert\nabla\varphi_\delta'\rVert\lVert\varphi_\delta'\rVert +\lVert\nabla\varphi_\delta'\rVert\lVert\nabla\varphi_\delta\rVert^\frac{1}{2}\lVert\varphi_\delta\rVert_{H^2}^\frac{1}{2}\lVert\varphi_\delta'\rVert^\frac{1}{2}\lVert\nabla\varphi_\delta'\rVert^\frac{1}{2} \\&\hspace{.5cm} +\lVert\nabla\varphi_\delta\rVert\lVert\varphi_\delta\rVert_{H^2}\lVert\varphi_\delta'\rVert\lVert\nabla\varphi_\delta'\rVert +\lVert\nabla\varphi_\delta\rVert\lVert\varphi_\delta'\rVert\lVert\nabla\varphi_\delta'\rVert +\lVert\varphi'\rVert^2 \big]\\
    &\geq  \alpha_1^2\lVert\nabla\varphi_\delta'\rVert^2 - \frac{\alpha_1^2}{4}\lVert\nabla\varphi_\delta'\rVert^2 -c\big[ \lVert\varphi_\delta'\rVert^2 + \lVert\nabla\varphi_\delta\rVert\lVert\varphi_\delta\rVert_{H^2}\lVert\varphi_\delta'\rVert\lVert\nabla\varphi_\delta'\rVert +\lVert\nabla\varphi_\delta\rVert\lVert\varphi_\delta'\rVert\lVert\nabla\varphi_\delta'\rVert\big]\\
    &\geq \frac{3\alpha_1^2}{4}\lVert\nabla\varphi_\delta'\rVert^2 - \frac{\alpha_1^2}{4}\lVert\nabla\varphi_\delta'\rVert^2 - c[\lVert\varphi_\delta'\rVert^2 + \lVert\nabla\varphi_\delta\rVert^2\lVert\varphi_\delta\rVert_{H^2}^2\lVert\varphi_\delta'\rVert^2 + \lVert\nabla\varphi_\delta\rVert^2\lVert\varphi_\delta'\rVert^2]\\
    &\geq \frac{\alpha_1^2}{2}\lVert\nabla\varphi_\delta'\rVert^2 - C_1\lVert\varphi_\delta'\rVert^2\big( 1 + \lVert\nabla\varphi_\delta\rVert^2 + \lVert\nabla\varphi_\delta\rVert^2\lVert\varphi_\delta\rVert_{H^2}^2\big).
\end{align*}
Therefore,
\begin{equation}\label{eq036}
    \frac{\alpha_1^2}{2}\lVert\nabla\varphi_\delta'\rVert^2 \leq \lVert\nabla\widetilde{B}(\varphi_\delta)'\rVert^2 + C\lVert\varphi_\delta'\rVert^2\big( 1 + \lVert\nabla\varphi_\delta\rVert^2 + \lVert\nabla\varphi_\delta\rVert^2\lVert\varphi_\delta\rVert_{H^2}^2\big).
\end{equation}
Applying uniform estimates \eqref{eq2.35}, \eqref{eq2.34}, \eqref{eq2.29} and \eqref{eq2.26} in \eqref{eq036}, we obtain,\vspace{-.15cm}
\begin{align}\label{eq2.37}
\lVert\varphi_\delta'\rVert_{L^2(0,T;V)} \leq C.
\end{align}
Similarly, applying \eqref{eq2.34} and \eqref{eq2.26} in \eqref{eq029} we get,\vspace{-.15cm}
\begin{align} 
\lVert\widetilde{B}(\varphi_\delta)\rVert_{L^\infty(0,T;H^2)} &\leq C,\label{eq2.38}
\end{align}
Moreover, using uniform estimates \eqref{eq2.26} and \eqref{eq2.38} in \eqref{eq033}, we have,\vspace{-.15cm}
\begin{align}
\lVert\varphi_\delta\rVert_{L^\infty(0,T;H^2)} &\leq C.\label{eq2.39}
\end{align}
Therefore from the estimates \eqref{eq93}, \eqref{eq2.29}, \eqref{eq2.34}, \eqref{eq2.37} and \eqref{eq2.39}  that are uniform in $\delta$, we have, $(\varphi, \textbf{u})$ of regularity \eqref{eq2.08}-\eqref{eq2.010} such that up to a subsequence,
 
 \begin{align}
\varphi_{\delta} &\overset{\ast}{\rightharpoonup} \varphi &&\text{ in }L^{\infty}(0,T;H^2)\label{eq98}\\
\varphi_{\delta} &\rightharpoonup \varphi &&\text{ in }  L^2(0,T;H^2(\Omega))\label{eq099}\\
\varphi_{\delta}' &\overset{\ast}{\rightharpoonup} \varphi' &&\text{ in }L^{\infty}(0,T;H)\label{eq99}\\
\varphi_{\delta}' &\rightharpoonup \varphi' &&\text{ in } L^2(0,T;V)\label{eq100}\\
\textbf{u}_{\delta} &\rightharpoonup \textbf{u} &&\text{ in } L^2(0,T;\textbf{H}^2(\Omega))\label{eq101} 
\end{align}

Using Aubin-Lions lemma for the embedding, $H^2 \hookrightarrow H^2 \hookrightarrow V$ and regularity of $(\varphi, \textbf{u})$ as in \eqref{eq98} and \eqref{eq100}, we have, $\varphi_{\delta} \rightarrow \varphi$ strongly in $C([0,T];H^2)$.
 The above proved convergence results enables us to pass to the limit $\delta \rightarrow 0$ in the approximated problem \eqref{eq23a}- \eqref{eq29a} by using similar arguments as in \cite{CEG} and hence existence of a strong solution, $(\varphi, \textbf{u})$ to the system \eqref{eq1}-\eqref{eq07}.
\end{proof}
The uniqueness of solution follows from a weak-strong uniqueness argument. 
\begin{theorem}[\textbf{Weak-strong uniqueness}]
Let $\varphi_0 \in H^2(\Omega) \cap L^\infty(\Omega),$ $\textbf{h}\in \mathcal{U}$ and $J \in W^{2,1}(\Omega)$. $F$ be either $F_{do}$ or $F_{log}$. Assume $[\textbf{J}], [\textbf{A1}]-[\textbf{A4}]$ are satisfied and in addition assume that $m \in C^1[-1,1]$ and $\lambda \in C^1[-1,1]$. Let $(\varphi_1,\textbf{u}_1)$ be a weak solution and $(\varphi_2,\textbf{u}_2)$ be a strong solution to the nonlocal CHB system that corresponds to initial data, $\varphi_{10}, \varphi_{20} \in H^2(\Omega)\cap L^\infty(\Omega)$ and external forces, $\textbf{h}_1, \textbf{h}_2\in \mathcal{U}.$ Then the following estimate holds.
\begin{align}\label{eq10}
    \lVert\varphi_1-\varphi_2\rVert_{L^\infty(0,T;H)\cap L^2(0,T;V)}^2 &+ c_1\lVert \textbf{u}_1-\textbf{u}_2\rVert_{L^2(0,T;\mathbb{V}_{div})}^2 + \lVert \varphi_1'-\varphi_2'\rVert_{L^2(0,T;V')}^2 \nonumber\\&\leq C\big(\gamma(t)|\varphi_{10}-\varphi_{20}|+\lVert \textbf{h}_1-\textbf{h}_2\rVert_{L^2(0,T;\mathbb{V}_{div}')}^2\big).
\end{align}
\end{theorem}
The proof follows using similar techniques as in \cite[Theorem 6.1]{CSM}. In fact we can get a continuous dependence estimate in a weaker space using the regularity of the strong solution. Hence the uniqueness of the solution follows.

\subsection{Continuous Dependence Estimates of Higher Regularity}

In this section, we prove a higher regular continuous dependence estimate, \eqref{eq3.1} holds under the higher regularity assumptions on the mobility and potential. The higher-order stability estimate has a crucial role in deriving the necessary optimality conditions. 

\begin{theorem}[Stability estimate]Let $\varphi_0 \in H^2(\Omega) \cap L^\infty(\Omega),$ $\textbf{h}\in \mathcal{U}$ and $J \in W^{2,1}(\Omega)$. Assume $[\textbf{J}], [\textbf{A1}]-[\textbf{A4}]$ holds. In addition, assume that the mobility, $m \in C^2[-1,1]$ and $\lambda \in C^2[-1,1]$. Let $(\varphi_1,\textbf{u}_1)$, $(\varphi_2,\textbf{u}_2)$ be two strong solutions to the system \eqref{eq1}-\eqref{eq07} corresponding to initial states, $\varphi_{10},\varphi_{20} \in H^2(\Omega)\cap L^\infty(\Omega)$ and external forces, $\textbf{h}_1, \textbf{h}_2\in \mathcal{U}$ respectively. Then the following estimate holds.
\begin{align}\label{eq3.1}
   \lVert\varphi_1-\varphi_2\rVert_{L^2(0,T:H^2)}^2+\frac{c_0}{2}\lVert\varphi_1-\varphi_2\rVert_{L^\infty(0,T:V)}^2 &+ \frac{\alpha_0}{2}\lVert\varphi_1'-\varphi_2'\rVert_{L^2(0,T;H)}^2 +  \frac{c_1}{2}\lVert \textbf{u}_1-\textbf{u}_2\rVert_{L^2(0,T;\mathbb{V}_{div})}^2 \nonumber\\&\leq  C\big(1 + \gamma(t)\big) \big(|\varphi_{10}-\varphi_{20}|^2 + \lVert \textbf{h}_1-\textbf{h}_2\rVert_{L^2(0,T;\mathbb{V}'_{div})}^2\big). 
\end{align} 
\end{theorem}

\begin{proof}
 Let $(\varphi_1,\textbf{u}_1)$, $(\varphi_2,\textbf{u}_2)$ be two strong solutions to the system \eqref{eq1}-\eqref{eq07} defined as in theorem 3.1 corresponding to the initial data, $\varphi_{10}, \varphi_{20}\in H^2(\Omega)\cap L^\infty(\Omega)$ and external forces, $\textbf{h}_1, \textbf{h}_2\in \mathcal{U}$ respectively.
Let $\varphi:= \varphi_1-\varphi_2$, $\textbf{u}:= \textbf{u}_1-\textbf{u}_2$, $\varphi_0 = \varphi_{10}-\varphi_{20}$ and $\textbf{h}:= \textbf{h}_1-\textbf{h}_2$. Then $(\varphi, \textbf{u})$ satisfies the following system. 
\begin{align}
    \varphi' + \textbf{u}_1\nabla\varphi + \textbf{u}\nabla\varphi_2 &= \nabla\cdot\big((m(\varphi_1)a+ \lambda(\varphi_1))\nabla\varphi\big) + \nabla\cdot\big((m(\varphi_1)-m(\varphi_2))a\nabla\varphi_2\big) + \nabla\cdot\big((\lambda(\varphi_1)- \lambda(\varphi_2))\nabla\varphi_2\big) \nonumber\\&\hspace{.5cm} +
    \nabla\cdot\big(m(\varphi_1)(\varphi\nabla a - \nabla J*\varphi)\big) +
    \nabla\cdot\big((m(\varphi_1)-m(\varphi_2))(\varphi_2\nabla a - \nabla J*\varphi_2)\big), \label{eq3.052}\\
    -\nu\Delta\textbf{u} + \eta\textbf{u} + \nabla\widetilde{\pi} &= (a\varphi_1 - J*\varphi_1)\nabla\varphi + (a\varphi - J*\varphi)\nabla\varphi_2 + \nabla(F(\varphi_1)-F(\varphi_2)) + \textbf{h}.\label{eq3.053}
\end{align}
Recall the notations \eqref{eq1.13a}-\eqref{eq1.15}. Using this notation we have the following weak formulation of the system \eqref{eq3.052}-\eqref{eq3.053}.  $\forall \psi \in V$ and $\textbf{v} \in \mathbb{G}_{div}$,
\begin{align}
&\langle\varphi', \psi\rangle + (\nabla(\widetilde{B}(\varphi_1) -\widetilde{B}(\varphi_2)),\nabla\psi)+ ((\nabla a)(b(\varphi_2)-b(\varphi_1)) , \nabla\psi)+ \big(m(\varphi_1)((\nabla a)\varphi-\nabla J*\varphi),\nabla\psi\big)  \nonumber\\ &+\big((m(\varphi_1)-m(\varphi_2))((\nabla a)\varphi_2-\nabla J*\varphi_2), \nabla\psi\big)+ (\textbf{u}_1\nabla\varphi, \psi)+ (\textbf{u}\nabla\varphi_2, \psi) \hspace{.1cm}= 0,\label{eq2.2}\\
&(\nu\nabla\textbf{u}, \nabla \textbf{v}) + (\eta \textbf{u},\textbf{v}) = -\big(\frac{\nabla a}{2}\varphi(\varphi_1+\varphi_2),\textbf{v}\big)-((J*\varphi_1)\nabla\varphi, \textbf{v})+ ((J*\varphi)\nabla\varphi_2, \textbf{v})+\langle\textbf{h},\textbf{v}\rangle.\label{eq2.3}
\end{align}
 Substitute for $\psi$ in \eqref{eq2.2}, $\psi = (\widetilde{B}(\varphi_1)-\widetilde{B}(\varphi_2))'= (m(\varphi_1)a + \lambda(\varphi_1)) \varphi'+ (m(\varphi_1)-m(\varphi_2))a\varphi_2'+ (\lambda(\varphi_1)-\lambda(\varphi_2))\varphi_2'$.
\begin{align}
&(\varphi', (m(\varphi_1)a + \lambda(\varphi_1))\varphi')+ (\varphi', (m(\varphi_1)-m(\varphi_2))a\varphi_2') + (\varphi', (\lambda(\varphi_1)-\lambda(\varphi_2))\varphi_2') + \frac{1}{2}\frac{d}{dt}\lVert\nabla(\widetilde{B}(\varphi_1)-\widetilde{B}(\varphi_2))\rVert^2  \nonumber\\ &+ \frac{d}{dt}\big(m(\varphi_1)((\nabla a)\varphi-\nabla J*\varphi),\nabla (\widetilde{B}(\varphi_1)-\widetilde{B}(\varphi_2))\big)  - \big(m'(\varphi_1)\varphi_1'((\nabla a)\varphi-\nabla J*\varphi),\nabla (\widetilde{B}(\varphi_1)-\widetilde{B}(\varphi_2))\big) \nonumber\\ & - \big(m(\varphi_1)(\varphi'\nabla a-\nabla J*\varphi'),\nabla (\widetilde{B}(\varphi_1)-\widetilde{B}(\varphi_2))\big)
+\frac{d}{dt}\big((m(\varphi_1)-m(\varphi_2))(\varphi_2\nabla a-\nabla J*\varphi_2), \nabla(\widetilde{B}(\varphi_1)-\widetilde{B}(\varphi_2))\big) \nonumber\\
& -\big((m(\varphi_1)-m(\varphi_2))(\varphi_2'\nabla a-\nabla J*\varphi_2'), \nabla(\widetilde{B}(\varphi_1)-\widetilde{B}(\varphi_2))\big) -\big(m'(\varphi_1)\varphi'((\nabla a)\varphi_2-\nabla J*\varphi_2), \nabla(\widetilde{B}(\varphi_1)-\widetilde{B}(\varphi_2))\big) \nonumber
\end{align}
\begin{align}\label{eq3.53}
& -\big((m'(\varphi_1)-m'(\varphi_2))\varphi_2'(\varphi_2\nabla a-\nabla J*\varphi_2), \nabla(\widetilde{B}(\varphi_1)-\widetilde{B}(\varphi_2))\big) + \frac{d}{dt}\big((\nabla a)(b(\varphi_2)-b(\varphi_1)) , \nabla(\widetilde{B}(\varphi_1)-\widetilde{B}(\varphi_2))\big) \nonumber\\ &+ \big((\nabla a)m(\varphi_1)\varphi' , \nabla(\widetilde{B}(\varphi_1)-\widetilde{B}(\varphi_2))\big) + \big((\nabla a)(m(\varphi_1)-m(\varphi_2))\varphi_2' , \nabla(\widetilde{B}(\varphi_1)-\widetilde{B}(\varphi_2))\big) \nonumber\\ &
  + \big(\textbf{u}_1\nabla\varphi, (\widetilde{B}(\varphi_1)-\widetilde{B}(\varphi_2))'\big)+ \big(\textbf{u}\nabla\varphi_2, (\widetilde{B}(\varphi_1)-\widetilde{B}(\varphi_2))'\big) = 0 .
\end{align}
We define, 
\begin{align*}
    \Psi &= \lVert\nabla(\widetilde{B}(\varphi_1)-\widetilde{B}(\varphi_2))\rVert^2 + 2\big(m(\varphi_1)((\nabla a)\varphi-\nabla J*\varphi),\nabla (\widetilde{B}(\varphi_1)-\widetilde{B}(\varphi_2))\big) \\
    &+ 2\big((m(\varphi_1)-m(\varphi_2))((\nabla a)\varphi_2-\nabla J*\varphi_2), \nabla(\widetilde{B}(\varphi_1)-\widetilde{B}(\varphi_2))\big) + 2\big((\nabla a)(b(\varphi_2)-b(\varphi_1)) , \nabla(\widetilde{B}(\varphi_1)-\widetilde{B}(\varphi_2))\big).
\end{align*}
Using the above notation \eqref{eq3.53} reduces to the following form.
\begin{align}\label{eq3.54}
&\frac{1}{2}\frac{d}{dt}\Psi + (\varphi', (m(\varphi_1)a + \lambda(\varphi_1))\varphi')+ (\varphi', (m(\varphi_1)-m(\varphi_2))a\varphi_2') + (\varphi', (\lambda(\varphi_1)-\lambda(\varphi_2))\varphi_2')   \nonumber\\ & - \big(m'(\varphi_1)\varphi_1'((\nabla a)\varphi-\nabla J*\varphi),\nabla (\widetilde{B}(\varphi_1)-\widetilde{B}(\varphi_2))\big) - \big(m(\varphi_1)(\varphi'\nabla a-\nabla J*\varphi'),\nabla (\widetilde{B}(\varphi_1)-\widetilde{B}(\varphi_2))\big)
\nonumber\\ & -\big((m(\varphi_1)-m(\varphi_2))(\varphi_2'\nabla a-\nabla J*\varphi_2'), \nabla(\widetilde{B}(\varphi_1)-\widetilde{B}(\varphi_2))\big) -\big(m'(\varphi_1)\varphi'((\nabla a)\varphi_2-\nabla J*\varphi_2), \nabla(\widetilde{B}(\varphi_1)-\widetilde{B}(\varphi_2))\big) \nonumber\\ & -\big((m'(\varphi_1)-m'(\varphi_2))\varphi_2'(\varphi_2\nabla a-\nabla J*\varphi_2), \nabla(\widetilde{B}(\varphi_1)-\widetilde{B}(\varphi_2))\big) + \big((\nabla a)m(\varphi_1)\varphi' , \nabla(\widetilde{B}(\varphi_1)-\widetilde{B}(\varphi_2))\big) \nonumber\\&+ \big((\nabla a)(m(\varphi_1)-m(\varphi_2))\varphi_2' , \nabla(\widetilde{B}(\varphi_1)-\widetilde{B}(\varphi_2))\big) + \big(\textbf{u}_1\nabla\varphi, (\widetilde{B}(\varphi_1)-\widetilde{B}(\varphi_2))'\big) \nonumber\\ &+ \big(\textbf{u}\nabla\varphi_2, (\widetilde{B}(\varphi_1)-\widetilde{B}(\varphi_2))'\big) = 0.
\end{align}
Let $I_1, I_2,..., I_{13}$ represent the terms on the left-hand side  of \eqref{eq3.54}, respectively. We estimate each of these terms by applying H\"olders, Gagliardo-Nirenberg and Young's inequalities. From now on we assume c is a generic constant that varies from step to step.  
\begin{align*}
|I_2| &\geq \alpha_1\lVert \varphi'\rVert^2 \hspace{1cm}\text{ by [\textbf{A4}] },\\
|I_3| &\leq c \lVert\varphi'\rVert\lVert\varphi\rVert_{L^4}\lVert\varphi_2'\rVert_{L^4} \leq c \lVert\varphi'\rVert\lVert\varphi\rVert_{V}\lVert\varphi_2'\rVert_{V}\\
 &\leq \epsilon \lVert\varphi'\rVert^2 + c\lVert\varphi_2'\rVert_{V}^2\lVert\varphi\rVert_{V}^2,\\
 |I_4| &\leq c \lVert\varphi'\rVert\lVert\varphi\rVert_{L^4}\lVert\varphi_2'\rVert_{L^4} \leq c \lVert\varphi'\rVert\lVert\varphi\rVert_{V}\lVert\varphi_2'\rVert_{V}\\
 &\leq \epsilon \lVert\varphi'\rVert^2 + c\lVert\varphi_2'\rVert_{V}^2\lVert\varphi\rVert_{V}^2,\\
 |I_5| &\leq c \lVert\varphi_1'\rVert_{L^4}\lVert\varphi\rVert_{L^4}\lVert\nabla (\widetilde{B}(\varphi_1)-\widetilde{B}(\varphi_2))\rVert
 \leq c \lVert\varphi_1'\rVert_{V}\lVert\varphi\rVert_{V}\lVert\nabla (\widetilde{B}(\varphi_1)-\widetilde{B}(\varphi_2))\rVert\\
 &\leq c\lVert\varphi_1'\rVert_{V}^2\lVert\varphi\rVert_{V}^2+ c\lVert\nabla (\widetilde{B}(\varphi_1)-\widetilde{B}(\varphi_2))\rVert^2,\\
 |I_6| &\leq c \lVert\varphi'\rVert\lVert\nabla (\widetilde{B}(\varphi_1)-\widetilde{B}(\varphi_2))\rVert \\ &\leq \epsilon \lVert\varphi'\rVert^2 + c\lVert\nabla (\widetilde{B}(\varphi_1)-\widetilde{B}(\varphi_2))\rVert^2,\\
 |I_7| &\leq c \lVert\varphi_2'\rVert_{L^4}\lVert\varphi\rVert_{L^4}\lVert\nabla (\widetilde{B}(\varphi_1)-\widetilde{B}(\varphi_2))\rVert\leq c \lVert\varphi_2'\rVert_{V}\lVert\varphi\rVert_{V}\lVert\nabla (\widetilde{B}(\varphi_1)-\widetilde{B}(\varphi_2))\rVert\\
 &\leq c\lVert\varphi_2'\rVert_{V}^2\lVert\varphi\rVert_{V}^2+ c\lVert\nabla (\widetilde{B}(\varphi_1)-\widetilde{B}(\varphi_2))\rVert^2,\\
  |I_8| &\leq c \lVert\varphi_2\rVert_{L^\infty}\lVert\varphi'\rVert\lVert\nabla (\widetilde{B}(\varphi_1)-\widetilde{B}(\varphi_2))\rVert\\
  &\leq \epsilon \lVert\varphi'\rVert^2 + c\lVert\nabla (\widetilde{B}(\varphi_1)-\widetilde{B}(\varphi_2))\rVert^2,\\
 |I_9| & \leq c \lVert\varphi_2'\rVert_{L^4}\lVert\varphi\rVert_{L^4}\lVert\nabla (\widetilde{B}(\varphi_1)-\widetilde{B}(\varphi_2))\rVert\\
 &\leq c\lVert\varphi_2'\rVert_{V}^2\lVert\varphi\rVert_{V}^2+ c\lVert\nabla (\widetilde{B}(\varphi_1)-\widetilde{B}(\varphi_2))\rVert^2,\\ 
    |I_{10}| &\leq c\lVert\varphi'\rVert \lVert\nabla(\widetilde{B}(\varphi_1)-\widetilde{B}(\varphi_2))\rVert \\
    &\leq  \epsilon\lVert\varphi'\rVert^2+ c\lVert\nabla(B(\varphi_1)-B(\varphi_2))\rVert^2,\\
    |I_{11}|&\leq c\lVert\varphi\rVert_{L^4}\lVert\varphi_2'\rVert_{L^4}\lVert\nabla(B(\varphi_1)-B(\varphi_2))\rVert\\
    &\leq  c\lVert\varphi_2'\rVert_{V}^2\lVert\varphi\rVert_{V}^2 + c\lVert\nabla(B(\varphi_1)-B(\varphi_2))\rVert^2.
\end{align*}
Moreover, using the expression for $(\widetilde{B}(\varphi_1)-\widetilde{B}(\varphi_2))'$ and H\"older's, Gagliardo-Nirenberg, Sobolev and Young's inequalities at appropriate places we get,
\begin{align*}
    |I_{12}| &= \big|\big(\textbf{u}_1\nabla\varphi, (\widetilde{B}(\varphi_1)-\widetilde{B}(\varphi_2))'\big)\big|\\ 
    &\leq  \big|\big(\textbf{u}_1\nabla\varphi, (m(\varphi_1)a + \lambda(\varphi_1))\varphi'\big)\big| + \big|\big(\textbf{u}_1\nabla\varphi,(\lambda(\varphi_1)-\lambda(\varphi_2))\varphi_2'\big)\big| + \big|\big(\textbf{u}_1\nabla\varphi,(m(\varphi_1)-m(\varphi_2))a\varphi_2'\big)\big|\\
    &\leq c\lVert\textbf{u}_1\rVert_{L^\infty} \lVert\nabla\varphi\rVert \lVert\varphi'\rVert + c\lVert\textbf{u}_1\rVert_{L^\infty} \lVert\nabla\varphi\rVert \lVert\varphi\rVert_{L^4} \lVert\varphi_2'\rVert_{L^4}\\
    &\leq \epsilon \lVert\varphi'\rVert^2+ c\big(\lVert\textbf{u}_1\rVert_{L^\infty}^2+ \lVert\textbf{u}_1\rVert_{L^\infty}\lVert\varphi_2'\rVert_{V} \big) \lVert\varphi\rVert_{V}^2\\
    &\leq \epsilon \lVert\varphi'\rVert^2+ c\big(1+\lVert\varphi_2'\rVert_{V} \big)\lVert\varphi\rVert_{V}^2,\\
    |I_{13}| &= \big(\textbf{u}\nabla\varphi_2, (\widetilde{B}(\varphi_1)-\widetilde{B}(\varphi_2))'\big)\\ 
    &\leq \big|\big(\textbf{u}\nabla\varphi_2, (m(\varphi_1)a + \lambda(\varphi_1))\varphi'\big)\big| + \big|\big(\textbf{u}\nabla\varphi_2,(\lambda(\varphi_1)-\lambda(\varphi_2))\varphi_2'\big)\big| + \big|\big(\textbf{u}\nabla\varphi_2,(m(\varphi_1)-m(\varphi_2))a\varphi_2'\big)\big|\\
    &\leq c\lVert\textbf{u}\rVert_{L^4} \lVert\nabla\varphi_2\rVert_{L^4} \lVert\varphi'\rVert + c\lVert\textbf{u}\rVert_{L^4} \lVert\nabla\varphi_2\rVert_{L^4} \lVert\varphi\rVert_{L^4} \lVert\varphi_2'\rVert_{L^4}\\
    &\leq \epsilon \lVert\varphi'\rVert^2+ c\lVert\nabla\varphi_2\rVert \lVert\varphi_2\rVert_{H^2}\lVert\textbf{u}\rVert \lVert\nabla\textbf{u}\rVert+ c\lVert\nabla\textbf{u}\rVert \lVert\nabla\varphi_2\rVert^{\frac{1}{2}} \lVert\varphi_2\rVert_{H^2}^{\frac{1}{2}}\lVert\varphi\rVert_{V} \lVert\varphi_2'\rVert_{V}\\
    &\leq \epsilon \lVert\varphi'\rVert^2+ 2\epsilon_0 \lVert\nabla\textbf{u}\rVert^2+ 
    c\lVert\nabla\varphi_2\rVert^2 \lVert\varphi_2\rVert_{H^2}^2\lVert\textbf{u}\rVert^2  + c\lVert\varphi_2\rVert_{H^2}^2 \lVert\varphi_2'\rVert_{V}^2\lVert\varphi\rVert_{V}^2. 
\end{align*}
Combining all the above estimates, and for a choice of $\epsilon = \frac{\alpha_1}{14}$ and $\epsilon_0=\frac{1}{4} $, we obtain,
\begin{align}\label{eq2.5}
  \frac{1}{2}\frac{d}{dt}\Psi + \frac{\alpha_1}{2}\lVert\varphi'\rVert^2 &\leq   \frac{1}{2}\lVert\nabla\textbf{u}\rVert^2  + c\lVert\nabla\varphi_2\rVert^2\lVert\varphi_2\rVert_{H^2}^2\lVert \textbf{u}\rVert^2 \nonumber
  + C\lVert\nabla(\widetilde{B}(\varphi_1)-\widetilde{B}(\varphi_2))\rVert^2   \\ &+ C\big( 1+ \lVert\varphi_2'\rVert_{V}^2+\lVert\varphi_1'\rVert_{V}^2+ \lVert\varphi_2'\rVert_{V}^2\lVert\varphi_2\rVert_{H^2}^2 \big) \lVert\varphi\rVert_{V}^2.
\end{align} 
Estimate norm of $\nabla(\widetilde{B}(\varphi_1)-\widetilde{B}(\varphi_2)) =(m(\varphi_1)a +\lambda(\varphi_1))\nabla\varphi + (m(\varphi_1)-m(\varphi_2))a\nabla\varphi_2+ (\lambda(\varphi_1)-\lambda(\varphi_2))\nabla\varphi_2$ in $H$ we get,\vspace{-.25cm}
\begin{align}
    \lVert\nabla(\widetilde{B}(\varphi_1)-\widetilde{B}(\varphi_2))\rVert^2 &\leq C\lVert\varphi\rVert_{V}^2,\label{eq3.057}\\
    \lVert\nabla(\widetilde{B}(\varphi_1)-\widetilde{B}(\varphi_2))\rVert^2 &\geq \alpha_1^2\lVert\nabla\varphi\rVert^2-c\lVert\nabla\varphi\rVert\lVert\varphi\rVert_{L^4}\lVert\nabla\varphi_2\rVert_{L^4}\nonumber\\
    &\geq \alpha_1^2\lVert\nabla\varphi\rVert^2- c\lVert\nabla\varphi\rVert\lVert\varphi\rVert^\frac{1}{2}\lVert\nabla\varphi\rVert^\frac{1}{2}\nonumber\\
    & \geq \frac{\alpha_1^2}{2}\lVert\nabla\varphi\rVert^2-c\lVert\varphi\rVert^2.\label{eq3.058}
    \end{align}
Consider the expression for $\Psi$. By a direct calculation and comparing that with the estimate derived for $\lVert\nabla(\widetilde{B}(\varphi_1)-\widetilde{B}(\varphi_2))\rVert$ we obtain,\vspace{-.25cm}
\begin{equation}\label{eq3.06}
     \frac{\alpha_1^2}{4}\lVert\nabla\varphi\rVert^2-c\lVert\varphi\rVert^2 \leq |\Psi| \leq C\lVert\varphi\rVert_{V}^2.
\end{equation}
For velocity equation  \eqref{eq2.3}, choosing test function as  $\textbf{u}$ and  applying H\"older's, Young's and Korn's inequalities we get,
\begin{equation}\label{eq2.11}
   \lVert\textbf{u}\rVert_{\mathbb{V}_{div}}^2 \leq c(\lVert\varphi\rVert^2+ \lVert\nabla\varphi\rVert^2+ \lVert \textbf{h}\rVert_{\mathbb{V}'_{div}}^2).
\end{equation}
We will multiply \eqref{eq2.11} with an appropriate large constant and add  it to \eqref{eq2.5} so as to absorb the coefficients of ${\lVert\textbf{u}\rVert}^2 $ and ${\lVert \nabla \textbf{u}\rVert}^2 $  terms on the right hand side  of \eqref{eq2.5}. For example we can multiply  \eqref{eq2.5} by $c_0, \text{ such that }  c_0 \geq (1+2c\lVert\nabla\varphi_2\rVert^2\lVert\varphi_2\rVert_{H^2}^2)$.  
Further using \eqref{eq3.057} and \eqref{eq3.06}, we get,
\begin{align}\label{eq2.12}
  \frac{1}{2}\frac{d}{dt}\Psi + \frac{\alpha_1}{2}\lVert\varphi'\rVert^2 + \frac{c_0}{2}\lVert\textbf{u}\rVert_{\mathbb{V}_{div}}^2 &\leq  C\lVert \textbf{h}\rVert_{\mathbb{V}'_{div}}^2 + C\Big( 1+ \lVert\varphi_2'\rVert_{V}^2+\lVert\varphi_1'\rVert_{V}^2+ \lVert\varphi_2'\rVert_{V}^2\lVert\varphi_2\rVert_{H^2}^2 \Big) \lVert\varphi\rVert_{V}^2 \nonumber\\
  &\leq C\lVert \textbf{h}\rVert_{\mathbb{V}'_{div}}^2 + C\Lambda(\Psi +\lVert\varphi\rVert^2), 
\end{align}
Where $\Lambda = 1+ \lVert\varphi_2'\rVert_{V}^2+\lVert\varphi_1'\rVert_{V}^2+ \lVert\varphi_2'\rVert_{V}^2\lVert\varphi_2\rVert_{H^2}^2 $. 

We multiply the first equation in the weak formulation of  the difference equation \eqref{eq3.052} by $ \varphi$. Using assumption [A4], Holder's and Young's inequalities we arrive at 
\begin{align}\label{eq3.591}
    \frac{1}{2}\frac{d}{dt}\lVert\varphi\rVert^2  + \frac{\alpha_1}{2} \lVert \nabla \varphi\rVert^2 \leq \epsilon \lVert{\textbf{u}}\rVert^2 + c\lVert\varphi\rVert^2
    \end{align}
and using \eqref{eq2.11}\vspace{-.5cm}
\begin{align}\label{eq3.59}
    \frac{1}{2}\frac{d}{dt}\lVert\varphi\rVert^2 + \frac{\alpha_1}{2} \lVert \nabla \varphi\rVert^2   \leq C(\lVert\varphi\rVert^2 + \lVert\textbf{h}\rVert_{\mathbb{V}_{div}'}^2).
\end{align}
Add \eqref{eq3.59} to \eqref{eq2.12} and integrate over $[0,T]$. Applying Gronwall's inequality and \eqref{eq10} results in the following inequality
\begin{align}\label{eq2.13}
  \frac{1}{2}\Psi(t)+ \frac{1}{2}\lVert\varphi(t)\rVert^2 + \frac{\alpha_1}{2}\lVert\varphi'\rVert_{L^2(0,T;H)}^2 +  \frac{c_0}{2}\lVert \textbf{u}\rVert_{L^2(0,T;\mathbb{V}_{div})}^2&\leq  C\big(1+ \gamma(t)\big) \big(|\varphi_0| + \lVert \textbf{h}\rVert_{L^2(0,T;\mathbb{V}'_{div})}^2\big). 
\end{align} 
holds for almost all $t \in [0,T]$, for some $ C>0 $ and $\gamma(t)$ a function of time. 
Using \eqref{eq3.06} we have,
\begin{align*}
    \lVert\nabla\varphi(t)\rVert^2 \leq c(|\psi(t)|+\lVert\varphi(t)\rVert^2) &\leq  C\big(1+ \gamma(t)\big) \big(|\varphi_0| + \lVert \textbf{h}\rVert_{L^2(0,T;\mathbb{V}'_{div})}^2\big), 
\end{align*}
holds for almost all $t\in [0,T]$.
 In order to prove \eqref{eq3.1} we need to estimate $,\lVert\varphi\rVert_{(L^2 (0, T;H^2)}$. For, first we estimate $\lVert\widetilde{B}(\varphi_1)-\widetilde{B}(\varphi_2)\rVert_{H^2}$ and bound $\lVert\varphi\rVert_{H^2} $   estimate by it.
\\\\
(\textbf{Estimate for $\lVert\widetilde{B}(\varphi_1)-\widetilde{B}(\varphi_2)\rVert_{H^2}$})
Recall  as in \eqref{eq028} we have,
\begin{equation}\label{eq0.28}
    \lVert\widetilde{B}(\varphi_1)-\widetilde{B}(\varphi_2)\rVert_{H^2} \leq C(\lVert\Delta(\widetilde{B}(\varphi_1)-\widetilde{B}(\varphi_2))\rVert + \lVert\widetilde{B}(\varphi_1)-\widetilde{B}(\varphi_2)\rVert_{V} + \lVert\nabla(\widetilde{B}(\varphi_1)-\widetilde{B}(\varphi_2)).\textbf{n}\rVert_{H^\frac{1}{2}(\partial\Omega)} ).
\end{equation}
We estimate the terms on R.H.S one by one. Observe,
\begin{align*}
    \Delta(\widetilde{B}(\varphi_1)-\widetilde{B}(\varphi_2)) &= \varphi'+ \textbf{u}_1\nabla\varphi+ \textbf{u}\nabla\varphi_2 - \nabla\cdot\big(m(\varphi_1)((\nabla a)\varphi -\nabla J*\varphi)\big) +\nabla\cdot\big( (\nabla a)(b(\varphi_1)-b(\varphi_2))\big) \\& \hspace{.5cm}- \nabla\cdot\big((m(\varphi_1)-m(\varphi_2))((\nabla a)\varphi_2 -\nabla J*\varphi_2)\big)\\
     &= \varphi'+ \textbf{u}_1\nabla\varphi+ \textbf{u}\nabla\varphi_2 - m'(\varphi_1)\nabla\varphi_1((\nabla a)\varphi -\nabla J*\varphi) - m'(\varphi_1)\nabla\varphi((\nabla a)\varphi_2 -\nabla J*\varphi_2)
     \\& \hspace{.5cm}  - m(\varphi_1)((\Delta a)\varphi+ \nabla a\nabla\varphi  -\Delta J*\varphi) - (m'(\varphi_1)-m'(\varphi_2))\nabla\varphi_2((\nabla a)\varphi_2 -\nabla J*\varphi_2) \\ &\hspace{.5cm}- (m(\varphi_1)-m(\varphi_2))((\Delta a)\varphi_2+\nabla a\nabla\varphi_2 -\Delta J*\varphi_2) + (\nabla a)(m(\varphi_1)-m(\varphi_2))\nabla\varphi_2
     \\ &\hspace{.5cm} +(\Delta a)(b(\varphi_1)-b(\varphi_2))+ (\nabla a)m(\varphi_1)\nabla\varphi .
\end{align*}
As in \eqref{eq2.2}, testing $\Delta(\widetilde{B}(\varphi_1)-\widetilde{B}(\varphi_2)) $ with $\psi$ and using H\"olders, Gagliardo-Nirenberg and Young's inequalities and uniform estimates derived for the strong solution \eqref{eq093}, \eqref{eq2.26}, \eqref{eq2.39}, we have
\begin{align*}
    \lVert \Delta(\widetilde{B}(\varphi_1)-\widetilde{B}(\varphi_2))\rVert &\leq c\big(\lVert\varphi'\rVert + \lVert\textbf{u}_1\rVert_{L^\infty}\lVert\nabla\varphi\rVert + \lVert\textbf{u}\rVert_{L^4}\lVert\nabla\varphi_2\rVert_{L^4} + \lVert\nabla\varphi_1\rVert_{L^4}\lVert\varphi\rVert_{L^4} \\& \hspace{.5cm}+ \lVert\varphi\rVert_{V} + \lVert\nabla\varphi\rVert\lVert\varphi_2\rVert_{L^\infty} + \lVert\varphi\rVert_{L^4}\lVert\nabla\varphi_2\rVert_{L^4}\lVert\varphi_2\rVert_{L^\infty} \\ &\hspace{.5cm} +\lVert\varphi\rVert\lVert\varphi_2\rVert_{L^\infty} + \lVert\varphi\rVert_{L^4}\lVert\nabla\varphi_2\rVert_{L^4} + \lVert b(\varphi_1)-b(\varphi_2)\rVert\big) \\
     &\leq c\big(\lVert\varphi'\rVert + \lVert\textbf{u}_1\rVert_{L^\infty}\lVert\nabla\varphi\rVert + \lVert\nabla\textbf{u}\rVert\lVert\nabla\varphi_2\rVert^{\frac{1}{2}}\lVert\varphi_2\rVert_{H^2}^{\frac{1}{2}} + \lVert\varphi\rVert_{V} \\& \hspace{.5cm}+ \lVert\nabla\varphi_1\rVert^{\frac{1}{2}}\lVert\varphi_1\rVert_{H^2}^{\frac{1}{2}}\lVert\varphi\rVert_{V} 
      +  \lVert\nabla\varphi_2\rVert^{\frac{1}{2}}\lVert\varphi_2\rVert_{H^2}^{\frac{1}{2}}\lVert\varphi\rVert_{V}\big)  \\
     &\leq c\big( 1 +\lVert\textbf{u}_1\rVert_{L^\infty} + \lVert\nabla\varphi_1\rVert^{\frac{1}{2}}\lVert\varphi_1\rVert_{H^2}^{\frac{1}{2}} + \lVert\nabla\varphi_2\rVert^{\frac{1}{2}}\lVert\varphi_2\rVert_{H^2}^{\frac{1}{2}}\big) \lVert\varphi\rVert_{V}\\ &\hspace{.5cm} + c(\lVert\varphi'\rVert + \lVert\nabla\varphi_2\rVert^{\frac{1}{2}}\lVert\varphi_2\rVert_{H^2}^{\frac{1}{2}}\lVert\nabla\textbf{u}\rVert) .
\end{align*}
For the boundary term recall our boundary condition namely, 
\begin{align*}
   [m(\varphi_1)\nabla\mu_1 - m(\varphi_2)\nabla\mu_2]\cdot\textbf{n} = 0 &= [ \nabla(\widetilde{B}(\varphi_1)-\widetilde{B}(\varphi_2)) - (\nabla a)(b(\varphi_1)-b(\varphi_2)) + m(\varphi_1)\big((\nabla a)\varphi-\nabla J*\varphi\big) \\ &\hspace{.5cm} +(m(\varphi_1)-m(\varphi_2)\big((\nabla a)\varphi_2-\nabla J*\varphi_2\big)]\cdot\textbf{n} 
\end{align*}
Using the continuous embedding $V(\Omega)\subset H^\frac{1}{2}(\partial\Omega)$, regularity of the strong solution, $\varphi_i\in L^\infty(\Omega\times [0,T])$ and H\"older's inequality we get,
\begin{align*}
    \lVert\nabla(\widetilde{B}(\varphi_1)-\widetilde{B}(\varphi_2)).\textbf{n}\rVert_{H^\frac{1}{2}(\partial\Omega)} &\leq \lVert(\nabla a)(b(\varphi_1)-b(\varphi_2))\cdot\textbf{n}\rVert_{H^\frac{1}{2}(\partial\Omega)} + \lVert m(\varphi_1)\big((\nabla a)\varphi-\nabla J*\varphi\big)\cdot\textbf{n}\rVert_{H^\frac{1}{2}(\partial\Omega)} \\
    &\hspace{.5cm}    \lVert (m(\varphi_1)-m(\varphi_2)\big((\nabla a)\varphi_2-\nabla J*\varphi_2\big)\cdot\textbf{n}\rVert_{H^\frac{1}{2}(\partial\Omega)}  \\
    &\leq \lVert(\nabla a)(b(\varphi_1)-b(\varphi_2))\rVert_{V} + \lVert m(\varphi_1)\big((\nabla a)\varphi-\nabla J*\varphi\big)\rVert_{V} \\ &\hspace{.5cm} + \lVert (m(\varphi_1)-m(\varphi_2)\big((\nabla a)\varphi_2-\nabla J*\varphi_2\big)\rVert_{V}  \\ 
    &\leq c\lVert\varphi\rVert_{V}.
\end{align*}
Substituting back in \eqref{eq0.28}, we get,
\begin{align}\label{eq2.6}
     \lVert \widetilde{B}(\varphi_1)-\widetilde{B}(\varphi_2)\rVert_{H^2} &\leq c\lVert\varphi'\rVert + c\lVert\nabla\varphi_2\rVert^{\frac{1}{2}}\lVert\varphi_2\rVert_{H^2}^{\frac{1}{2}}\lVert\nabla\textbf{u}\rVert  \nonumber
     \\& \hspace{.5cm} + c\Big( 1 +\lVert\textbf{u}_1\rVert_{L^\infty} + \lVert\nabla\varphi_1\rVert^{\frac{1}{2}}\lVert\varphi_1\rVert_{H^2}^{\frac{1}{2}} + \lVert\nabla\varphi_2\rVert^{\frac{1}{2}}\lVert\varphi_2\rVert_{H^2}^{\frac{1}{2}} \Big) \lVert\varphi\rVert_{V}.
\end{align}
Next we will estimate $\lVert\varphi\rVert_{H^2}$ in terms of $\lVert \widetilde{B}(\varphi_1)-\widetilde{B}(\varphi_2)\rVert_{H^2}$.
\\\\
(\textbf{Estimate for $\lVert\varphi\rVert_{H^2}$})
\\Since $\int_\Omega\varphi =0$, using Poincar\'e inequality we have,
$$ \lVert\varphi\rVert_{H^2} \leq c(\lVert D^2\varphi\rVert+ \lVert \varphi\rVert).$$
The first order partial derivative, $\partial_i\widetilde{B}(\varphi_k)=(m(\varphi_k)a+\lambda(\varphi_k))\partial_i\varphi_k+(\partial_ia)b(\varphi_k)$, for $k=1,2$ and $ i \in \{1,2\}$. 
We have the second order derivative, for $i,j \in \{1,2\}$,
\begin{align*}
    \partial_{ij}\widetilde{B}(\varphi_k) = (m(\varphi_k)a+\lambda(\varphi_k))\partial_{ij}\varphi_k + \partial_j(m(\varphi_k)a+\lambda(\varphi_k))\partial_i\varphi_k + \partial_j\big((\partial_ia)b(\varphi_k)\big).
\end{align*}
By assumption $[\textbf{A4}]$, we have the term $(m(\varphi_k)a+\lambda(\varphi_k))$ is bounded away from $0$. Therefore,
\begin{align*}
    \partial_{ij}\varphi_k &= \frac{\Big(\partial_{ij}\widetilde{B}(\varphi_k)- \partial_j\big(b(\varphi_k)\partial_ia\big)\Big)}{\big(m(\varphi_k)a + \lambda(\varphi_k)\big)} - \frac{\partial_j(m(\varphi_k)a+\lambda(\varphi_k))\partial_i\varphi_k }{\big(m(\varphi_k)a + \lambda(\varphi_k)\big)}
    \\&=\frac{\Big(\partial_{ij}\widetilde{B}(\varphi_k)- \partial_j\big(b(\varphi_k)\partial_ia\big)\Big)}{\big(m(\varphi_k)a + \lambda(\varphi_k)\big)} - \frac{\partial_j(m(\varphi_k)a+\lambda(\varphi_k))\big( \partial_i\widetilde{B}(\varphi_k)- (\partial_ia)b(\varphi_k)\big)}{\big(m(\varphi_k)a + \lambda(\varphi_k)\big)^2}.
\end{align*}
The second step is obtained by substituting for $\partial_i\varphi_k = \frac{\big( \partial_i\widetilde{B}(\varphi_k)- (\partial_ia)b(\varphi_k)\big) }{\big(m(\varphi_k)a + \lambda(\varphi_k)\big)}$. Then for $\varphi = \varphi_1 - \varphi_2$ we have,
\begin{align*}
    \partial_{ij}\varphi&=  \frac{\Big(\partial_{ij}\widetilde{B}(\varphi_1)- \partial_j\big(b(\varphi_1)\partial_ia\big)\Big)}{\big(m(\varphi_1)a + \lambda(\varphi_1)\big)} -\frac{\Big(\partial_{ij}\widetilde{B}(\varphi_2)- \partial_j\big(b(\varphi_2)\partial_ia\big)\Big)}{\big(m(\varphi_2)a + \lambda(\varphi_2)\big)}
    \\&\hspace{.5cm} - \frac{\partial_j(m(\varphi_1)a+\lambda(\varphi_1))\big( \partial_i\widetilde{B}(\varphi_1)- (\partial_ia)b(\varphi_1)\big)}{\big(m(\varphi_1)a + \lambda(\varphi_1)\big)^2}  + \frac{\partial_j(m(\varphi_2)a+\lambda(\varphi_2))\big( \partial_i\widetilde{B}(\varphi_2)- (\partial_ia)b(\varphi_2)\big)}{\big(m(\varphi_2)a + \lambda(\varphi_2)\big)^2}
    \\&=  \frac{\partial_{ij}(\widetilde{B}(\varphi_1)-\widetilde{B}(\varphi_2))}{\big(m(\varphi_1)a + \lambda(\varphi_1)\big)}- \frac{\partial_j(b(\varphi_1)-b(\varphi_2))\partial_ia}{\big(m(\varphi_1)a + \lambda(\varphi_1)\big)} - \frac{(b(\varphi_1)-b(\varphi_2))\partial_{ij}a}{\big(m(\varphi_1)a + \lambda(\varphi_1)\big)}\\ 
    &\hspace{.5cm} \Big(\frac{1}{\big(m(\varphi_1)a + \lambda(\varphi_1)\big)}-\frac{1}{\big(m(\varphi_2)a + \lambda(\varphi_2)\big)} \Big)\Big(\partial_{ij}\widetilde{B}(\varphi_2)- \partial_j\big(b(\varphi_2)\partial_ia\big)\Big)\\
    &\hspace{.5cm} - \frac{\big( \partial_i\widetilde{B}(\varphi_1)- (\partial_i a)b(\varphi_1)\big)\big( \partial_j(a(m(\varphi_1)-m(\varphi_2))+(\lambda(\varphi_1)-\lambda(\varphi_2)))\big)}{\big(m(\varphi_1)a + \lambda(\varphi_1)\big)^2}\\
    &\hspace{.5cm} - \frac{\big( \partial_j(m(\varphi_2)a+\lambda(\varphi_2)\big)\big( \partial_i(\widetilde{B}(\varphi_1)-\widetilde{B}(\varphi_2))-(b(\varphi_1)-b(\varphi_2))\partial_i a\big)}{\big(m(\varphi_1)a + \lambda(\varphi_1)\big)^2}\\
    &\hspace{.5cm} - \Big(\frac{1}{\big(m(\varphi_1)a + \lambda(\varphi_1)\big)^2}-\frac{1}{\big(m(\varphi_2)a + \lambda(\varphi_2)\big)^2} \Big)\big( \partial_i\widetilde{B}(\varphi_2)- b(\varphi_2)\partial_i a\big)\big(\partial_j(m(\varphi_2)a+\lambda(\varphi_2))\big).
\end{align*}
Naming each term on R.H.S. $I_1, I_2,...I_7$ in order and we evaluate $H$ norm of each term separately by applying H\"olders, Agmon's, Young's and Gagliardo-Nirenberg inequalities and uniform estimates derived for the strong solution, \eqref{eq2.26}, \eqref{eq2.39}.  
\begin{align*}
   \lVert I_1\rVert &\leq  \frac{1}{\alpha_1}\lVert\partial_{ij}(\widetilde{B}(\varphi_1)-\widetilde{B}(\varphi_2))\rVert \leq \frac{1}{\alpha_1}\lVert \widetilde{B}(\varphi_1)-\widetilde{B}(\varphi_2)\rVert_{H^2},\\ 
   \lVert I_2\rVert &\leq  \frac{1}{\alpha_1}\lVert\partial_ia\rVert_{L^\infty}\lVert\partial_j(b(\varphi_1)-b(\varphi_2))\rVert \leq c\lVert\varphi\rVert_{V},\\
   \lVert I_3\rVert &\leq  \frac{1}{\alpha_1} \lVert\partial_{ij}a\rVert_{L^\infty}\lVert b(\varphi_1)-b(\varphi_2))\rVert \leq c\lVert\varphi\rVert,\\
   \lVert I_4\rVert &\leq c_1\lVert\varphi\rVert_{L^\infty}(\lVert\partial_{ij} \widetilde {B}(\varphi_2)\rVert + \lVert\partial_j b(\varphi_2)\rVert +\lVert b(\varphi_2)\rVert )\leq c\lVert\varphi\rVert^{\frac{1}{2}}\lVert\varphi\rVert_{H^2}^{\frac{1}{2}} \\
   &\leq \epsilon \lVert\varphi\rVert_{H^2} + c \lVert\varphi\rVert, \\
   \lVert I_5\rVert &\leq \frac{1}{\alpha_1^2}(\lVert\partial_i \widetilde{B}(\varphi_1)\rVert_{L^4}+\lVert b(\varphi_1)\rVert_{L^4})\lVert\partial_j\varphi\rVert_{L^4} \leq c\lVert\nabla\varphi\rVert^{\frac{1}{2}}\lVert\varphi\rVert_{H^2}^{\frac{1}{2}}\\
   &\leq \epsilon \lVert\varphi\rVert_{H^2} + c \lVert\varphi\rVert_{V}, \\
   \lVert I_6 \rVert &\leq \lVert\partial_j\varphi_2\rVert_{L^4}(\lVert\partial_i(\widetilde{B}(\varphi_1)-\widetilde{B}(\varphi_2))\rVert_{L^4} +\lVert b(\varphi_1)-b(\varphi_2)\rVert_{L^4})\\
   &\leq c(\lVert \widetilde{B}(\varphi_1)-\widetilde{B}(\varphi_2)\rVert_{H^2} + \lVert\varphi\rVert_{V}),\\
   \lVert I_7 \rVert &\leq \lVert\varphi\rVert_{L^\infty}\lVert\partial_j \varphi_2\rVert_{L^4}(\lVert\partial_i\widetilde{B}(\varphi_2)\rVert_{L^4}+ \lVert b(\varphi_2)\rVert_{L^4})\leq c_3\lVert\varphi\rVert^{\frac{1}{2}}\lVert\varphi\rVert_{H^2}^{\frac{1}{2}} \\
   &\leq \epsilon \lVert\varphi\rVert_{H^2} + c \lVert\varphi\rVert.
\end{align*}
Combining all the above estimates in \eqref{eq2.6}, for a sufficiently small  $\epsilon >0$  we will get,
\begin{align}
    \lVert\varphi\rVert_{H^2} &\leq  C(\lVert \widetilde{B}(\varphi_1)-\widetilde{B}(\varphi_2)\rVert_{H^2} +\lVert\varphi\rVert_{V})\nonumber\\ &\leq C(\lVert\varphi'\rVert + \lVert\nabla\textbf{u}\rVert + \lVert\varphi\rVert_{V}).
\end{align}
This implies,
\begin{equation}\label{eq2.7}
    \lVert\varphi\rVert_{L^2(0,T;H^2)}^2 \leq  C(\lVert\varphi'\rVert_{L^2(0,T;H)}^2 + \lVert\textbf{u}\rVert_{L^2(0,T;\mathbb{V}_{div})}^2 + \lVert\varphi\rVert_{L^2(0,T;V)}^2).
\end{equation}
Combining \eqref{eq2.13} and \eqref{eq2.7}, we get the stability estimate, \eqref{eq3.1}.
\end{proof}

We have proved all the prerequisites to study the optimal control problem of the system \eqref{eq1}-\eqref{eq07} with a singular type potential. In the next section we will state the optimal control problem(OCP) and study the existence of a solution to the OCP.

\section{The Optimal Control Problem}

We define the optimal control problem as to minimise the tracking type cost functional defined by,
\begin{align}\label{eq04.1}
\mathcal{J}(\varphi,\textbf{u},\textbf{U}) \coloneqq \int\limits_{0}^{T} \lVert\varphi(t)-\varphi_{d}(t)\rVert^{2} dt + \int\limits_{0}^{T}\lVert\textbf{u}(t)-\textbf{u}_{d}(t)\rVert^{2} dt  +\int_{\Omega}\lVert\varphi(x,T)-\varphi_{\Omega}\rVert^{2} dx + \int\limits_{0}^{T}\lVert \textbf{U}(t)\rVert^{2} dt
\end{align}
in a bounded, closed and convex set of admissible controls defined by,
$$\mathcal{U}_{ad}=\{\textbf{U}\in \mathcal{U}: U_1(x,t)\leq \textbf{U}(x,t)\leq U_2(x,t), \text{ a.e } (x,t)\in \Omega\times (0,T) \},$$
subject to the system \eqref{eq1}-\eqref{eq07} with potential either $F_{log}$ or $F_{do}$.\vspace{.25cm}
\\Recall that $\mathcal{U}:=\{\textbf{h}\in L^\infty(0,T;\mathbb{G}_{div})|\textbf{h}_t\in L^2(0,T;\mathbb{V}_{div}')\}$. Here $\varphi_d \in L^2((0,T)\times \Omega )$, $\textbf{u}_d\in L^2(0,T; \mathbb{G}_{div})$, $\varphi_{\Omega}\in L^2(\Omega)$ are the desirable states  and the external forcing term $\textbf{U}$ acts as a control.  We denote by $(\varphi,\textbf{u})$  the unique strong solution to the system \eqref{eq1}-\eqref{eq07} corresponding to the control $\textbf{U} \in \mathcal{U}_{ad}$. The controls  $U_1, U_2 \in \mathcal{U} \cap L^\infty((0,T)\times \Omega )$ are fixed. 
\\Consider the control to state map defined from  $\mathcal{S}: \mathcal{U} \rightarrow \mathcal{V}$ with 
\begin{equation}\label{eq2.1}
    \mathcal{V} := \{ (\varphi, \textbf{u})\in L^{\infty}(0,T;H^{2}(\Omega))\times L^{2}(0,T; \textbf{H}^{2}(\Omega)) : \varphi' \in  L^\infty(0,T;H) \cap L^2(0,T;V)\}. 
\end{equation}
\\
In particular, $\mathcal{S}(\textbf{U}):=(\varphi, \textbf{u})$, where $(\varphi, \textbf{u})$ is the unique strong solution to the system \eqref{eq1}-\eqref{eq07} that corresponds to the control $\textbf{U}$. 

 From the continuous dependence estimates \eqref{eq10} and \eqref{eq3.1}, we can conclude that the map $\mathcal{S}$ is Lipschitz.  
Using the definition of $\mathcal{S}$, the optimal control problem can be written as follows:
$$ \underset{\textbf{U}\in \mathcal{U}_{ad}}{\text{min}} \{ \mathcal{J}(\mathcal{S}(\textbf{U}), \textbf{U})\}  \hspace{1cm} (OCP)$$

We will prove the existence of an optimal control to the (OCP) in the following theorem. Later on  we will characterise the optimal control in terms of  the adjoint variables. 
\vspace{.5cm}
\begin{theorem}[The existence of an optimal control]
Let $\varphi_0 \in H^2(\Omega) \cap L^\infty(\Omega)$  and $J \in W^{2,1}(\Omega)$. Assume $[\textbf{J}], [\textbf{A1}]-[\textbf{A4}]$ holds. In addition, assume that the mobility, $m \in C^2[-1,1]$ and $\lambda \in C^2[-1,1]$. Then the Optimal control problem (OCP) admits a solution.
\end{theorem}

\begin{proof}
Let $l=\underset{\textbf{U}\in \mathcal{U}_{ad}}{\text{inf}}\mathcal{J}(\mathcal{S}(\textbf{U}),\textbf{U}).$ Then there exists a minimising sequence $\{\textbf{U}_n\} \subseteq \mathcal{U}_{ad},$ such that 
$$\underset{n\rightarrow\infty}{\text{lim}} \mathcal{J}(\mathcal{S}(\textbf{U}_n),\textbf{U}_n) =l. $$
Since $\textbf{U}_n \in \mathcal{U}_{ad} $,  $\lVert \textbf{U}_n\rVert$ is uniformly bounded in $L^2(0,T;\mathbb{G}_{div}).$  And up to a subsequence, $\textbf{U}_n \rightharpoonup \Bar{\textbf{U}}.$ for some $\Bar{\textbf{U}} \in L^2(0,T;\mathbb{G}_{div}).$ 
 Since $\mathcal{U}_{ad}$ is a closed convex subset of $\mathcal{U}$, $\mathcal{U}_{ad}$ is weakly sequentially closed  
and hence $\Bar{\textbf{U}} \in \mathcal{U}_{ad}.$ Let $\mathcal{S}(\textbf{U}_n)=(\varphi_n, \textbf{u}_n).$ Then we have from regularity of the strong solution $(\varphi_n, \textbf{u}_n)$,
\begin{align*}
    \lVert\varphi_n\rVert_{L^\infty(0,T;H^2)} &\leq C,\\
    \lVert\varphi_n'\rVert_{L^\infty(0,T;H) \cap L^2(0,T;V)} &\leq C,\\
    \lVert\textbf{u}_n\rVert_{L^\infty(0,T;\mathbb{V}_{div})} &\leq C,\\
    \lVert\textbf{u}_n\rVert_{L^2(0,T;\textbf{H}^2)} &\leq C.
\end{align*}
Hence, up to a subsequence we get $(\Bar{\varphi}, \Bar{\textbf{u}})$ such that 
\begin{align*}
    \varphi_n &\overset{\ast}{\rightharpoonup} \Bar{\varphi}  \text{ in  } L^\infty(0,T;H^2),\\
    \varphi_n' &\overset{\ast}{\rightharpoonup} \Bar{\varphi}' \text{  in  } L^\infty(0,T;H),\\
    \varphi'_n &\rightharpoonup \Bar{\varphi}'  \text{ in  } L^2(0,T;V),\\
    \textbf{u}_n &\overset{\ast}{\rightharpoonup} \Bar{\textbf{u}} \text{ in  } L^\infty(0,T;\mathbb{V}_{div}),\\
    \textbf{u}_n &\rightharpoonup \Bar{\textbf{u}} \text{ in  } L^2(0,T;\textbf{H}^2).
\end{align*}

From the above convergences and using the continuous dependence estimates \eqref{eq10} and \eqref{eq3.1} we get, $\mathcal{S}(\Bar{\textbf{U}})= (\Bar{\varphi}, \Bar{\textbf{u}}) \in \mathcal{V}.$ Since the functional $J$ is continuous and convex, we conclude that it is  weakly sequentially lower semi continuous \cite[Proposition 5, Chapter 1]{AUE}. Hence we can write,

$$l \leq \mathcal{J}(\Bar{\varphi},\Bar{\textbf{u}},\Bar{\textbf{U}}) \leq \lim \inf J(\varphi_n,\textbf{u}_n, \textbf{U}_n) = l.$$

Hence $\Bar{\textbf{U}}$ is the optimal control and corresponding optimal state is $\mathcal{S}(\Bar{\textbf{U}}) = (\Bar{\varphi},\Bar{\textbf{u}})$. 
\end{proof}

To derive the necessary optimality conditions, we need to study the differentiability properties of the control to state operator $\mathcal{S}$. Now we will prove $\mathcal{S}$ is Fr\'echet differentiable and the Fr\'echet derivative of $\mathcal{S}$ at $\Bar{\textbf{U}}$ is given by the solution to the linearized system  obtained by linearizing \eqref{eq1}-\eqref{eq07} around $(\Bar{\varphi},\Bar{\textbf{u}})=\mathcal{S}(\Bar{\textbf{U}})$. 

\subsection{Differentiability of the control to state operator}

\textbf{The Linearized system:}
Let $\Bar{\textbf{U}}$ be an optimal control and $(\Bar{\varphi},\Bar{\textbf{u}})$ be the corresponding strong solution to the system \eqref{eq1}-\eqref{eq07}. Let $\textbf{U} \in \mathcal{U}$ be given. Consider the following system in $(\psi, \textbf{w})$ that obtained by linearising the system \eqref{eq1}-\eqref{eq07} around the optimal state $(\Bar{\varphi},\Bar{\textbf{u}})$. 
\begin{align}
    \psi'+ \bar{\textbf{u}}\nabla\psi+ \textbf{w}\nabla\Bar{\varphi} &= \nabla\cdot(m(\Bar{\varphi})(\nabla(a\psi)-\nabla J*\psi)) + \nabla\cdot(m'(\Bar{\varphi})\psi(\nabla(a\Bar{\varphi})-\nabla J*\Bar{\varphi})) \nonumber\\ &\hspace{.5cm}+ \nabla\cdot(\lambda(\Bar{\varphi})\nabla\psi+ \lambda'(\Bar{\varphi})\psi\nabla\Bar{\varphi}),\label{eq3.2}\\
    -\nu\Delta\textbf{w} + \eta\textbf{w} + \nabla \pi &= 
    a\nabla(\Bar{\varphi}\psi) + (\nabla J*\Bar{\varphi})\psi + (\nabla J*\psi)\Bar{\varphi}+\nabla(F(\Bar{\varphi}+\psi)-F(\Bar{\varphi})) + \textbf{U} \label{eq3.3}\\
    \nabla\cdot \textbf{w} &= 0 \text{ in } Q,\label{eq3.4}\\
    \textbf{w} &= 0 \text{ on } \Sigma,\label{eq3.5}\\
    \Big(m(\Bar{\varphi})(\nabla(a\psi)-\nabla J*\psi)&+ m'(\Tilde{\varphi})\psi(\nabla(a\Bar{\varphi})-\nabla J*\Bar{\varphi}) + \lambda(\Bar{\varphi})\nabla\psi + \lambda'(\Tilde{\varphi})\psi\nabla\Bar{\varphi}\Big).n
     = 0 \text{ on } \Sigma,\label{eq3.6}\\
    \psi(x,0) &= \psi_0(x) \text{ in } \Omega.\label{eq3.7}
\end{align}
We will study the linearized system for well-posedness.
\begin{theorem}[\textbf{Existence of a weak solution to the linearized system}]
  Let all the hypotheses of Theorem 4.1 hold. Then for $\textbf{U}\in \mathcal{U}$, there exists a unique weak solution $(\psi, \textbf{w})$  to the system \eqref{eq3.2}-\eqref{eq3.7} such that 
  \begin{align}
      \psi &\in L^\infty(0,T;H) \cap L^2(0,T;V) \cap H^1(0,T;V'),\\
      \textbf{w} &\in L^2(0,T;\mathbb{V}_{div}).
  \end{align}
  \end{theorem}
  \begin{proof}
  Proof is given by Galerkin approximation method. Consider an orthonormal family, $\{\eta_i\} \subseteq V$ of eigenvectors of the Neumann operator $-\Delta+I,$ and a family, $\{\bm{\nu}_i\} \subseteq \mathbb{V}_{div}$ of eigenvectors of the Stokes operator. Let $\Psi_n:=<\eta_1,...,\eta_n>$ and $\mathbb{V}_n:=<\bm{\nu}_1,...,\bm{\nu}_n>.$ $P_n$ and $\mathbb{P}_n$ represents the orthogonal projections from $V, \mathbb{V}_{div}$ onto the spaces $\Psi_n$ and $\mathbb{V}_n$ respectively. Looking for
  $\psi_n=\sum_{i=1}^{n}a_i^n(t)\eta_i,$ $w_n=\sum_{i=1}^{n}b_i^n(t)\bm{\nu}_i$ that solves the finite dimensional problem,

\begin{align}
    \langle\psi_n',\eta_i\rangle+ (\bar{\textbf{u}}\nabla\psi_n,\eta_i)+ (\textbf{w}_n\nabla\Bar{\varphi},\eta_i) &= -(m(\Bar{\varphi})(\nabla(a\psi_n)-\nabla J*\psi_n),\nabla\eta_i) - (m'(\Bar{\varphi})\psi_n(\nabla(a\Bar{\varphi})-\nabla J*\Bar{\varphi}),\nabla\eta_i) \nonumber\\ &\hspace{.5cm} -(\lambda(\Bar{\varphi})\nabla\psi_n,\nabla\eta_i) - (\lambda'(\Bar{\varphi})\psi_n\nabla\Bar{\varphi},\nabla\eta_i),\label{eq3.100}\\
    (\nu\nabla\textbf{w}_n,\nabla\bm{\nu}_i) + (\eta\textbf{w}_n,\bm{\nu}_i) &= 
    -(\nabla a(\Bar{\varphi}\psi_n),\bm{\nu}_i) + ((\nabla J*\Bar{\varphi})\psi_n,\bm{\nu}_i) + ((\nabla J*\psi_n)\Bar{\varphi}, \bm{\nu}_i) + \langle\textbf{U},\bm{\nu}_i\rangle. \label{eq3.110}
\end{align}

where $i=1,2,...,n$ and for a.e $t\in [0,T].$ Using the expression for $\psi_n$ and $\textbf{w}_n$ the above system reduces to a system of $n$ differential equations in $2n$ variables $a_i^n(t)$ and $b_i^n(t)$. Using \eqref{eq3.110}, we can express $b_i^n(t)$ in terms of $a_i^n(t)$ thereby obtaining a system of O.D.E in $n$ variables, $a_i^n$. By applying Carath\'eodory theorem \cite[Theorem 1.1, Chapter 2]{CDT} we will get a unique solution $a_i^n(t)$. Thus we get a solution of \eqref{eq3.100}-\eqref{eq3.110} namely $(\psi_n, \textbf{w}_n)$ in a maximal interval $[0,T^*].$
Multiplying \eqref{eq3.100} with $a_i^n(t)$ and taking summation over $i$ varying from 1 to $n$, we get,
\begin{align}
    \frac{1}{2}\frac{d}{dt}\lVert\psi_n\rVert^2 +((m(\Bar{\varphi})a+\lambda(\Bar{\varphi}))\nabla\psi_n,\nabla\psi_n) &= -(\textbf{w}_n\nabla\Bar{\varphi},\psi_n)-(m(\Bar{\varphi})((\nabla a)\psi_n-\nabla J*\psi_n),\nabla\psi_n) \nonumber\\ &\hspace{.5cm} - (m'(\Bar{\varphi})\psi_n(\nabla(a\Bar{\varphi})-\nabla J*\Bar{\varphi}),\nabla\psi_n) - (\lambda'(\Bar{\varphi})\psi_n\nabla\Bar{\varphi},\nabla\psi_n).
\end{align}
Let $I_1, I_2$ be the terms on L.H.S. and $I_3,...,I_6$ be terms on the R.H.S. Then by applying H\"olders, Ladyzhenskaya, Gagliardo-Nirenberg and Young's inequalities we obtain, 
\begin{align*}
    |I_2| &\geq \alpha_1\lVert\nabla\psi_n\rVert^2,\\
    |I_3| &\leq \lVert\textbf{w}_n\rVert_{L^4}\lVert\nabla\Bar{\varphi}\rVert_{L^4}\lVert\psi_n\rVert \leq \frac{\nu}{4}\lVert\nabla\textbf{w}_n\rVert^2+ c\lVert\nabla\Bar{\varphi}\rVert_{L^4}^2\lVert\psi_n\rVert^2,\\
    |I_4| &\leq |m|_{L^\infty}(\lVert\nabla a\rVert_{L^\infty} + \lVert\nabla J\rVert_{L^1})\lVert\psi_n\rVert\lVert\nabla\psi_n\rVert \leq \frac{\alpha_1}{12}\lVert\nabla\psi_n\rVert^2+ c\lVert\psi_n\rVert^2,\\
    |I_5| &\leq |m'|_{L^\infty} (\lVert\nabla a\rVert_{L^\infty} + \lVert\nabla J\rVert_{L^1})\lVert\Bar{\varphi}\rVert_{L^\infty}\lVert\psi_n\rVert\lVert\nabla\psi_n\rVert+ |m'|_{L^\infty}\lVert a\rVert_{L^\infty}\lVert\nabla\Bar{\varphi}\rVert_{L^4}\lVert\psi_n\rVert_{L^4}\lVert\nabla\psi_n\rVert\\ &\leq \frac{\alpha_1}{8}\lVert\nabla\psi_n\rVert^2 + c\lVert\psi_n\rVert^2 + c \lVert\nabla\Bar{\varphi}\rVert_{L^4}^2\lVert\psi_n\rVert\lVert\nabla\psi_n\rVert \leq \frac{\alpha_1}{4}\lVert\nabla\psi_n\rVert^2 + c(1+ \lVert\nabla\Bar{\varphi}\rVert_{L^4}^4)\lVert\psi_n\rVert^2,\\
    |I_6| &\leq |\lambda'|_{L^\infty}\lVert\nabla\Bar{\varphi}\rVert_{L^4}\lVert\psi_n\rVert_{L^4}\lVert\nabla\psi_n\rVert \leq \frac{\alpha_1}{12}\lVert\nabla\psi_n\rVert^2 + c\lVert\nabla\Bar{\varphi}\rVert_{L^4}^2\lVert\nabla\psi_n\rVert\lVert\psi_n\rVert
    \leq \frac{\alpha_1}{6}\lVert\nabla\psi_n\rVert^2 + c\lVert\nabla\Bar{\varphi}\rVert_{L^4}^4\lVert\psi_n\rVert^2.
\end{align*}
Combining all the above estimates we get, 
\begin{align}\label{eq3.013}
    \frac{1}{2}\frac{d}{dt}\lVert\psi_n\rVert^2 + \frac{\alpha_1}{2}\lVert\nabla\psi_n\rVert^2 &\leq  \frac{\nu}{4}\lVert\nabla\textbf{w}_n\rVert^2 + C(1+ \lVert\nabla\Bar{\varphi}\rVert_{L^4}^2 +\lVert\nabla\Bar{\varphi}\rVert_{L^4}^4)\lVert\psi_n\rVert^2. 
\end{align}
Multiplying \eqref{eq3.110} with $b_i^n(t)$ and taking summation over $i$ we get,
\begin{align}\label{eq3.014}
   \nu\lVert\nabla\textbf{w}_n\rVert^2 + \eta\lVert\textbf{w}_n\rVert^2 &= 
    -(\nabla a(\Bar{\varphi}\psi_n),\textbf{w}_n) + ((\nabla J*\Bar{\varphi})\psi_n,\textbf{w}_n) + ((\nabla J*\psi_n)\Bar{\varphi}, \textbf{w}_n) + \langle\textbf{U},\textbf{w}_n\rangle. 
\end{align}
 Let $J_1,...,J_4$ denote the terms on R.H.S of \eqref{eq3.014}.
\begin{align*}
     |J_1| &\leq \lVert\nabla a \rVert_{L^\infty}\lVert\Bar{\varphi}\rVert_{L^\infty}\lVert\psi_n\rVert\lVert\textbf{w}_n\rVert \leq \frac{\eta}{6}\lVert\textbf{w}_n\rVert^2 + c\lVert\psi_n\rVert^2, \\
     |J_2| &\leq \lVert\nabla J \rVert_{L^1}\lVert\Bar{\varphi}\rVert_{L^\infty}\lVert\psi_n\rVert\lVert\textbf{w}_n\rVert \leq \frac{\eta}{6}\lVert\textbf{w}_n\rVert^2 + c\lVert\psi_n\rVert^2 ,\\
     |J_3| &\leq \lVert\nabla J \rVert_{L^1}\lVert\Bar{\varphi}\rVert_{L^\infty}\lVert\psi_n\rVert\lVert\textbf{w}_n\rVert \leq \frac{\eta}{6}\lVert\textbf{w}_n\rVert^2 + c\lVert\psi_n\rVert^2,\\
     |J_4| &\leq \lVert \textbf{U}\rVert_{\mathbb{V}_{div}'}\lVert\textbf{w}_n\rVert_{\mathbb{V}_{div}} \leq \frac{\nu}{4}\lVert\nabla\textbf{w}_n\rVert^2 + c\lVert \textbf{U}\rVert_{\mathbb{V}_{div}'}^2.
 \end{align*}
Therefore,
\begin{align}\label{eq3.015}
   \nu\lVert\nabla\textbf{w}_n\rVert^2 + \eta\lVert\textbf{w}_n\rVert^2 &\leq \frac{\nu}{4}\lVert\nabla\textbf{w}_n\rVert^2 + \frac{\eta}{2}\lVert\textbf{w}_n\rVert^2 + C(\lVert\psi_n\rVert^2 + \lVert \textbf{U}\rVert_{\mathbb{V}_{div}'}^2).
 \end{align}
Adding \eqref{eq3.013} and \eqref{eq3.015}, we obtain,
\begin{align}\label{eq3.016}
    \frac{1}{2}\frac{d}{dt}\lVert\psi_n\rVert^2 + \frac{\alpha_1}{2}\lVert\nabla\psi_n\rVert^2 +\frac{\nu}{2}\lVert\nabla\textbf{w}_n\rVert^2 + \frac{\eta}{2}\lVert\textbf{w}_n\rVert^2 &\leq C(1+ \lVert\nabla\Bar{\varphi}\rVert_{L^4}^2 +\lVert\nabla\Bar{\varphi}\rVert_{L^4}^4)\lVert\psi_n\rVert^2  + C\lVert \textbf{U}\rVert_{\mathbb{V}_{div}'}^2.
\end{align}
After applying Gronwall's inequality we get, 
\begin{align}\label{eq3.017}
    \frac{1}{2}\lVert\psi_n\rVert^2_{L^\infty(0,T;H)} + \frac{\alpha_1}{2}\lVert\psi_n\rVert_{L^2(0,T,V)}^2 + k_0\lVert\textbf{w}_n\rVert_{L^2(0,T:\mathbb{V}_{div})}^2 &\leq C.
\end{align}
In \eqref{eq3.100} considering test function $\xi\in V,$ we have,
\begin{align}
    <\psi_n',\xi> &= (\bar{\textbf{u}}\psi_n,\nabla\xi)+ (\textbf{w}_n\Bar{\varphi},\nabla\xi) -(m(\Bar{\varphi})(\nabla(a\psi_n)-\nabla J*\psi_n),\nabla\xi) - (m'(\Bar{\varphi})\psi_n(\nabla(a\Bar{\varphi})-\nabla J*\Bar{\varphi}),\nabla\xi) \nonumber\\ &\hspace{.5cm} -(\lambda(\Bar{\varphi})\nabla\psi_n,\nabla\xi) - (\lambda'(\Bar{\varphi})\psi_n\nabla\Bar{\varphi},\nabla\xi),\label{eq3.018}
\end{align}
We denote the terms on the R.H.S of \eqref{eq3.018} by $L_1,L_2,...,L_6.$ Estimating each of its terms, we get,
\begin{align*}
    |L_1| &\leq \lVert\bar{\textbf{u}}\rVert_{L^\infty}\lVert\psi_n\rVert\lVert\nabla\xi\rVert  \leq c\lVert\psi_n\rVert\lVert\xi\rVert_{V},\\
    |L_2| &\leq \lVert\Bar{\varphi}\rVert_{L^\infty}\lVert\textbf{w}_n\rVert\lVert\nabla\xi\rVert  \leq 
    c \lVert\textbf{w}_n\rVert\lVert\xi\rVert_{V},\\
    |L_3| &\leq |m|_{L^\infty}(\lVert\nabla a\rVert_{L^\infty} +\lVert\nabla J\rVert_{L^1})\lVert\psi_n\rVert\lVert\nabla\xi\rVert + |m|_{L^\infty}\lVert a\rVert_{L^\infty}\lVert\nabla\psi_n\rVert\lVert\nabla\xi\rVert \leq c\lVert\psi_n\rVert_{V}\lVert\xi\rVert_{V},\\
    |L_4| &\leq |m'|_{L^\infty}(\lVert\nabla a\rVert_{L^\infty} +\lVert\nabla J\rVert_{L^1})\lVert\Bar{\varphi}\rVert_{L^\infty}\lVert\psi_n\rVert\lVert\nabla\xi\rVert + |m'|_{L^\infty}\lVert a\rVert_{L^\infty} \lVert\nabla\Bar{\varphi}\rVert_{L^4}\lVert\psi_n\rVert_{L^4}\lVert\nabla\xi\rVert \\ &\leq c\lVert\psi_n\rVert_{V}\lVert\xi\rVert_{V},\\
    |L_5| &\leq \lVert\lambda\rVert_{L^\infty}\lVert\nabla\psi_n\rVert\lVert\nabla\xi\rVert \leq c\lVert\psi_n\rVert_{V}\lVert\xi\rVert_{V},\\
    |L_6| &\leq |\lambda'|_{L^\infty}\lVert\nabla\Bar{\varphi}\rVert_{L^4}\lVert\psi_n\rVert_{L^4}\lVert\nabla\xi\rVert \leq c\lVert\psi_n\rVert_{V}\lVert\xi\rVert_{V}.
\end{align*}
Substituting the estimates,we get  
\begin{align}
    |\langle\psi_n',\xi \rangle |  &\leq C\big(\lVert\textbf{w}_n\rVert+ \lVert\psi_n\rVert_{V}\big)\lVert\xi\rVert_{V}, \nonumber\\
  \text{hence,}  \ \|\psi_n'\|_{V'} &\leq C\big(\lVert\textbf{w}_n\rVert+ \lVert\psi_n\rVert_{V}\big) .\label{eq4.19}
\end{align}
Integrating from 0 to $T$ we can conclude that $\lVert\psi_n'\rVert_{L^2(0,T;V')} \leq C.$ From all the derived uniform estimates, we can extract  subsequences such that
\begin{align}
    \psi_n &\overset{\ast}{\rightharpoonup} \psi  \text{  in  } L^\infty(0,T;H),\label{eq4.20}\\
    \psi_n &\rightharpoonup \psi \text{ in  } L^2(0,T;V),\label{eq4.21}\\
    \psi'_n &\rightharpoonup \psi'  \text{ in  } L^2(0,T;V'),\label{eq4.22}\\
    \textbf{w}_n &\rightharpoonup \textbf{w} \text{ in  } L^2(0,T;\mathbb{V}_{div}).\label{eq4.23}
\end{align}
Using Aubin-Lions lemma and \eqref{eq4.21} and \eqref{eq4.22} we can conclude that $\psi \in L^2(0,T;H)$ and $\psi_n \rightarrow \psi$ strongly for a.e $x\in \Omega, t\in [0,T]$. Moreover, from \eqref{eq4.20} and \eqref{eq4.22} we get, $\psi\in C([0,T];H)$. By passing to the limit as $ n \rightarrow \infty$ we arrive at a weak solution $(\psi, \textbf{w})$  of the linearized system. \eqref{eq3.2}-\eqref{eq3.7}. 
\\To prove the uniqueness of the solution, let $(\psi_1,\textbf{w}_1)$, $(\psi_2,\textbf{w}_2)$ be two solutions to the linearized system \eqref{eq3.2}-\eqref{eq3.7} corresponding to initial data $\psi_{10}$ and $\psi_{20}$ respectively. Let $\psi=\psi_1-\psi_2$, $\textbf{w}= \textbf{w}_1-\textbf{w}_2$ and $\psi_0 = \psi_{10}-\psi_{20}$. Then $(\psi,\textbf{w})$ solves,
\begin{align}
    \psi'+ \bar{\textbf{u}}\nabla\psi+ \textbf{w}\nabla\Bar{\varphi} &= \nabla\cdot(m(\Bar{\varphi})(\nabla(a\psi)-\nabla J*\psi)) + \nabla\cdot(m'(\Bar{\varphi})\psi(\nabla(a\Bar{\varphi})-\nabla J*\Bar{\varphi})) \nonumber\\ &\hspace{.5cm}+ \nabla\cdot(\lambda(\Bar{\varphi})\nabla\psi+ \lambda'(\Bar{\varphi})\psi\nabla\Bar{\varphi}),\label{eq3.019}\\
    -\nu\Delta\textbf{w} + \eta\textbf{w} + \nabla \pi &= 
    a\nabla(\Bar{\varphi}\psi) + (\nabla J*\Bar{\varphi})\psi + (\nabla J*\psi)\Bar{\varphi}+\nabla(F(\Bar{\varphi}+\psi)-F(\Bar{\varphi})),\label{eq3.020} \\
    \nabla\cdot \textbf{w} &= 0 \text{ in } Q,\\
    \textbf{w} &= 0 \text{ on } \Sigma,\\
    \big[m(\Bar{\varphi})(\nabla(a\psi)-\nabla J*\psi)&+ m'(\Bar{\varphi})\psi(\nabla(a\Bar{\varphi})-\nabla J*\Bar{\varphi}) + \lambda(\Bar{\varphi})\nabla\psi + \lambda'(\Bar{\varphi})\psi\nabla\Bar{\varphi}\big].\textbf{n}
     = \psi_0 \text{ on } \Sigma,\\
    \psi(x,0) &= 0 \text{ in } \Omega.
\end{align}
Testing \eqref{eq3.019} with $\psi$ and \eqref{eq3.020} with $\textbf{w}$, and adding both equations and estimating each  term like earlier, we get, 
\begin{align}
    \frac{1}{2}\frac{d}{dt}\lVert\psi\rVert^2 + \frac{\alpha_1}{2}\lVert\nabla\psi\rVert^2 +\frac{\nu}{2}\lVert\nabla\textbf{w}\rVert^2 + \frac{\eta}{2}\lVert\textbf{w}\rVert^2 &\leq \mathcal{C}_1(1+ \lVert\nabla\Bar{\varphi}\rVert_{L^4}^2 +\lVert\nabla\Bar{\varphi}\rVert_{L^4}^4)\lVert\psi\rVert^2. 
\end{align}
After applying Gronwall's inequality, we get,
\begin{align}
    \|\psi(t)\|^2 + \alpha_1\|\nabla\psi\|_{L^2(0,T;H)}^2 + \nu\|\nabla\textbf{w}\|_{L^2(0,T;H)}^2 + \eta\|\textbf{w}\|_{L^2(0,T;H)}^2 \leq C\lVert{\psi_0}\rVert^2
\end{align}
Further using similar estimate as in \eqref{eq4.19}, we get $\|\psi'\|_{L^2(0,T;V')}^2 \leq C\lVert{\psi_0}\rVert^2$. Hence the uniqueness of the weak solution. 
 \end{proof}

\begin{theorem}[\textbf{Differentiability of the control to state operator}]
Assume that all the hypothesis of theorem 4.1 holds. Then the control to state operator $\mathcal{S}: \mathcal{U} \rightarrow \mathcal{V}$ is Fr\'echet differentiable when viewed as a mapping from $\mathcal{U} \rightarrow \mathcal{Z}$ where  $\mathcal{Z}:= L^\infty(0,T;H) \cap L^2(0,T;V) \cap H^1(0,T;V') \times L^2(0,T;\mathbb{V}_{div}),$ the weak solution space. And the Fr\'echet derivative $\mathcal{S}'$ at $\Bar{\textbf{U}}$ in the direction of $\textbf{U}\in \mathcal{U}$ is given by, 
$$\mathcal{S}'(\Bar{\textbf{U}})(\textbf{U}) = (\psi, \textbf{w}).$$ 
Where $(\psi, \textbf{w})$ is the unique weak solution to the linearized system \eqref{eq3.2}-\eqref{eq3.7}, which is linearization of \eqref{eq1}-\eqref{eq07} around the state $(\Bar{\varphi},\Bar{\textbf{u}}) =\mathcal{S}(\Bar{\textbf{U}})$ and with a control $\textbf{U}$.
\end{theorem}

\begin{proof}
Consider $\mathcal{S}(\Bar{\textbf{U}})=(\Bar{\varphi},\Bar{\textbf{u}})$ and $\mathcal{S}(\Bar{\textbf{U}}+\textbf{h})=(\varphi^\textbf{h},\textbf{u}^\textbf{h})$. Let $(\psi^\textbf{h},\textbf{w}^\textbf{h})$ be the solution to the linearized system \eqref{eq3.2}-\eqref{eq3.7} with  ${\textbf{U}} = \textbf{h}$. 
Define, 
\begin{align*}
    \rho^\textbf{h} &=\varphi^\textbf{h}-\Bar{\varphi}-\psi^\textbf{h} \\
    \textbf{v}^\textbf{h} &=\textbf{u}^\textbf{h}-\Bar{\textbf{u}}-\textbf{w}^\textbf{h} 
\end{align*}
We denote $\xi = \varphi^\textbf{h}-\Bar{\varphi}$, $\bm{\tau} = \textbf{u}^\textbf{h}-\Bar{\textbf{u}}$.  Note that, $(\xi, \bm{\tau})$ satisfies \eqref{eq3.052}-\eqref{eq3.053} and using the estimate \eqref{eq3.1}, we have $\xi \in  L^2(0,T;H^2) \cap L^\infty(0,T;V)$, $\xi' \in L^2(0,T;H)$ and $\bm{\tau} \in L^2(0,T; \mathbb{V}_{div})$.
\\We use $(\rho, \textbf{v})$ and $(\psi, \textbf{w})$ instead of $(\rho^\textbf{h}, \textbf{v}^\textbf{h})$ and $(\psi^\textbf{h}, \textbf{w}^\textbf{h})$
for the ease of notation. Therefore $(\rho, \textbf{v})$ solves the following system.
\begin{align}
    \rho'+\bm{\tau}\nabla\xi + \Bar{u}\nabla\rho + \textbf{v}\nabla\Bar{\varphi} &= \nabla\cdot\big( m(\Bar{\varphi})(\nabla(a\rho)-\nabla J*\rho)+m'(\Bar{\varphi})\rho(\nabla(a\Bar{\varphi})-\nabla J*\Bar{\varphi})\big) + \nabla\cdot(\lambda'(\Bar{\varphi})\xi\nabla\xi)\nonumber\\
    &\hspace{.25cm}+ \nabla\cdot\big( m'(\Bar{\varphi})\xi(\nabla(a\xi)-\nabla J*\xi)\big) + \nabla\cdot\big(\lambda(\Bar{\varphi})\nabla\rho+ \lambda'(\Bar{\varphi})\nabla\Bar{\varphi}\rho\big),\label{eq5.026}\\
    -\nabla\cdot(\nu\nabla \textbf{v})+ \eta \textbf{v} +\nabla \pi &= a\nabla(\rho\Bar{\varphi})+a\xi\nabla\xi+ (\nabla J*\Bar{\varphi})\rho+ (\nabla J*\rho)\Bar{\varphi}+ (\nabla J*\xi)\xi,\label{eq5.027}
\end{align}
\begin{align}
    \nabla\cdot\textbf{v} &= 0 \hspace{.5cm}\text{ in } Q,\label{eq5.028}\\
    \textbf{v} &= 0 \hspace{.5cm}\text{ on } \Sigma,\label{eq5.029}\\
     \big[m(\Bar{\varphi})\big(\nabla (a\rho)-\nabla J*\rho\big) + (m(\varphi^\textbf{h})&-m(\Bar{\varphi}))\big( \nabla (a\xi)-\nabla J*\xi\big) + \big( m(\varphi^\textbf{h})-m(\Bar{\varphi})-m'(\Tilde{\varphi})\psi\big)\big(\nabla (a\Bar{\varphi})-\nabla J*\Bar{\varphi}\big)\nonumber\\
    +\lambda(\Bar{\varphi})\nabla\rho +\big(\lambda(\varphi^\textbf{h})-\lambda(\Bar{\varphi})\big)\nabla\xi &+ \big(\lambda(\varphi^\textbf{h})-\lambda(\Bar{\varphi})-\lambda'(\Tilde{\varphi})\psi\big)\nabla\Bar{\varphi}\big]\cdot\textbf{n} =0 \hspace{.5cm} \text{  on } \Sigma,\label{eq5.030}\\
    \rho(x,0) &= 0 \text{ in } \Omega.\label{eq5.031}
\end{align}

Our aim is to prove the following limit in order to get the differentiability of the operator $\mathcal{S}$. In particular we want to show that
\begin{equation}\label{eq4.39}
\frac{\lVert S(\Bar{\textbf{U}}+\textbf{h})-S(\Bar{\textbf{U}})-(\psi,\textbf{w})\rVert_{\mathcal{Z}}}{\lVert \textbf{h}\rVert_\mathcal{U}}  = \frac{\lVert(\rho,\textbf{v})\rVert_\mathcal{Z}}{\lVert \textbf{h}\rVert_{\mathcal{V}}} \longrightarrow 0 \hspace{.75cm} \text{  as  } \lVert\textbf{h}\rVert_{\mathcal{V}}\rightarrow 0.
\end{equation}
To achieve this we need to prove the regularity of a solution to the system \eqref{eq5.026}-\eqref{eq5.031}.
Consider the following weak formulation of the system. For $x\in V,$ $\textbf{y}\in\mathbb{V}_{div}$,
\begin{align}
    <\rho',x>+(\bm{\tau}\nabla\xi,x) + (\Bar{u}\nabla\rho,x) + (\textbf{v}\nabla\Bar{\varphi},x) &= -\big(\big( m(\Bar{\varphi})(\nabla(a\rho)-\nabla J*\rho)+m'(\Bar{\varphi})\rho(\nabla(a\Bar{\varphi})-\nabla J*\Bar{\varphi})\big),\nabla x\big) \nonumber\\
    &\hspace{.5cm} - \big(\big( m'(\Bar{\varphi})\xi(\nabla(a\xi)-\nabla J*\xi)\big),\nabla x\big)-\big((\lambda'(\Bar{\varphi})\xi\nabla\xi),\nabla x\big)  \nonumber\\&\hspace{.5cm}  -\big(\big(\lambda(\Bar{\varphi})\nabla\rho+ \lambda'(\Bar{\varphi})\nabla\Bar{\varphi}\rho\big),\nabla x\big),\label{eq3.10}\\
    (\nu\nabla \textbf{v},\nabla\textbf{y})+ \eta (\textbf{v},\textbf{y}) &= -(\nabla(a) \rho \Bar{\varphi},\textbf{y})+ (a\xi\nabla\xi,\textbf{y})+ ((\nabla J*\Bar{\varphi})\rho,\textbf{y}) \nonumber\\
    &\hspace{.5cm}+ ((\nabla J*\rho)\Bar{\varphi},\textbf{y}) + ((\nabla J*\xi)\xi,\textbf{y}).\label{eq3.11}
\end{align}
For $x=\rho$ and $\textbf{y}=\textbf{v}$ we get, 
\begin{align}
    \frac{1}{2}\frac{d}{dt}\lVert\rho\rVert^2+(\bm{\tau}\nabla\xi,\rho) + (\textbf{v}\nabla\Bar{\varphi},\rho) &= -\big( m(\Bar{\varphi})(\nabla(a)\rho-\nabla J*\rho),\nabla\rho\big)-\big(m'(\Bar{\varphi})\rho(\nabla(a\Bar{\varphi})-\nabla J*\Bar{\varphi}),\nabla \rho\big) \nonumber\\
    &\hspace{.5cm} - \big( m'(\Bar{\varphi})\xi(\nabla(a\xi)-\nabla J*\xi),\nabla \rho\big) -\big((m(\Bar{\varphi})a+\lambda(\Bar{\varphi}))\nabla\rho,\nabla\rho\big) \nonumber\\
    &\hspace{.5cm}  -\big(\lambda'(\Bar{\varphi})\nabla\Bar{\varphi}\rho,\nabla \rho\big) -\big(\lambda'(\Bar{\varphi})\xi\nabla\xi,\nabla \rho\big) \label{eq3.12}\\
    \nu\lVert\nabla \textbf{v}\rVert^2+ \eta \lVert\textbf{v}\rVert^2&= -(\nabla(a)\rho\Bar{\varphi},\textbf{v})+ (a\xi\nabla\xi,\textbf{v})+ ((\nabla J*\Bar{\varphi})\rho,\textbf{v})+ ((\nabla J*\rho)\Bar{\varphi},\textbf{v})\nonumber\\
    &\hspace{.5cm}+ ((\nabla J*\xi)\xi,\textbf{v})\label{eq3.13}
\end{align}
We denote the terms of \eqref{eq3.12} on L.H.S by $I_1, I_2, I_3$  and the terms on R.H.S by $I_4, I_5,...,I_9$. Similarly we denote the terms of \eqref{eq3.13} on R.H.S by $J_1, J_2,...,J_5.$ We estimate each term by using H\"olders, Young's and Gagliardo-Nirenberg inequalities as applicable.
\begin{align*}
   |I_2|&=|(\bm{\tau}\nabla\xi,\rho)| \leq  \lVert\bm{\tau}\rVert_{L^4}\lVert\xi\rVert_{L^4}\lVert\nabla\rho\rVert \leq \epsilon\lVert\nabla\rho\rVert^2+ c\lVert\xi\rVert\lVert\nabla\xi\rVert\lVert\nabla\bm{\tau}\rVert^2\\ &\leq \epsilon\lVert\nabla\rho\rVert^2+ c(\lVert\xi\rVert^2+\lVert\nabla\xi\rVert^2)\lVert\nabla\bm{\tau}\rVert^2,\\
   |I_3|&= |(\textbf{v}\nabla\Bar{\varphi},\rho)| \leq \lVert\textbf{v}\rVert_{L^4}\lVert\nabla\Bar{\varphi}\rVert_{L^4}\lVert\rho\rVert \leq \lVert\nabla\textbf{v}\rVert\lVert\nabla\Bar{\varphi}\rVert^\frac{1}{2}\lVert\Bar{\varphi}\rVert_{H^2}^\frac{1}{2}\lVert\rho\rVert\\
    &\leq \epsilon_1\lVert\nabla\textbf{v}\rVert^2+c\lVert\nabla\Bar{\varphi}\rVert\lVert\Bar{\varphi}\rVert_{H^2}\lVert\rho\rVert^2,\\
    |I_4| &\leq \|m\|_{L^\infty}(\|\nabla a\|_{L^\infty}+\|\nabla J\|_{L^1})\lVert\rho\rVert\lVert\nabla\rho\rVert \leq \epsilon\lVert\nabla\rho\rVert^2+c\lVert\rho\rVert^2,\\ 
    |I_5| &\leq \|m'\|_{L^\infty}(\|\nabla a\|_{L^\infty}+\|\nabla J\|_{L^1})\lVert\Bar{\varphi}\rVert_{L^\infty}\lVert\rho\rVert\lVert\nabla\rho\rVert + \|m'\|_{L^\infty}\|a\|_{L^\infty}\lVert\nabla\Bar{\varphi}\rVert_{L^4}\lVert\rho\rVert_{L^4}\lVert\nabla\rho\rVert\\
    &\leq \epsilon\lVert\nabla\rho\rVert^2+c\lVert\rho\rVert^2+ c\lVert\lVert\nabla\Bar{\varphi}\rVert^\frac{1}{2}\lVert\Bar{\varphi}\rVert_{H^2}^\frac{1}{2}\lVert\rho\rVert^\frac{1}{2}\lVert\nabla\rho\rVert^\frac{1}{2}\lVert\nabla\rho\rVert \\&\leq 2\epsilon\lVert\nabla\rho\rVert^2+c\lVert\rho\rVert^2+ c\lVert\nabla\Bar{\varphi}\rVert\lVert\Bar{\varphi}\rVert_{H^2}\lVert\rho\rVert\lVert\nabla\rho\rVert \\
    &\leq 3\epsilon\lVert\nabla\rho\rVert^2+c\lVert\rho\rVert^2+ c\lVert\nabla\Bar{\varphi}\rVert^2\lVert\Bar{\varphi}\rVert_{H^2}^2\lVert\rho\rVert^2,\\    
    |I_6| &\leq \|m'\|_{L^\infty}(\|\nabla a\|_{L^\infty}+\|\nabla J\|_{L^1})\lVert\xi\rVert_{L^4}^2\lVert\nabla\rho\rVert + \|m'\|_{L^\infty}\|a\|_{L^\infty}\lVert\xi\rVert_{L^4}\lVert\nabla\xi\rVert_{L^4}\lVert\nabla\rho\rVert\\
    &\leq 2\epsilon\lVert\nabla\rho\rVert^2+ c\lVert\xi\rVert^2\lVert\nabla\xi\rVert^2+ c\lVert\nabla\xi\rVert^2\lVert\xi\rVert_{H^2}^2,\\
    I_7 &= \big((m(\Bar{\varphi})a+\lambda(\Bar{\varphi}))\nabla\rho,\nabla\rho\big) \geq \alpha_1\|\nabla \rho\|^2\\  
    |I_8| &\leq \|\lambda'\|_{L^\infty}\lVert\nabla\Bar{\varphi}\rVert_{L^4}\lVert\rho\rVert_{L^4}\lVert\nabla\rho\rVert \leq \epsilon\lVert\nabla\rho\rVert^2+ c\lVert\lVert\nabla\Bar{\varphi}\rVert\lVert\Bar{\varphi}\rVert_{H^2}\lVert\rho\rVert\lVert\nabla\rho\rVert\\
    &\leq 2\epsilon\lVert\nabla\rho\rVert^2+ c\lVert\lVert\nabla\Bar{\varphi}\rVert^2\lVert\Bar{\varphi}\rVert_{H^2}^2\lVert\rho\rVert^2 \end{align*}
\begin{align*} 
    |I_9| &\leq \|\lambda'\|_{L^\infty}\lVert\xi\rVert_{L^4}\lVert\nabla\xi\rVert_{L^4}\lVert\nabla\rho\rVert  \leq  \epsilon\lVert\nabla\rho\rVert^2+ c\lVert\xi\rVert_{L^4}^4+ c\lVert\nabla\xi\rVert_{L^4}^4\\
    &\leq \epsilon\lVert\nabla\rho\rVert^2+ c \lVert\xi\rVert^2\lVert\nabla\xi\rVert^2+ c\lVert\nabla\xi\rVert^2\lVert\xi\rVert_{H^2}^2.
\end{align*}
Further,\vspace{-.15cm}
\begin{align*}
    |J_1| &\leq \|\nabla a\|_{L^\infty}\lVert\rho\rVert\lVert\Bar{\varphi}\rVert_{L^\infty}\lVert\textbf{v}\rVert \leq \epsilon_1\lVert\textbf{v}\rVert^2+ c\lVert\rho\rVert^2,\\
    |J_2| &\leq \|a\|_{L^\infty}\lVert\xi\rVert_{L^4}\lVert\nabla\xi\rVert_{L^4}\lVert\textbf{v}\rVert 
    \leq \epsilon_1\lVert\textbf{v}\rVert^2+ c\lVert\xi\rVert_{L^4}^4 + c\lVert\nabla\xi\rVert_{L^4}^4\\
    &\leq \epsilon_1\lVert\textbf{v}\rVert^2+ c \lVert\xi\rVert^2\lVert\nabla\xi\rVert^2+ c\lVert\nabla\xi\rVert^2\lVert\xi\rVert_{H^2}^2,\\
    |J_3| &\leq \|\nabla J\|_{L^1}\lVert\Bar{\varphi}\rVert_{L^\infty}\lVert\rho\rVert\lVert\textbf{v}\rVert
    \leq \epsilon_1\lVert\textbf{v}\rVert^2+ c\lVert\rho\rVert^2,\\
    |J_4| &\leq \|\nabla J\|_{L^1}\lVert\rho\rVert\lVert\Bar{\varphi}\rVert_{L^\infty}\lVert\textbf{v}\rVert
    \leq \epsilon_1\lVert\textbf{v}\rVert^2+ c\lVert\rho\rVert^2,\\
    |J_5| &\leq \|\nabla J\|_{L^1}\lVert\xi\rVert_{L^4}^2\lVert\textbf{v}\rVert \leq \epsilon_1\lVert\textbf{v}\rVert^2+ c\lVert\xi\rVert^2\lVert\nabla\xi\rVert^2.
\end{align*}
\\Substituting above derived estimates in \eqref{eq3.12} and \eqref{eq3.13} results in following two inequalities.
\begin{align}
    \frac{1}{2}\frac{d}{dt}\lVert\rho\rVert^2 + \alpha_1\lVert\nabla\rho\rVert^2 &\leq 10\epsilon\lVert\nabla\rho\rVert^2 + \epsilon_2\lVert\nabla\textbf{v}\rVert^2 +C( 1+ \lVert\nabla\Bar{\varphi}\rVert^2 + \lVert\Bar{\varphi}\rVert_{H^2}^2 + \lVert\lVert\nabla\Bar{\varphi}\rVert^2\lVert\nabla\Bar{\varphi}\rVert_{H^2}^2)\lVert\rho\rVert^2 \nonumber\\
    &\hspace{.5cm} + C(\lVert\xi\rVert^2\lVert\nabla\bm{\tau}\rVert^2+\lVert\nabla\xi\rVert^2\lVert\nabla\bm{\tau}\rVert^2 + \lVert\xi\rVert^2\lVert\nabla\xi\rVert^2+ \lVert\nabla\xi\rVert^2\lVert\xi\rVert_{H^2}^2),  \label{eq3.14}\\
    \nu\lVert\nabla \textbf{v}\rVert^2+ \eta \lVert\textbf{v}\rVert^2 &\leq 5\epsilon_1\lVert\textbf{v}\rVert^2+ C(\lVert\rho\rVert^2 + \lVert\xi\rVert^2\lVert\nabla\xi\rVert^2+ \lVert\nabla\xi\rVert^2\lVert\xi\rVert_{H^2}^2). \label{eq3.15}
\end{align}
\\Adding \eqref{eq3.14} and \eqref{eq3.15}, and for a choice of $\epsilon < \frac{\alpha_1}{20}$, $\epsilon_1 < \frac{\eta}{10}$ and $\epsilon_2 < \frac{\nu}{2}$ we get,
\begin{align}
    \frac{1}{2}\frac{d}{dt}\lVert\rho\rVert^2 + \frac{\alpha_1}{2}\lVert\nabla\rho\rVert^2 +\frac{\nu}{2}\lVert\nabla \textbf{v}\rVert^2+ \frac{\eta}{2} \lVert\textbf{v}\rVert^2 &\leq   C( 1+ \lVert\nabla\Bar{\varphi}\rVert^2 + \lVert\Bar{\varphi}\rVert_{H^2}^2 + \lVert\lVert\nabla\Bar{\varphi}\rVert^2\lVert\Bar{\varphi}\rVert_{H^2}^2)\lVert\rho\rVert^2 + C\Lambda(t).\label{eq3.16}
\end{align}
\\Where $\Lambda(t) = (\lVert\xi\rVert^2\lVert\nabla\bm{\tau}\rVert^2+\lVert\nabla\xi\rVert^2\lVert\nabla\bm{\tau}\rVert^2 + \lVert\xi\rVert^2\lVert\nabla\xi\rVert^2+ \lVert\nabla\xi\rVert^2\lVert\xi\rVert_{H^2}^2)$. 
Integrate the inequality \eqref{eq3.16} over $[0,T]$ and then apply Gronwall's inequality; we get 
\begin{align}\label{eq4.47}
   \lVert\rho(t)\rVert^2 + \alpha_1\lVert\nabla\rho\rVert_{L^2(0,T;H)}^2 + \nu\lVert\nabla\textbf{v}\rVert_{L^2(0,T;\mathbb{G}_{div})}^2 + \eta\lVert\textbf{v}\rVert_{L^2(0,T;\mathbb{G}_{div})}^2 &\leq  C(1+ \gamma(t))\int\limits_{0}^{T}\Lambda(t)
\end{align}
Note that, using H\"olders inequality and the estimate \eqref{eq3.1},
\begin{align*}
    \int\limits_{0}^{T}\Lambda(t) &= \int\limits_{0}^{T}(\lVert\xi\rVert^2\lVert\nabla\bm{\tau}\rVert^2+\lVert\nabla\xi\rVert^2\lVert\nabla\bm{\tau}\rVert^2 + \lVert\xi\rVert^2\lVert\nabla\xi\rVert^2+ \lVert\nabla\xi\rVert^2\lVert\xi\rVert_{H^2}^2)
    \leq C \lVert \textbf{h}\rVert_{L^2(0,T;\mathbb{V}'_{div})}^4 
\end{align*}
Substituting the above estimate in \eqref{eq4.47} and applying Korn's inequality we get,
\begin{align}
    \lVert\rho\rVert_{L^\infty(0,T;H)}^2 + \alpha_1\lVert\nabla\rho\rVert_{L^2(0,T;H)}^2 + c_1\lVert\textbf{v}\rVert_{L^2(0,T;\mathbb{V}_{div})}^2 &\leq C(1+\gamma(T)) \lVert \textbf{h}\rVert_{L^2(0,T;\mathbb{V}'_{div})}^4 ,\label{eq3.17}
\end{align}
for a constant $c_1 > 0$.
Thus we obtain,
\begin{align*}
    \rho^h &\in L^\infty(0,T;H) \cap L^2(0,T;V),\\
    \textbf{v}^h &\in L^2(0,T;\mathbb{V}_{div}).
\end{align*}
Further, to obtain the regularity of $\rho',$ consider \eqref{eq3.10}.
\begin{align*}
    |(\bm{\tau}\xi,\nabla x)| &\leq \lVert\bm{\tau}\rVert_{L^4}\lVert\xi\rVert_{L^4}\lVert\nabla x\rVert \leq \lVert\nabla\bm{\tau}\rVert\lVert\xi\rVert_{V}\lVert\nabla x\rVert,\\
    |(\Bar{u}\nabla\rho,x)| &\leq \lVert\Bar{u}\rVert_{L^\infty}\lVert\rho\rVert\lVert\nabla x\rVert,\\
    |(\textbf{v}\Bar{\varphi},\nabla x)| &\leq \lVert\Bar{\varphi}\rVert_{L^\infty}\lVert\textbf{v}\rVert\lVert\nabla x\rVert,\\   
    |( m(\Bar{\varphi})(\nabla(a\rho)-\nabla J*\rho),\nabla x)| &\leq \|m\|_{L^\infty}(\|\nabla a\|_{L^\infty}+ \|\nabla J\|_{L^1})\lVert\rho\rVert\lVert\nabla x\rVert + \|m\|_{L^\infty}\|a\|_{L^\infty}\lVert\nabla\rho\rVert\lVert\nabla x\rVert \\
    &\leq c\lVert\rho\rVert_{V}\lVert\nabla x\rVert,\\    
    |(m'(\Bar{\varphi})\rho(\nabla(a\Bar{\varphi})-\nabla J*\Bar{\varphi}),\nabla x)| &\leq \|m'\|_{L^\infty}\lVert\rho\rVert(\|\nabla a\|_{L^\infty}+\|\nabla J\|_{L^1})\lVert\Bar{\varphi}\rVert_{L^\infty}\lVert\nabla x\rVert + \|m'\|_{L^\infty}\|a\|_{L^\infty}\lVert\Bar{\varphi}\rVert_{L^4}\lVert\rho\rVert_{L^4}\lVert\nabla x\rVert\\
    &\leq c(1+\lVert\Bar{\varphi}\rVert_V)\lVert\rho\rVert_{V}\lVert\nabla x\rVert,
\end{align*}
\begin{align*}     
    |(m'(\Bar{\varphi})\xi(\nabla(a\xi)-\nabla J*\xi),\nabla x)| &\leq \|m'\|_{L^\infty}(\|\nabla a\|_{L^\infty}+\|\nabla J\|_{L^1})\lVert\xi\rVert_{L^4}^2\lVert\nabla x\rVert + \|m'\|_{L^\infty}\|a\|_{L^\infty}\lVert\nabla\xi\rVert_{L^4}\lVert\xi\rVert_{L^4}\lVert\nabla x\rVert\\
    &\leq c(\lVert\xi\rVert\lVert\nabla\xi\rVert+\lVert\nabla\xi\rVert\lVert\xi\rVert_{H^2})\lVert\nabla x\rVert,\\
    |(\lambda(\Bar{\varphi})\nabla\rho+ \lambda'(\Bar{\varphi})\nabla\Bar{\varphi}\rho,\nabla x)| &\leq \|\lambda\|_{L^\infty}\lVert\nabla\rho\rVert\lVert\nabla x\rVert+ \|\lambda'\|_{L^\infty}\lVert\nabla\Bar{\varphi}\rVert_{L^4}\lVert\rho\rVert_{L^4}\lVert\nabla x\rVert\\
    &\leq c(1+\lVert\Bar{\varphi}\rVert_{H^2})\lVert\rho\rVert_{V}\lVert\nabla x\rVert,\\
    |(\lambda'(\Bar{\varphi})\xi\nabla\xi,\nabla x)| &\leq \|\lambda'\|_{L^\infty}\lVert\xi\rVert_{L^4}\lVert\nabla\xi\rVert_{L^4}\lVert\nabla x\rVert  \leq \|\lambda'\|_{L^\infty}(\lVert\xi\rVert_{L^4}^2+\lVert\nabla\xi\rVert_{L^4}^2)\lVert\nabla x\rVert\\
    &\leq c(\lVert\xi\rVert\lVert\nabla\xi\rVert+\lVert\nabla\xi\rVert\lVert\xi\rVert_{H^2})\lVert\nabla x\rVert.
    \end{align*}
Therefore,
\begin{equation*}
   |<\rho',x>| \leq C\big(\lVert\textbf{v}\rVert+ \lVert\rho\rVert_{V} + \lVert\nabla\bm{\tau}\rVert\lVert\xi\rVert_{V} +  \lVert\xi\rVert\lVert\nabla\xi\rVert+\lVert\nabla\xi\rVert\lVert\xi\rVert_{H^2}\big)\lVert x\rVert_{V}. 
\end{equation*}
This implies $ \lVert\rho'\rVert_{V'} \leq C\big(\lVert\textbf{v}\rVert+ \lVert\rho\rVert_{V} + \lVert\nabla\bm{\tau}\rVert\lVert\xi\rVert_{V} +  \lVert\xi\rVert\lVert\nabla\xi\rVert+\lVert\nabla\xi\rVert\lVert\xi\rVert_{H^2}\big)$.   Integration over $[0,T]$ and  substituting estimates \eqref{eq3.1} and \eqref{eq3.17}, we get,
\begin{equation}\label{eq3.19}
   \lVert\rho'\rVert_{L^2(0,T;V')} \leq C\lVert\textbf{h}\rVert_{L^2(0,T;\mathbb{V}_{div}')}^2 .
\end{equation}
Since $\mathcal{U} \subseteq L^2(0,T;\mathbb{V}_{div}'),$ as $\lVert \textbf{h}\rVert_\mathcal{U}\rightarrow 0,$
\begin{align}\label{eq4.50}
    \frac{\lVert(\rho,\textbf{v})\rVert_{\mathcal{Z}}}{\lVert \textbf{h}\rVert_\mathcal{U}}  
    \leq {C}\lVert \textbf{h}\rVert_\mathcal{U} 
\end{align}
which goes to 0 as $\lVert \textbf{h}\rVert_\mathcal{U} $ goes to 0 and thus we have proved the differentiability of the control to state operator.
\end{proof}

\section{Characterisation of an optimal control}
In this section we will use the differentiability of the control to state operator to obtain a first order necessary optimality condition for the OCP. Further, we characterise an optimal control in terms of an adjoint variable. 
\subsection{First order necessary optimality condition}
\begin{theorem}
Let $\Bar{\textbf{U}} \in \mathcal{U}_{ad}$ be an optimal control and $(\Bar{\varphi},\Bar{\textbf{u}})$ be the corresponding optimal state. Assume all the hypotheses of Theorem 4.1 holds. Then the following inequality holds. 
 \begin{align}\label{eq5.01}
    \int\limits_{0}^{T}\int_{\Omega}\psi(\Bar{\varphi}-\varphi_d) + \int\limits_{0}^{T}\int_{\Omega}\textbf{w}(\Bar{\textbf{u}}-\textbf{u}_d) + \int_{\Omega}\psi(T)(\Bar{\varphi}(T)-\varphi_\Omega)+ \int\limits_{0}^{T}\int_\Omega (\textbf{U}-\Bar{\textbf{U}})\Bar{\textbf{U}} \geq 0.
\end{align}
 $\forall \ \textbf{U}\in \mathcal{U}_{ad}$, where $(\psi,\textbf{w})$ is the unique weak solution to the linearized system \eqref{eq3.2}-\eqref{eq3.7} corresponding to the control $\textbf{U}-\Bar{\textbf{U}}.$
\end{theorem}
\begin{proof}
Consider $G(\textbf{U}):=\mathcal{J}(S(\textbf{U}),\textbf{U})$, $\textbf{U}\in \mathcal{U}_{ad}$. Let $\Bar{\textbf{U}}$ be an optimal control that minimises $G(\textbf{U})$. Then by \cite[Lemma 2.21]{FTO},
\begin{equation}\label{eq5.2}
    G'(\Bar{\textbf{U}})(\textbf{U}-\Bar{\textbf{U}}) \geq 0,\,\,\,\, \forall\,\, \textbf{U}\in \mathcal{U}_{ad}.
\end{equation}
We denote $\mathcal{S}(\Bar{\textbf{U}}+\beta \textbf{U})$ and $\mathcal{S}(\Bar{\textbf{U}})$ by $(\varphi_\beta, \textbf{u}_\beta),$ $(\Bar{\varphi},\Bar{\textbf{u}})$ respectively.
\begin{align*}
    G(\Bar{\textbf{U}}+\beta \textbf{U})-G(\Bar{\textbf{U}}) &= \frac{1}{2}\Big( \int\limits_{0}^{T}\int_{\Omega}(\varphi_\beta-\varphi_d)^2- (\Bar{\varphi}-\varphi_d)^2\Big)+ \frac{1}{2}\Big( \int\limits_{0}^{T}\int_{\Omega}(\textbf{u}_\beta-\textbf{u}_d)^2- (\Bar{\textbf{u}}-\textbf{u}_d)^2\Big)\\&\hspace{.5cm} + \frac{1}{2}\Big( \int_{\Omega}(\varphi_\beta-\varphi_\Omega)^2- (\Bar{\varphi}-\varphi_\Omega)^2\Big) + \int\limits_{0}^{T}\int_\Omega \beta^2\textbf{U}^2+2\beta \textbf{U}\Bar{\textbf{U}}
\end{align*}
\begin{align*}
    &= \frac{1}{2}\Big[ \int\limits_{0}^{T}\int_{\Omega}(\varphi_\beta-\Bar{\varphi})(\varphi_\beta+\Bar{\varphi}-2\varphi_d) + \int\limits_{0}^{T}\int_{\Omega}(\textbf{u}_\beta-\Bar{\textbf{u}})(\textbf{u}_\beta+\Bar{\textbf{u}}-2\textbf{u}_d) \\
    & \hspace{.5cm}+ \int_{\Omega}(\varphi_\beta(T)-\Bar{\varphi}(T))(\varphi_\beta(T)+\Bar{\varphi}(T)-2\varphi_\Omega)+ \int\limits_{0}^{T}\int_\Omega \beta^2\textbf{U}^2+2\beta \textbf{U}\Bar{\textbf{U}}\Big],\\
    \frac{G(\Bar{\textbf{U}}+\beta \textbf{U})-G(\Bar{\textbf{U}})}{\beta} &= \frac{1}{2}\Big[ \int\limits_{0}^{T}\int_{\Omega}\Big(\frac{\varphi_\beta-\Bar{\varphi}}{\beta}\Big)(\varphi_\beta+\Bar{\varphi}-2\varphi_d) + \int\limits_{0}^{T}\int_{\Omega}\Big(\frac{\textbf{u}_\beta-\Bar{\textbf{u}}}{\beta}\Big)(\textbf{u}_\beta+\Bar{\textbf{u}}-2\textbf{u}_d) \\
    & \hspace{.5cm}+ \int_{\Omega}\Big(\frac{\varphi_\beta(T)-\Bar{\varphi}(T)}{\beta}\Big)(\varphi_\beta(T)+\Bar{\varphi}(T)-2\varphi_\Omega)+ \int\limits_{0}^{T}\int_\Omega \beta \textbf{U}^2+2 \textbf{U}\Bar{\textbf{U}}\Big].
\end{align*}
We have, $(\varphi_\beta-\Bar{\varphi}, \textbf{u}_\beta-\Bar{\textbf{u}})$ satisfies equations similar to \eqref{eq3.052}-\eqref{eq3.053}. Using \eqref{eq4.39}, \eqref{eq4.50} we have,  
\begin{equation*}
\frac{\lVert S(\Bar{\textbf{U}}+\beta\textbf{U})-S(\Bar{\textbf{U}})-(\psi_\beta,\textbf{w}_\beta)\rVert_{\mathcal{Z}}}{\lVert \beta\textbf{U}\rVert_\mathcal{U}} = \frac{\lVert (\varphi_\beta-\Bar{\varphi},\textbf{u}_\beta-\Bar{\textbf{u}} ) -(\psi_\beta,\textbf{w}_\beta)\rVert_{\mathcal{Z}}}{\lVert \beta\textbf{U}\rVert_\mathcal{U}} \leq \lVert \beta\textbf{U}\rVert_\mathcal{U} ,
\end{equation*} 
 where $(\psi_\beta, \textbf{w}_\beta)$ is the unique weak solution to the linearized system \eqref{eq3.2}-\eqref{eq3.7} corresponding to the control $\beta\textbf{U}$. And as $\beta \rightarrow 0$, $ (\frac{\varphi_\beta-\Bar{\varphi}}{\beta}, \frac{\textbf{u}_\beta-\Bar{\textbf{u}}}{\beta})$ goes to $(\psi_{\textbf{U}}, \textbf{w}_\textbf{U})$ which is the unique weak solution of the linearized system \eqref{eq3.2}-\eqref{eq3.7} with control $\textbf{U}$. 
\\Since $\Bar{\textbf{U}}+\beta\textbf{U} \rightarrow \Bar{\textbf{U}}$ as $\beta \rightarrow 0$, we have  $(\varphi_\beta, \textbf{u}_\beta ) \rightarrow (\Bar{\varphi}, \Bar{\textbf{u}})$ in $L^2(0,T;H) \times L^2(0,T;\mathbb{G}_{div})$ as $\beta \rightarrow 0$ using the continuous dependence estimate \eqref{eq3.1}. Therefore by passing the limit $\beta \rightarrow 0$, we obtain,

\begin{align}\label{eq4.01}
    G'(\Bar{\textbf{U}})(\textbf{U}) &= \int\limits_{0}^{T}\int_{\Omega}\psi_{\textbf{U}}(\Bar{\varphi}-\varphi_d) + \int\limits_{0}^{T}\int_{\Omega}\textbf{w}_\textbf{U}(\Bar{\textbf{u}}-\textbf{u}_d) + \int_{\Omega}\psi_{\textbf{U}}(T)(\Bar{\varphi}(T)-\varphi_\Omega)+ \int\limits_{0}^{T}\int_\Omega \textbf{U}\Bar{\textbf{U}}.
\end{align}
 Hence using $\textbf{U}-\Bar{\textbf{U}}$ in place of $\textbf{U}$ in the above inequality and   using \eqref{eq5.2} we get, first order optimality condition that $\forall \textbf{U}\in \mathcal{U}_{ad},$ namely 
\begin{align*}
    G'(\Bar{\textbf{U}})(\textbf{U}-\Bar{\textbf{U}}) &= \int\limits_{0}^{T}\int_{\Omega}
    \psi_{{\textbf{U}-\Bar{\textbf{U}} }}(\Bar{\varphi}-\varphi_d) + \int\limits_{0}^{T}\int_{\Omega}\textbf{w}_{{\textbf{U}-\Bar{\textbf{U}} }}(\Bar{\textbf{u}}-\textbf{u}_d) + \int_{\Omega}\psi_{{\textbf{U}-\Bar{\textbf{U}} }}(T)(\Bar{\varphi}(T)-\varphi_\Omega)+ \int\limits_{0}^{T}\int_\Omega (\textbf{U}-\Bar{\textbf{U}})\Bar{\textbf{U}} \\ &\geq 0.
\end{align*}
\end{proof}

\subsection{ The adjoint problem}
Now we wish to characterise an optimal control in terms of an adjoint variable. For, we introduce an adjoint system and we discuss the well-posedness of the system. 
Consider the following adjoint system in $(\xi,\textbf{v})$ corresponding to the system \eqref{eq1}-\eqref{eq07}. 
\begin{align}
     -\xi' - \Bar{u}\nabla\xi +m'(\Bar{\varphi})((\nabla a)\Bar{\varphi}-\nabla J*\Bar{\varphi})\nabla\xi  &- \nabla J*(m(\Bar{\varphi})\nabla\xi)  - (m(\Bar{\varphi})a   + \lambda(\Bar{\varphi}))\Delta\xi 
      \nonumber\\+ ((\nabla a)\Bar{\varphi} - \nabla J*\Bar{\varphi})\textbf{v}- \nabla J*(\Bar{\varphi}\textbf{v}) &= (\Bar{\varphi}-\varphi_d),\label{eq4.1}\\
     -\nabla\cdot(\nu\nabla\textbf{v})  +\eta\textbf{v}+\xi\nabla\Bar{\varphi}+\nabla\pi &=(\Bar{\textbf{u}}-\textbf{u}_d),\label{eq4.2}\\
     \nabla\cdot \textbf{v} &=0, \text{ in } Q,\\
     \textbf{v}=0, \hspace{.25cm}\frac{\partial\xi}{\partial \textbf{n}} &=0, \text{ on } \Sigma,\\
     \xi(x,T) &=(\Bar{\varphi}(T)-\varphi_\Omega)(x),  \text{ in }\Omega.\label{eq4.3}
\end{align}

\begin{theorem}(Existence and uniqueness of a weak solution to the adjoint system)
Let $\varphi_0 \in H^2(\Omega) \cap L^\infty(\Omega)$, $J \in W^{2,1}(\Omega)$ and $[\textbf{J}], [\textbf{A1}]-[\textbf{A4}]$ holds true. In addition, assume that the mobility, $m \in C^2[-1,1]$ and $\lambda \in C^2[-1,1]$. Then there exists a unique weak solution to the adjoint system \eqref{eq4.1}-\eqref{eq4.3} defined as in Theorem 4.2.
\end{theorem}

\begin{proof}
Consider the following weak formulation of the adjoint problem, $\forall$ $\zeta \in V$ and $\textbf{z}\in \mathbb{V}_{div}$,
\begin{align*}
     -<\xi',\zeta> - (\Bar{u}\nabla\xi,\zeta) + (m'(\Bar{\varphi})((\nabla (a)\Bar{\varphi}-\nabla J*\Bar{\varphi})\nabla\xi,\zeta)  - (\nabla J*(m(\Bar{\varphi})\nabla\xi),\zeta)  \nonumber\\- ((m(\Bar{\varphi})a+ \lambda(\Bar{\varphi}))\Delta\xi,\zeta) 
     + ((\nabla (a)\Bar{\varphi}-\nabla J*\Bar{\varphi})\textbf{v},\zeta)- (\nabla J*(\Bar{\varphi}\textbf{v}),\zeta) &= (\Bar{\varphi}-\varphi_d,\zeta),\\
     (\nu\nabla\textbf{v},\nabla\textbf{z})  +\eta(\textbf{v},\textbf{z})+ (\xi\nabla\Bar{\varphi},\textbf{z}) &=(\Bar{\textbf{u}}-\textbf{u}_d, \textbf{z}) .
\end{align*}
Consider $(\xi, \textbf{v}) \in  L^\infty(0,T;H) \cap L^2(0,T;V) \cap H^1(0,T;V')\times  L^2(0,T;\mathbb{V}_{div})$. We say, $(\xi, \textbf{v})$ is a weak solution to the system \eqref{eq4.1}-\eqref{eq4.3} if it satisfies the above weak formulation $\forall$ $\zeta \in V$ and $\textbf{z}\in \mathbb{V}_{div}$.
\\We prove the existence of a weak solution by the Faedo-Galerkin approximation method. We omit the detailed calculations as they are similar to the ones in the proof of  Theorem 4.2. We will formally derive basic a-priori estimates which would help to extract a weakly convergent subsequence and hence a weak solution. 
For, testing \eqref{eq4.1} with $\xi \in V$,  we obtain, 
\begin{align*}
     -\frac{1}{2}\frac{d}{dt}\lVert\xi\rVert^2- (\Bar{u}\nabla\xi,\xi) + (m'(\Bar{\varphi})((\nabla (a)\Bar{\varphi}-\nabla J*\Bar{\varphi})\nabla\xi,\xi)  - (\nabla J*(m(\Bar{\varphi})\nabla\xi),\xi) &- ((m(\Bar{\varphi})a+ \lambda(\Bar{\varphi}))\Delta\xi,\xi)\nonumber\\
     + ((\nabla (a)\Bar{\varphi}-\nabla J*\Bar{\varphi})\textbf{v},\xi)- (\nabla J*(\Bar{\varphi}\textbf{v}),\xi) &= (\Bar{\varphi}-\varphi_d,\xi).
\end{align*}
We will denote the terms on L.H.S of above equation by $I_1,I_2,...,I_7,$ and estimate each term separately,
\begin{align*}
    |I_2| &\leq \lVert\Bar{\textbf{u}}\rVert_{L^\infty}\lVert\nabla\xi\rVert\lVert\xi\rVert  \leq \frac{\alpha_1}{12}\lVert\nabla\xi\rVert^2 + c\lVert\Bar{\textbf{u}}\rVert_{L^\infty}^2\lVert\xi\rVert^2,\\
    |I_3| &\leq \|m'\|_{L^\infty}(\lVert\nabla a\rVert_{L^\infty} + \lVert\nabla J\rVert_{L^1})\lVert\Bar{\varphi}\rVert_{L^\infty}\lVert\nabla\xi\rVert\lVert\xi\rVert \\
    &\leq \frac{\alpha_1}{12}\lVert\nabla\xi\rVert^2 + c\lVert\xi\rVert^2, \\
    |I_4| &\leq \lVert\nabla J\rVert_{L^1}\|m\|_{L^\infty}\lVert\nabla\xi\rVert\lVert\xi\rVert \leq \frac{\alpha_1}{12}\lVert\nabla\xi\rVert^2 +  c\lVert\xi\rVert^2,\\
    \text{ Using boundary condition on $\xi$ we can write }\\
    I_5 &= ((m(\Bar{\varphi})a+ \lambda(\Bar{\varphi}))\nabla\xi,\nabla\xi)+ (\nabla(m(\Bar{\varphi})a+ \lambda(\Bar{\varphi}))\nabla\xi,\xi),\\
    |((m(\Bar{\varphi})a+ \lambda(\Bar{\varphi}))\nabla\xi,\nabla\xi)| &\geq \alpha_1\lVert\nabla\xi\rVert^2,\\
    |(\nabla(m(\Bar{\varphi})a+ \lambda(\Bar{\varphi}))\nabla\xi,\xi)| &\leq c\lVert\nabla\Bar{\varphi}\rVert_{L^4}\lVert\xi\rVert_{L^4}\lVert\nabla\xi\rVert + c\lVert\xi\rVert\lVert\nabla\xi\rVert \\
    &\leq \frac{\alpha_1}{6}\lVert\nabla\xi\rVert^2  +\lVert\nabla\Bar{\varphi}\rVert\lVert\Bar{\varphi}\rVert_{H^2}\lVert\xi\rVert\lVert\nabla\xi\rVert + c\lVert\xi\rVert^2\\
    &\leq \frac{\alpha_1}{4}\lVert\nabla\xi\rVert^2 + c(1+\lVert\nabla\Bar{\varphi}\rVert^2\lVert\Bar{\varphi}\rVert_{H^2}^2)\lVert\xi\rVert^2,\\
    |I_6| &\leq (\lVert\nabla a\rVert_{L^\infty} + \lVert\nabla J\rVert_{L^1})\lVert\Bar{\varphi}\rVert_{L^\infty}\lVert\textbf{v}\rVert\lVert\xi\rVert \leq \frac{\eta}{6}\lVert\textbf{v}\rVert^2 + c \lVert\xi\rVert ^2,\\
    |I_7| &\leq \lVert\nabla J\rVert_{L^1}\lVert\Bar{\varphi}\rVert_{L^\infty}\lVert\textbf{v}\rVert\lVert\xi\rVert \leq \frac{\eta}{6}\lVert\textbf{v}\rVert^2+ c\lVert\Bar{\varphi}\rVert_{L^\infty}^2\lVert\xi\rVert^2,\\
    |(\Bar{\varphi}-\varphi_d,\xi)| &\leq \lVert\Bar{\varphi}-\varphi_d\rVert\lVert\xi\rVert \leq c(\lVert\Bar{\varphi}-\varphi_d\rVert^2 + \lVert\xi\rVert^2).
\end{align*}
Combining all the above estimates we get, 
\begin{align}\label{eq4.10}
     -\frac{1}{2}\frac{d}{dt}\lVert\xi\rVert^2 + \alpha_1\lVert\nabla\xi\rVert^2 &\leq \frac{\alpha_1}{2}\lVert\nabla\xi\rVert^2 +\frac{\eta}{3}\lVert\textbf{v}\rVert^2 + c\lVert\Bar{\varphi}-\varphi_d\rVert^2 + C(1+ \lVert\Bar{\textbf{u}}\rVert_{L^\infty}^2  +\lVert\nabla\Bar{\varphi}\rVert^2\lVert\Bar{\varphi}\rVert_{H^2}^2)\lVert\xi\rVert^2.
\end{align}
    Similarly testing \eqref{eq4.2} with $\textbf{v} \in \mathbb{V}_{div} $, we get, 
\begin{align*}
    \nu\lVert\nabla\textbf{v}\rVert^2  +\eta\lVert\textbf{v}\rVert^2 + (\xi\nabla\Bar{\varphi},\textbf{v}) &=(\Bar{\textbf{u}}-\textbf{u}_d, \textbf{v}).
\end{align*}
\begin{align}\label{eq4.11}
    \nu\lVert\nabla\textbf{v}\rVert^2  +\eta\lVert\textbf{v}\rVert^2 &\leq  \lVert\xi\rVert\lVert\nabla\Bar{\varphi}\rVert_{L^4}\lVert\textbf{v}\rVert_{L^4} + \lVert\Bar{\textbf{u}}-\textbf{u}_d\rVert \lVert\textbf{v}\rVert \nonumber\\
    &\leq \frac{\nu}{2}\lVert\nabla\textbf{v}\rVert^2 + c\lVert\nabla\Bar{\varphi}\rVert\lVert\Bar{\varphi}\rVert_{H^2}\lVert\xi\rVert^2 + \frac{\eta}{6} \lVert\textbf{v}\rVert^2+ c\lVert\Bar{\textbf{u}}-\textbf{u}_d\rVert^2.
\end{align}
Adding \eqref{eq4.10} and \eqref{eq4.11} followed by an application of Korn's inequality gives the following inequality,
 \begin{align}\label{eq4.13}
     -\frac{1}{2}\frac{d}{dt}\lVert\xi\rVert^2 + \frac{\alpha_1}{2}\lVert\nabla\xi\rVert^2 +c_1\lVert\textbf{v}\rVert_{\mathbb{V}_{div}}^2 &\leq c(\lVert\Bar{\varphi}-\varphi_d\rVert^2+ \lVert\Bar{\textbf{u}}-\textbf{u}_d\rVert^2) + C(1+ \lVert\Bar{\textbf{u}}\rVert_{L^\infty}^2 + \lVert\Bar{\varphi}\rVert_{H^2}^2  +\lVert\nabla\Bar{\varphi}\rVert^2\lVert\Bar{\varphi}\rVert_{H^2}^2)\lVert\xi\rVert^2,
\end{align}
where $c_1>0$. Now integrating \eqref{eq4.13} from $t$ to $T,$
\begin{align}\label{eq4.14}
     \frac{1}{2}\lVert\xi(t)\rVert^2 + \frac{\alpha_1}{2}\int\limits_{t}^{T}\lVert\nabla\xi\rVert^2 +c_1\int\limits_{t}^{T}\lVert\textbf{v}\rVert_{\mathbb{V}_{div}}^2 \leq \frac{1}{2}\lVert\Bar{\varphi}(T)-\varphi_\Omega(T)\rVert^2 + c\int\limits_{t}^{T}\lVert\Bar{\varphi}-\varphi_d\rVert^2+ c\int\limits_{t}^{T}\lVert\Bar{\textbf{u}}-\textbf{u}_d\rVert^2 +C\Lambda(t)\int\limits_{t}^{T}\lVert\xi\rVert^2,
\end{align} 
Where $\Lambda(t) = \underset{\bm{\tau}\in [t,T]}{\text{Sup}} \big(1 + \lVert\Bar{\textbf{u}}(\tau)\rVert_{L^\infty}^2 + \lVert\Bar{\varphi}(\bm{\tau})\rVert_{H^2}^2  +\lVert\nabla\Bar{\varphi}(\bm{\tau})\rVert^2\lVert\Bar{\varphi}(\bm{\tau})\rVert_{H^2}^2\big)$. Applying Gronwall's inequality, we get,
\begin{align}\label{eq4.140}
     \lVert\xi(t)\rVert^2 + \lVert\nabla\xi\rVert_{L^2(0,T;H)}^2 +\lVert\textbf{v}\rVert_{L^2(0,T;{\mathbb{V}_{div}})}^2 &\leq C(1+ \Lambda(t))\big(\lVert\Bar{\varphi}(T)-\varphi_\Omega(T)\rVert^2 +\lVert\Bar{\varphi}-\varphi_d\rVert_{L^2(0,T;H)}^2\nonumber\\ &\hspace{.5cm}+ \lVert\Bar{\textbf{u}}-\textbf{u}_d\rVert_{L^2(0,T;H)}^2 \big).
\end{align}
Therefore we have,
\begin{align*}
    \xi &\in L^\infty(0,T;H) \cap L^2(0,T;V) ,\\
    \textbf{v} &\in L^2(0,T;\mathbb{V}_{div}).
\end{align*}
Further, taking the inner product of \eqref{eq4.1} with $\zeta \in V $ and estimating each of its terms, we obtain,
\begin{align*}
   |(\xi',\zeta)| &\leq c(\lVert\nabla\xi\rVert\lVert\zeta\rVert + \lVert\nabla\xi\rVert\lVert\nabla\zeta\rVert + \lVert\textbf{v}\rVert\lVert\zeta\rVert + \lVert\Bar{\varphi}-\varphi_d\rVert\lVert\zeta\rVert) \\ 
    &\leq C\big(\lVert\xi\rVert_{V}+ \lVert\textbf{v}\rVert + \lVert\Bar{\varphi}-\varphi_d\rVert\big) \lVert\zeta\rVert_{V}.
\end{align*}
implies $ \lVert\xi'\rVert_{V'} \leq C\big(\lVert\xi\rVert_{V}+ \lVert\textbf{v}\rVert + \lVert\Bar{\varphi}-\varphi_d\rVert\big)$. Integrating over $[0,T]$ and using \eqref{eq4.140} we conclude that, $\xi'\in L^2(0,T;V').$ Then by Aubin-Lions lemma, $\xi \in C([0,T];H)\cap L^2(0,T;V)$. 
To prove the uniqueness of a weak solution we use similar arguments as in the proof of uniqueness of a weak solution of the linearized system. 
\end{proof}

\subsection{Characterisation of an optimal control}

\begin{theorem}[First order necessary optimality condition]
Let $\Bar{\textbf{U}} \in \mathcal{U}_{ad}$ be an optimal control for (OCP) with associated optimal state given by  $(\Bar{\varphi},\Bar{\textbf{u}})$ and adjoint state $(\xi,\textbf{v})$. Then the following variational inequality holds. 
 \begin{align}\label{eq5.13}
     \int\limits_{0}^{T}\int_\Omega (\Bar{\textbf{U}}+\textbf{v})(\textbf{U}-\Bar{\textbf{U}}) \geq 0 \hspace{.5cm} \forall \textbf{U}\in \mathcal{U}_{ad}.
\end{align}
\end{theorem}

\begin{proof}

 As in Theorem 5.1, let $(\psi, \textbf{w}) $ be the unique weak solution of the linearized system \eqref{eq3.2}-\eqref{eq3.7} with control $ \textbf{U} - \bar{\textbf{U}} $. Taking inner product of the adjoint system \eqref{eq4.1}, \eqref{eq4.2} with $\psi $ and $\textbf{w} $ respectively and adding the resulting equations  we get, 
\begin{align*}
 \int\limits_{0}^{T}\int_{\Omega}\psi(\Bar{\varphi}-\varphi_d) + \int\limits_{0}^{T}\int_{\Omega}\textbf{w}(\Bar{\textbf{u}}-\textbf{u}_d) &= \int\limits_{0}^{T}\int_{\Omega}\psi (-\xi' - \Bar{u}\nabla\xi +m'(\Bar{\varphi})((\nabla a)\Bar{\varphi}-\nabla J*\Bar{\varphi})\nabla\xi  - \nabla J*(m(\Bar{\varphi})\nabla\xi))\\ &\hspace{.5cm} -\int\limits_{0}^{T}\int_{\Omega}\psi(m(\Bar{\varphi})a+ \lambda(\Bar{\varphi})\Delta\xi 
     + \int\limits_{0}^{T}\int_{\Omega}\psi\big((\nabla a)\Bar{\varphi} - \nabla J*\Bar{\varphi}\textbf{v}- \nabla J*(\Bar{\varphi}\textbf{v})\big)\\ &\hspace{.5cm}+ \int\limits_{0}^{T}\int_{\Omega}\textbf{w}( -\nabla\cdot(\nu\nabla\textbf{v})  +\eta\textbf{v}+\xi\nabla\Bar{\varphi})
\end{align*}
\begin{align*}
     &= \int\limits_{0}^{T}\int_{\Omega}\Big(\psi'+ \Bar{u}\nabla\psi-\nabla\cdot\big(m'(\Bar{\varphi})(\nabla(a\Bar{\varphi})-\nabla J*\Bar{\varphi})\psi + m(\Bar{\varphi})(\nabla(a\psi)-\nabla J*\psi) \big)\Big)\xi \\& \hspace{.5cm} -\int\limits_{0}^{T}\int_{\Omega}\nabla\cdot\big( \lambda'(\Bar{\varphi})\nabla\Bar{\varphi}\psi + \lambda(\Bar{\varphi})\nabla\psi\big)\xi + \int\limits_{0}^{T}\int_{\Omega}\big((\Bar{\varphi}\nabla a-\nabla J*\Bar{\varphi})\psi-(\nabla J*\psi)\Bar{\varphi}\big)\textbf{v} \\ &\hspace{.5cm} -\int_{\Omega}\psi(T)(\Bar{\varphi}(T)-\varphi_\Omega)  + \int\limits_{0}^{T}\int_{\Omega}(-\nu\Delta\textbf{w}+ \eta\textbf{w})\textbf{v}+(\textbf{w}\nabla\Bar{\varphi})\xi. 
\end{align*}
Comparing with the linearized system
\eqref{eq3.2}-\eqref{eq3.7}, we obtain,
\begin{align*}
    \int\limits_{0}^{T}\int_{\Omega}\psi(\Bar{\varphi}-\varphi_d) + \int\limits_{0}^{T}\int_{\Omega}\textbf{w}(\Bar{\textbf{u}}-\textbf{u}_d) + \int_{\Omega}\psi(T)(\Bar{\varphi}(T)-\varphi_\Omega) &= \int\limits_{0}^{T}\int_{\Omega}(\textbf{U}-\Bar{\textbf{U}})\textbf{v}.
\end{align*}
Substituting back in \eqref{eq5.01} we get, 
\begin{equation*}
 \int\limits_{0}^{T}\int_{\Omega}(\textbf{U}-\Bar{\textbf{U}})\ \textbf{v} + \int\limits_{0}^{T}\int_\Omega (\textbf{U}-\Bar{\textbf{U}}) \ \Bar{\textbf{U}} \geq 0 \hspace{.5cm}\forall \; \textbf{U}\in \mathcal{U}_{ad} .  
\end{equation*}
Hence the variational inequality \eqref{eq5.13} holds true.
\end{proof}

\begin{remark}
The inequality \eqref{eq5.13} together with the adjoint system \eqref{eq4.1}-\eqref{eq4.3} form the first order necessary optimality condition. Since the set of admissible controls, $\mathcal{U}_{ad}$ is a non-empty closed convex subset of $L^2(0,T;\mathbb{G}_{div}),$ \eqref{eq5.13} gives an equivalent condition,
\begin{equation*}
    \Bar{\textbf{U}} = \mathbb{P}_{\mathcal{U}_{ad}}(-\textbf{v}),
\end{equation*}
where $\mathbb{P}_{\mathcal{U}_{ad}}$ represents the orthogonal projection from $L^2(0,T;\mathbb{G}_{div})$ onto the space $\mathcal{U}_{ad}.$ Further, we have pointwise condition,
\begin{equation}
    \Bar{\textbf{U}}_i(x,t) = \text{max}\{ \textbf{U}_{1,i}(x,t), \text{min}\{-\textbf{v}_i(x,t), \textbf{U}_{2,i}(x,t)\}\}, \hspace{.5cm}i=1,2
\end{equation}
for a.e. $(x,t) \in \Omega\times(0,T)$.
\end{remark}
\vspace{.75cm}



\addcontentsline{toc}{chapter}{Bibliography}


\begin{thebibliography}{6}
\bibitem{AUE} Aubin, J. P., and  Ekeland, I. (2006). Applied nonlinear analysis. Courier Corporation.
\bibitem{BCG} Bosia, S., Conti, M., and  Grasselli, M. (2015). On the Cahn–Hilliard–Brinkman system. Communications in Mathematical Sciences, 13(6), 1541-1567.
\bibitem{CSM} Cavaterra, C., Frigeri, S., and  Grasselli, M. (2022). Nonlocal Cahn–Hilliard–Hele–Shaw systems with singular potential and degenerate mobility. Journal of Mathematical Fluid Mechanics, 24(1), 13.
\bibitem{CDT} Coddington, E. A. (1955). Theory of ordinary differential equations. 
\bibitem{CFG} Colli, P., Frigeri, S., and Grasselli, M. (2012). Global existence of weak solutions to a nonlocal Cahn–Hilliard–Navier–Stokes system. Journal of Mathematical Analysis and Applications, 386(1), 428-444.
\bibitem{MCG} Conti, M., and Giorgini, A. (2018). The three-dimensional Cahn-Hilliard-Brinkman system with unmatched viscosities.
\bibitem{DDU} Dai, S., and Du, Q. (2016). Weak solutions for the Cahn–Hilliard equation with degenerate mobility. Archive for Rational Mechanics and Analysis, 219, 1161-1184.
\bibitem{DAG} Della Porta, F., Giorgini, A., and Grasselli, M. (2018). The nonlocal Cahn–Hilliard–Hele–Shaw system with logarithmic potential. Nonlinearity, 31(10), 4851.
\bibitem{DPG} Della Porta, F., and Grasselli, M. (2015). On the nonlocal Cahn-Hilliard-Brinkman and Cahn-Hilliard-Hele-Shaw systems. Communications on Pure and Applied Analysis, 15(2), 299-317.
\bibitem{SDM} Dharmatti, S., and Perisetti, L. N. M. (2021). Nonlocal Cahn-Hilliard-Brinkman System with regular potential: Regularity and optimal control. Journal of Dynamical and Control Systems, 27, 221-246.
\bibitem{MEL} Ebenbeck, M., and Lam, K. F. (2020). Weak and stationary solutions to a Cahn–Hilliard–Brinkman model with singular potentials and source terms. Advances in Nonlinear Analysis, 10(1), 24-65.
\bibitem{CEG} Elliott, C. M., and Garcke, H. (1996). On the Cahn–Hilliard equation with degenerate mobility. Siam journal on mathematical analysis, 27(2), 404-423.
\bibitem{LCE} Evans, L. C. (2022). Partial differential equations (Vol. 19). American Mathematical Society.
\bibitem{SEF} Frigeri, S. (2021). On a nonlocal Cahn-Hilliard/Navier-Stokes system with degenerate mobility and singular potential for incompressible fluids with different densities. Annales de l'Institut Henri Poincaré C, 38(3), 647-687.
\bibitem{FGG} Frigeri, S., Gal, C. G., and Grasselli, M. (2016). On nonlocal Cahn–Hilliard–Navier–Stokes systems in two dimensions. Journal of Nonlinear Science, 26, 847-893.
\bibitem{SFG} Frigeri, S., Gal, C. G., Grasselli, M., and Sprekels, J. (2019). Two-dimensional nonlocal Cahn–Hilliard–Navier–Stokes systems with variable viscosity, degenerate mobility and singular potential. Nonlinearity, 32(2), 678.
\bibitem{FGR}Frigeri, S., Grasselli, M., and Rocca, E. (2015). A diffuse interface model for two-phase incompressible flows with non-local interactions and non-constant mobility. Nonlinearity, 28(5), 1257.
\bibitem{FGS}Frigeri, S., Grasselli, M., and Sprekels, J. (2020). Optimal distributed control of two-dimensional nonlocal Cahn–Hilliard–Navier–Stokes systems with degenerate mobility and singular potential. Applied Mathematics and Optimization, 81(3), 899-931.
\bibitem{GGG} Gal, C. G., Giorgini, A., and Grasselli, M. (2017). The nonlocal Cahn–Hilliard equation with singular potential: well-posedness, regularity and strict separation property. Journal of Differential Equations, 263(9), 5253-5297.
\bibitem{GAG} Giorgini, A., Grasselli, M., and Wu, H. (2018). The Cahn–Hilliard–Hele–Shaw system with singular potential. Annales de l'Institut Henri Poincaré C, 35(4), 1079-1118.
\bibitem{SKN} Kesavan, S. (2015). Topics in functional analysis and applications [2pr. Ed.], New Age. 
\bibitem{TEM} Temam, R. (2024). Navier–Stokes equations: theory and numerical analysis (Vol. 343). American Mathematical Society.
\bibitem{FTO} Tröltzsch, F. (2010). Optimal control of partial differential equations: theory, methods, and applications (Vol. 112). American Mathematical Soc.

\end{thebibliography}

\end{document}